\documentclass{article}
\usepackage[left=3cm,right=3cm, top=3cm]{geometry}
\usepackage{amssymb,amsthm,amsmath,amsxtra,dsfont,appendix}
\usepackage[nottoc]{tocbibind}   
\usepackage{hyperref} \hypersetup{colorlinks=true, citecolor=blue, linkcolor=black}
\usepackage{mathrsfs}
\usepackage{cancel}

\usepackage[colorinlistoftodos]{todonotes}
\usepackage{natbib}
\usepackage{float}

\newcommand\p{
\mathchoice
{{\scriptstyle\mathscr{P}}}
{{\scriptstyle\mathscr{P}}}
{{\scriptscriptstyle\mathscr{P}}}
{\scalebox{.7}{$\scriptscriptstyle\mathscr{P}$}}
}

\newcommand{\blue}{} 

\newcommand{\cst}[1]{{\blue c_{#1}}}
\newcommand{\wcst}[1]{{\blue \widehat{c_{#1}}}}
\newcommand{\eps}[1]{{\blue \epsilon_{#1}}}
\newcommand{\ups}[1]{{\blue \Upsilon_{#1}}}
\newcommand{\Tf}[1]{{\blue \Theta_{N,m}^{#1}}}

\newcommand{\sgn}{\operatorname{sgn}}

\newcommand{\tr}{\operatorname{Tr}}
\renewcommand{\P}[1]{\mathds{P}\left[#1\right]}
\newcommand{\E}[1]{\mathds{E}\left[#1\right]}

\renewcommand{\i}{\mathrm{i}}

\newcommand{\1}{\mathbf{1}}
\newcommand{\A}{\mathcal{A}}
\renewcommand{\d}{\mathrm{d}}
\newcommand{\C}{\mathbb{C}}
\newcommand{\Co}{\mathcal{C}}
\renewcommand{\E}{\mathbb{E}}

\newcommand{\g}{\mathrm{g}}
\newcommand{\G}{\mathbf{G}}
\newcommand{\N}{\mathbb{N}}
\renewcommand{\O}{\mathcal{O}}
\renewcommand{\P}{\mathbb{P}}
\newcommand{\R}{\mathds{R}}

\newcommand{\T}{\mathds{T}}
\newcommand{\U}{\mathscr{U}}
\renewcommand{\u}{\mathbf{U}}
\newcommand{\x}{\boldsymbol{\xi}}
\newcommand{\X}{\mathbf{X}}
\newcommand{\Z}{\mathbb{Z}}

\newtheorem{theorem}{Theorem}[section]

\newtheorem{proposition}[theorem]{Proposition}

\newtheorem{lemma}[theorem]{Lemma}

\theoremstyle{definition} \newtheorem{remark}{Remark}[section]

\numberwithin{equation}{section}

\title{
Multivariate normal approximation for traces of random unitary matrices}
\date{February 6, 2020}
\author{Kurt Johansson \footnote{KTH Royal Institute of Technology, Matematiska institutionen,  100 44 Stockholm, Sweden. 
\newline K.J. was supported by the grant KAW 2015.0270 from the Knut and Alice Wallenberg Foundation and the Swedish Research Council grant 2015-04872.
\newline Email: \href{mailto:kurtj@kth.se}{\nolinkurl{mailto:kurtj@kth.se}}}
\qquad
Gaultier Lambert \footnote{
University of Zurich, Winterthurerstrasse 190, 8057 Z\"urich, Switzerland. 
\newline G.L. research is supported  by the SNSF Ambizione grant S-71114-05-01.
\newline Email: \href{mailto:gaultier.lambert@math.uzh.ch}{\nolinkurl{gaultier.lambert@math.uzh.ch}}}}

\begin{document}

\maketitle

\begin{abstract}
In this article, we obtain a super-exponential rate of convergence in total variation between the traces of the first $m$  powers of an $n\times n$ random unitary matrices and  a $2m$-dimensional Gaussian random variable.
This generalizes previous results in the scalar case to the multivariate setting, and we also give the precise 
dependence on the dimensions $m$ and $n$ in the estimate with explicit constants. 
We are especially interested in the regime where  $m$ grows with $n$ and our main result basically states that if $m\ll \sqrt{n}$, then the rate of convergence in the Gaussian approximation is $\Gamma(\frac nm+1)^{-1}$ times a correction. 
We also show that the Gaussian approximation remains valid for all $m\ll n^{2/3}$ without a fast rate of convergence. 
\end{abstract}


\section{Introduction and main results}

\subsection{Introduction}

Let $\u$ be a random unitary matrix distributed according to the normalized Haar measure $\P_n$ on the unitary group $U(n)$ of size $n\in\N$. In random matrix theory this is known as the circular unitary ensemble or CUE. The joint law of the eigenvalues $(e^{\i\theta_1}, \dots, e^{\i\theta_n})$ of $\u$, $\i = \sqrt{-1}$, $\theta_j\in [-\pi,\pi]$, under this 
probability measure has an explicit density given by the Weyl integration formula,
\begin{equation} \label{pdf}
\frac 1{(2\pi)^nn!} \prod_{1\le k<j \le n}   \hspace{-.13cm}\left| e^{\i\theta_k}-e^{\i\theta_j}\right|^2=
\frac{\mho_n}{(2\pi)^n}  \hspace{-.13cm} \prod_{1\le k<j \le n}   \hspace{-.13cm}\sin^2\left(\frac{\theta_k-\theta_j}{2}\right),
\end{equation}
where $\mho_n =  2^{n(n-1)}/n!$. Consider the random variable
\begin{equation*}
Z=\sum_{k=1}^m\xi_{2k-1}\sqrt{\frac 2k}\Re\tr\u^k+\xi_{2k}\sqrt{\frac 2k}\Im\tr\u^k,
\end{equation*}
where $\xi=(\xi_1,\dots,\xi_{2m})\in\R^{2m}$, built from the traces of the unitary matrix $\u$.
It is a well--known consequence of the Strong Szeg\H{o} theorem (\citep{Szego52} -- see Theorem~\ref{thm:Szego} below) that for any fixed $m\in\N$, $Z\to\|\xi\|\mathcal{N}$ weakly
as $n\to\infty$, where $\mathcal{N}$ is a standard Gaussian random variable. This is a surprising result since the trace is the sum of $n$ random variables and there is no normalization in~$n$. 
This limit theorem is also a consequence of the striking fact proved by \citep{DS94} that all joint moments of $\sqrt{\frac 2k}\Re\tr\u^k$ and $\sqrt{\frac 2k}\Im\tr\u^k$ up to a certain order are identical to those of independent standard Gaussian random variables (see Theorem~\ref{thm:DS} below). Based on this result, Persi Diaconis \citep{Dia} conjectured that the rate of convergence in total variation norm of $Z$ to a normal random variable should be very fast, even super-exponential. This was proved in \citep{Johansson97}, where it was shown that there are positive constants $C$ and $\delta$ so that
\begin{equation} \label{JTV}
{\rm d_{TV}}\big(Z, ||\xi||\mathcal{N}\big)  \le Cn^{-\delta n}, 
\end{equation}
where ${\rm d_{TV}}$ denotes the total variation distance (see \eqref{TV} below for a definition). No explicit expression for $C$ or $\delta$ or their dependence on $m$ and the parameters was given.

\medskip

A related but separate problem is to consider the multivariate convergence of the random variables
\begin{equation} \label{def:X}
\mathrm{X}_{2k-1} := \sqrt{\frac{2}{k}} \Re \tr \u^k
\qquad\text{and}\qquad 
\mathrm{X}_{2k} := \sqrt{\frac{2}{k}} \Im \tr \u^k,
\end{equation}
$1\le k\le m$. We are interested in the law of the random vector  $\X =(\mathrm{X}_1, \dots, \mathrm{X}_{2m})$ when the dimension of the matrix $\u$ is large. 
Let  $\G =(\mathrm{G}_1, \dots, \mathrm{G}_{2m})$ be i.i.d. standard Gaussian random variables. For a fixed $m \in\N$ it again follows from the Strong Szeg\H{o} theorem that
$\X\to \G$ weakly as $n\to\infty$. Peter Sarnak \citep{Sar} raised the following problem in connection with his work with M. Rubinstein on computing zeros of L--functions and under--determined matrix moment problems.
How close is $\X$ to $\G$ in total variation distance as a function of $m$ for a given $n$?
Here, $m$ can depend on $n$, e.g. be a power of $n$. Is $\X$ still very close to $\G$? This is the main problem investigated in the present paper. Theorem \ref{thm:TV}
and Theorem \ref{thm:main} give our results. We get a statement for $m$ almost up to $\sqrt{n}$. The other classical groups can also be considered, see \citep{CouJoh}. Since
we are mainly interested in the case when $m$ is large we assume that $m\ge 3$ throughout this paper. In the case $m=1$ it is possible to get a more precise result and this together with results on single traces
will be considered for all the classical compact groups in a forthcoming publication, \citep{CouJohLam}. (A bound for $m=2$ can be directly inferred from the case $m=3$; a special
treatment of this case would only give a slight improvement.) An important aspect of the present work is that, in contrast to \citep{Johansson97}, we keep explicit track of the constants and the dependence on $m$. We have also made an effort to optimize in the argument and get reasonable numerical constants.

\medskip

Since $Z=\X \cdot \xi$, as a consequence of our multivariate results we can improve (\ref{JTV}), for a fixed $m$ and uniformly for all $\xi$, to
\begin{equation} \label{TVuniform}
{\rm d_{TV}}\big(Z, \|\xi\| \mathcal{N}\big) =  \O\bigg( \frac{ e^{\frac{n}{m}( \log(1+\log m)+ \frac12)}}{\sqrt{n}\ \Gamma(\frac nm+1)}\bigg) ,
\end{equation}
where the implied constant has an explicit dependence in $m\in\N$. Broadly speaking, we expect that the best possible estimate for the RHS of \eqref{TVuniform} is $\Gamma(\frac nm+1)^{-1}$ times some sub--exponential corrections. 
We can also let the  degree $m$ grow as $n\to+\infty$. From Proposition~\ref{cor:TVuniform}  we deduce the following estimate:
\[
\sup_{m\le  \frac{\sqrt{n}}{6.45 (\log n)^{1/4}}  }{\rm d_{TV}}\big(Z, \|\xi\| \mathcal{N}\big) 
\le \sqrt{n}\exp \big(19.4 - 0.83\sqrt{n}(\log n)^{5/4} \big) . 
\]
uniformly for all $\xi$ when $n$ is large enough.

Using Stein's method and the exact moment identities  from \citep{DS94}, one can infer the following rate of convergence in the multivariate problem: for any $m \le 2n$, 
\begin{equation} \label{DS}
\mathrm{W}_1(\X,\G) =  \O(m^2/n) , 
\end{equation}
where $\mathrm{W}_1$ denotes the Wasserstein 1 distance between two probability measures on $\R^{2m}$ -- see \cite[Theorem 3.1]{DS11}. 
By relying on the recent techniques from \citep{LLW19},  we can improve on \eqref{DS} -- see Theorem~\ref{thm:Stein} below. See also \citep{Webb} for an analogous multivariate result that applies to more general circular $\beta$--ensembles and Remark~\ref{rk:beta} below.
Recently, rates of convergence to the Gaussian law have also been obtained for $\tr f(\mathbf{M})$ where $f:\R\to\R$ is a real--analytic function and $\mathbf{M}$ is a random matrix from the Gaussian, Laguerre or Jacobi unitary ensembles by \cite{BB19}  using  Riemann--Hilbert techniques. 
In contrast to the CUE, in these cases, the optimal rates of convergence are expected to be polynomial in the dimension of the random matrix.

The fast rate of convergence of $\X$ to $\G$ holds for $m\ll\sqrt{n}$ by Theorem \ref{thm:TV}, but we see from Theorem~\ref{thm:Stein} below that we have
convergence to the multivariate Gaussian for $m\ll n^{2/3}$. We have no conjecture concerning the threshold $m\in\N$ at which the Gaussian approximation fails. Also, we do
not know whether there is some transition when varying $m$ where we go from a fast convergence rate to some other rate of convergence.

\subsection{Main results} \label{sect:main}

For any $m\in\N$, we denote by $\Omega_{m} = \frac{\pi^m}{m!}$  the volume of the unit ball and by $\|x\| = \sqrt{x_1^2+\cdots +x_{2m}^2}$ the Euclidean norm in  $\R^{2m}$.
It is straightforward to see that for any $m,n\in\N$, the random vector $\X$ has a density on $\R^{2m}$ that we denote by $\p_{n,m}$. 
For any $n,m\in \N$ and $k\in\N$, we define
\begin{equation} \label{Delta}
\Delta_{n,m}^{(k)} : = \bigg( \int_{\R^{2m}} \bigg| \p_{n,m}(x) - \frac{e^{-\|x\|^2/2}}{(2\pi)^m} \bigg|^k \d x \bigg)^{1/k}.
\end{equation}
In this paper, we focus on getting (non--asymptotic) estimates for $\Delta_{n,m}^{(1)}$ and $\Delta_{n,m}^{(2)}$ with explicit constants which hold for large $n\in\N$  when $m \ll \sqrt{n}$. 
Let us observe that $\Delta_{n,m}^{(1)}$ controls the \emph{total variation} distance between $\X$ and  $\G$ (a standard Gaussian random variable on $\R^{2m}$). 
Namely, we have 
\begin{equation} \label{TV}
{\rm d_{TV}}(\X, \G) := \sup_{A \subset \R^{2m}} \big| \P_n[\X\in A] - \P[\G \in A] \big| \le \Delta_{n,m}^{(1)} ,
\end{equation}
where the supremum is taken over all Borel subsets $A \subset \R^{2m}$. 
Our main result, which is a quantitative generalization of the estimates \eqref{JTV} from \citep{Johansson97} in a multi--dimensional setting can be summarized as follows. 

\begin{theorem} \label{thm:TV}
For all $n,m\in \N$ such that $n\ge  1911$ and $N = n/m \ge  146.5 m \sqrt{1+\log m}$, 
we have the following estimate in total variation distance, 
\[
{\rm d_{TV}}(\X, \G)
\le 16 \sqrt{\Omega_m}  m^{\frac 52}  4^me^{\frac N2+\frac{m^2}{4N}} \frac{ \big(N\sqrt{\log N} \big)^{m}(1+\log m)^{N}  }{\sqrt{N}\ \Gamma(N+1)} .
\]
\end{theorem}

We expect that, up to corrections, the  factor $\Gamma(\frac nm+1)^{-1}$  is actually the correct order for the statistical distance between the random vectors $\X$ and $\G$ as long as $m \ll \sqrt{n}$. 
To clarify the meaning of this estimate in the regime where $m$ grows with $n$, let us also give the following consequence when $m$ is like $n^\alpha$, $\alpha<1/2$.

\begin{proposition} \label{cor:TV}
Let $m = \lfloor n^\alpha \rfloor$ with $0<\alpha<1/2$, then for all $n\ge n_\alpha$, 
\[
{\rm d_{TV}}(\X, \G)
\le \tfrac{18 e^{8\pi}}{(2\pi)^{\frac 34}} n^{3\alpha-\frac32} \exp\big( - (1-\epsilon_n)  n^{1-\alpha} \log(n^{1-\alpha}) \big)  , 
\]
where $n_\alpha := \inf\big\{ n\ge 18^{1/\alpha}  : n^{1-2\alpha} \ge   20.4 \sqrt{\log n} \big\}$, $1-\epsilon_n \ge 87 \cdot 10^{-3}$
and $\epsilon_n\to 0$ as $n\to\infty$; see \eqref{TV:error} for a more precise bound on $\epsilon_n$.
\end{proposition}

The proof of Theorem~\ref{thm:TV} and Proposition~\ref{cor:TV} are given in Section~\ref{sect:proofTV}. According to \eqref{TV}, these results are consequences of the following more precise bounds. 
We postpone the definition of $\Tf{}$ to the Appendix~\ref{sect:approx} since it is rather involved.

\begin{theorem} \label{thm:main}
Let $\Tf{}$ be given by \eqref{Theta}. 
For any $n,m\in \N$ such that $m\ge 3$ and $ N = n/m > 4m$, we have 
\begin{equation} \label{main}
\Delta_{n,m}^{(2)}  \le  8  \sqrt{\Omega_m} N^{\frac m2}  \Theta_{N,m}   .
\end{equation}
If we assume that $  \Delta_{n,m}^{(2)}  \le 5 \cdot 2^{-m}m^{1-\frac m2} e^{- \frac m2}$,  then
\begin{equation} \label{Delta1}
\Delta_{n,m}^{(1)}  \le 2  \big(8 \log  \Delta_{n,m}^{(2)-1} \big)^{\frac m2}\Delta_{n,m}^{(2)}  . 
\end{equation}
\end{theorem}

The proof of Theorem~\ref{thm:main} is explained in Section~\ref{sect:proofs} and it is given in Section~\ref{sect:mainproof}. This shows that the parameter $\Tf{}$ controls the statistical distance between the random vectors $\X$ and $\G$.
We have made significant efforts to keep track carefully of the dependency in $n,m$ of our estimate with reasonable numerical constants. Unfortunately, this leads to an expression for  $\Tf{}$ which is rather involved --  see Section~\ref{sect:constants}.
In particular, there are several regimes depending on  $n$ and $m$ where different contributions are relevant. Let us just point out that in the cases we are most interested in, that is when $m$ is large and $N= \frac nm$ is sufficiently large compared to $m$, we obtain the following bounds which allow us to verify the second assumption in the formulation of  Theorem~\ref{thm:main}. 

\begin{proposition} \label{prop:Thetabound}
Fix an integer $M\ge 3$.
For all $m\ge M$ and $N = n/m \ge  \cst{}(M) m \sqrt{1+\log m}$,
\[
\Tf{} \le (1+\eps{} ) \, m^{\frac 52}  2^{\frac{m}{2}}e^{\frac{m^2}{4N}} \frac{ e^{\frac{N}{2}} (1+\log m)^{N}  }{\sqrt{N}\ \Gamma(N+1)}  , 
\]
with $\eps{} \le 25\cdot 10^{-5}$ and $\cst{}(M)$ are explicit constants given in the table \eqref{table:cM}. We emphasize that $\cst{}(M)$ is non-increasing in $M\in\N$ with $\cst{}(3) = 146.5$ as in Theorem~\ref{thm:TV} and  $\cst{}(M) = 19.4$ for $M\ge 70$. 
\end{proposition}

The proof of Proposition~\ref{prop:Thetabound} involves rather technical numerical estimates (which have been obtained with \emph{Mathematica}) and it is given in the Appendix~\ref{sect:error}.

\medskip

Our next result shows that it is still possible to approximate $\X$ by a Gaussian random vector when $m \gg \sqrt{n}$. 
It is an interesting question whether the approximation also holds for the total variation distance. 
Recall that the Kantorovich or Wasserstein distances between the random vectors $\X$ and the Gaussian $\G$ are defined by for any $q\ge 1$,  
\begin{equation} \label{def:W}
\mathrm{W}_q(\X,\G) = \inf_\P  \Big( \E\big[\| \mathrm{x}-\g\|^q  \big]\Big)^{1/q} ,
\end{equation}
where the infimum is taken over all probability measures  on $\R^{2m} \times \R^{2m}$ such that the first marginal of $\P$, $\mathrm{x}$ has the same law as $\X$ and the second marginal of $\P$, $\g$ is a standard Gaussian on $
\R^{2m}$.

\begin{theorem} \label{thm:Stein}
For any $n,m\in \N$ such that $n \ge 2m$, it holds 
\[
\mathrm{W}_2(\X,\G) \le   (\sqrt{8}+\sqrt{2})\frac{ (m+1)\sqrt{m}}{3n} . 
\]
\end{theorem}

This shows that if $m \to +\infty$ in such a way that $m = o(n^{2/3})$, then  the Kantorovich distance between the random vector $\X$ and a standard Gaussian $\G$ on $\R^{2m}$ converges to 0 as $n\to+\infty$.
The proof of Theorem~\ref{thm:Stein} is given in Section~\ref{sect:Stein} and it relies on the normal approximation method from  \citep{LLW19}, see Proposition~\ref{prop:LLW} below. 
This result allows to turn the moments' identities of \citep{DS94} into a quantitative statement about the rate of convergence to the normal distribution in the Kantorovich distance. 
Let us emphasize that the result from \citep{LLW19} which is used to prove Theorem~\ref{thm:Stein} is inspired by Stein's method and is therefore completely unrelated to the techniques that we develop in Sections \ref{sect:proofs}--\ref{sect:QF} to prove our main result. 

\begin{remark} \label{rk:beta}
If we let for $k\ge 1$, 
\[
\mathrm{X}_{2k-1} = \frac2{\sqrt{\beta k}} \sum_{j=1}^n \cos(k\theta_j)
\qquad\text{and}\qquad 
\mathrm{X}_{2k} = \frac2{\sqrt{\beta k}} \sum_{j=1}^n \sin(k\theta_j) ,
\]
then, the counterpart of Theorem~\ref{thm:Stein} also holds for the circular $\beta$--ensembles $\{\theta_1, \dots, \theta_n\}$.  That is, for any $\beta>0$, there exists a constant $C_\beta>0$ such that for all $n,m\in\N$ with  $n \ge 2m$, 
\[
\mathrm{W}_2(\X,\G) \le  C_\beta\frac{m^{3/2}}{n} . 
\]
The proof is similar to that of Theorem~\ref{thm:Stein} and it relies on Proposition~\ref{prop:LLW} and Lemma~\ref{lem:eig}. The only differences lie in that instead of using the moments' identities of \citep{DS94}, one can make use of the estimates from \citep[Theorem~1]{JM15}. 
These estimates for the joint moments of $\X$ corresponds to the analogue for general $\beta>0$ of Theorem~\ref{thm:DS} with constants which are not sharp an they are obtained by using the Jack functions instead of Schur functions as in the case of the unitary group $(\beta=2)$. 
 Then, it is straightforward to control the errors as in Lemmae~\ref{lem:A} and~\ref{lem:B}. 
 Likewise, a similar result also holds for the other classical compact groups (that is for the circular orthogonal and symplectic ensembles) with the appropriate normalization. 
 \hfill$\blacksquare$ 
\end{remark}

Let us give a final application of Theorem~\ref{thm:main} when $m$ is close to $\sqrt{n}$ and the dimension~$n$ of the random matrix $\u$ is  large. Namely, we obtain the following corollary. 

\begin{proposition}\label{cor:TVuniform}
Let us assume that  $n \ge 4322$. Then, it holds for any integer $m \le \sqrt{\frac{n}{41.5\sqrt{\log n}}}$, 
\begin{equation*} \label{TV:uniform}
{\rm d_{TV}}(\X, \G)  \le \sqrt{n}\exp \big(19.4 - 0.93\sqrt{n}(\log n)^{5/4} \big) . 
\end{equation*}
\end{proposition}

The proof of Proposition~\ref{cor:TVuniform} is also given in Section~\ref{sect:proofTV}.
We verify numerically  that under the assumptions of Proposition~\ref{cor:TVuniform},  ${\rm d_{TV}}(\X, \G)  \le 10^{-367}$ which is far below \emph{Machine Epsilon}  (of order of $10^{-33}$ for quad(ruple) precision decimal).  
In the Appendix~\ref{sect:plot}, we present further numerical plots which illustrate our estimates in the case $m=3$.

\section{Overview of the proof of Theorem	~\ref{thm:main}} \label{sect:proofs}

The core of the proof of Theorem~\ref{thm:main} is to obtain the estimate \eqref{main} for the $L^2$ distance $\Delta_{n,m}^{(2)}$ between the density $\p_{n,m}$ of the random vector $\X$ and the standard Gaussian density on $\R^{2m}$. 
Observe that by Parseval's formula, we can rewrite for any $n,m \in\N$, 
\begin{equation} \label{Delta2}
\Delta_{n,m}^{(2)}  =   \bigg(\int_{\R^{2m}} \bigg| F_{n,m}(\xi) - e^{-\|\xi\|^2/2} \bigg|^2 \d \xi \bigg)^{1/2}   , 
\end{equation}
where $F_{n,m}$ denotes the characteristic function of the random vector $\X$.
Like in the proof of \citep{Johansson97}, the general strategy is  to obtain  precise estimates for $F_{n,m}$ and  we need to distinguish different \emph{regimes} depending the parameters $\xi$, $m$ and $N= n/m$.
These \emph{regimes}  are explained in Section~\ref{sect:est} and we use  different methods to treat them.  
Compared with the arguments of \citep{Johansson97} considerable improvement is needed. There are two new challenges that come up since we allow the degree $m\in\N$ to grow with $n$ and we want to keep track carefully of the constants. 
Let us also point out that the improvements of Theorem~\ref{thm:main} come from new techniques, especially from using the \emph{Borodin--Okounkov formula} that we recall in the next section. 
We also make a more careful use of the \emph{change of variables method} from \citep{Johansson97} that we review in Section~\ref{sect:cv}. 
The main steps of the proof of the estimate \eqref{main} are presented in Section~\ref{sect:est},  while the details of the proof are given in Section~\ref{sect:mainproof}. 


\subsection{Notation}

In this section, we collect the main notation that will be use throughout the rest of this paper. 

\medskip

We let $\T = \R/[2\pi]$ and view the CUE measure \eqref{pdf} as a probability measure on $\T^n$. For any $f:\T\to\C$ which is integrable, the random variable  $ \tr f(\u) = \sum_{k=1}^n f(\theta_k)$ is well--defined  with $\E_n\big[ \tr f(\u) \big] = \widehat{f}_0$.
Then, for any $\xi\in\R^{2m}$, we have
$\X  \cdot \xi = \tr \g(\u)$ where $\g$ is a real--valued trigonometric polynomial:
\begin{equation} \label{g_function}
\g(\theta)= \sum_{\begin{subarray}{c}|k| \le m \\ k\neq 0 \end{subarray}}  \frac{\zeta_k}{\sqrt{2|k|}} e^{\i k \theta} , 
\end{equation}
with $\zeta_k =\xi_{2k-1} -\i \xi_{2k}$ and  $\zeta_{-k} = \overline{\zeta_{k}}$ for all $k = 1, \dots, m$. 
In particular the characteristic function of the random vector $\X$ can be written as
\begin{equation} \label{Fg}
\begin{aligned}
F_{n,m}(\xi) & :=  \int_{\R^{2m}}  e^{\i \xi \cdot x}  \p_{n,m}(x)\d x  \\
& = \E_n\big[ e^{\i \tr \g(\u)}\big]    . 
\end{aligned}
\end{equation}

For any function $f\in L^1$, we define its Fourier coefficients for all $k\in\Z$,
\[
\widehat{f}_k = \int_\T f(\theta) e^{-\i k \theta} \frac{\d\theta}{2\pi}  . 
\]
Then, we define the following (semi)--norm 
\[
\| f\|_{H^{1/2}}^2  = \sum_{k\in \Z} |k| |\widehat{f}_k|^2 . 
\]
If  $f\in H^{1/2}$, that is if $f\in L^1$ and $ \| f\|_{H^{1/2}}^2 <+\infty$, we let 
\begin{equation}  \label{A}
\A(f) = \sum_{k\ge 1} k \widehat{f}_k  \widehat{f}_{-k}  .
\end{equation}
If the real-valued function $f$ lies in the Sobolev space $H^1$, we also verify that 
\[
\| f\|_{H^{1/2}}^2   = - \int f'(\theta) \U f(\theta) \frac{\d\theta}{2\pi} , 
\]
where $\U f = -\sum_{k\in\Z} \i \sgn(k)   \widehat{f}_k e^{\i k \theta}  $ denotes the Hilbert tranform of $f$. 



\subsection{Preliminaries: Toeplitz determinants and the Borodin--Okounkov formula} \label{sect:BO}

Recall that the CUE refers to a random matrix $\u$ which is distributed according to the Haar measure on the unitary group  $U(n)$ and that the eigenvalues of $\u$ have a joint law which is explicitly given by \eqref{pdf}. 
One of the most remarkable feature of the CUE is the connection with Toeplitz determinants. 
Namely, for any integrable function $w= e^f $, $f:\T \to \C$ and $n\in\N$,  if $\tr f(\u) = \sum_{j=1}^n f(\theta_j)$, then we have
\begin{equation} \label{Toeplitz}
\E_n\big[ e^{\tr f(\u)} \big] = \det_{n\times n}[  \widehat{w}_{i-j} ].
\end{equation}

Formula \eqref{Toeplitz} implies that we can obtain the asymptotics of the Laplace transform of the random variable $\tr f(\u)$ by using the Strong Szeg\H{o} limit theorem. 

\begin{theorem} \label{thm:Szego}
If $f\in H^{1/2}$, then as  $n\to+\infty$,
\begin{equation} \label{Szego}
\E_n\big[ e^{\tr f(\u)} \big]  = \exp\left( n \widehat{f}_0 + \A(f) + o(1) \right) ,
\end{equation}
where $\A(f) = \sum_{k\ge 1} k \widehat{f}_k  \widehat{f}_{-k}  \in\C$. 
\end{theorem}

The first version of Theorem~\ref{thm:Szego} was first proved by \citep{Szego52} when $f\in C^{1,\alpha}$ is real--valued.
The hypothesis from Theorem~\ref{thm:Szego} are optimal and this version was first obtained for real--valued $f$ by \citep{Ibragimov68} and \citep{GI71}. 
We refer to the survey paper of \citep{DIK13} for a history of the Szeg\H{o} Strong Limit theorem and its later generalizations and to the book \cite[Chapter~6]{Simon05} for a detailed presentation of several proofs. 
A proof of Theorem~\ref{thm:Szego} which holds for complex--valued $f$ can be found in \citep{Johansson88}.

\medskip

Actually, one can also obtain Theorem~\ref{thm:Szego} as a consequence of the \emph{Borodin--Okounkov formula}. 
This formula  expresses the Toeplitz determinant \eqref{Toeplitz} in terms of Fredholm determinant which is more amenable for asymptotic analysis.
If $f :\T\to \C$ is an $L^2$ function, we denote
\[
f^+(\theta) = \sum_{k\ge 1}  \widehat{f}_k e^{\i k\theta} , 
\qquad\qquad
f^-(\theta) = \sum_{k\ge 1}  \widehat{f}_{-k} e^{-\i k\theta} . 
\]
Let $w:\T\to\C$ be an integrable function such that $\sum_{k\in \Z} |k| |\widehat{w}_k|^2<+\infty$, and define  two Hankel operators:
\begin{equation} \label{HO}
H_+(w) = \begin{pmatrix} 
\widehat{w}_1 & \widehat{w}_2 &\widehat{w}_3 & \hdots \\
\widehat{w}_2 &\widehat{w}_3 & \widehat{w}_4 & \hdots \\
\widehat{w}_3 & \widehat{w}_4 & \widehat{w}_5 & \hdots \\
\vdots & \vdots &\vdots &\ddots
\end{pmatrix} 
\qquad\text{and}\qquad
H_-(w) =\begin{pmatrix} 
\widehat{w}_{-1} & \widehat{w}_{-2} & \widehat{w}_{-3} & \hdots \\
\widehat{w}_{-2} & \widehat{w}_{-3} & \widehat{w}_{-4} &\hdots \\
\widehat{w}_{-3} & \widehat{w}_{-4} & \widehat{w}_{-5} &\hdots \\
\vdots & \vdots &\vdots &\ddots
\end{pmatrix}  .
\end{equation}
Note that the condition  $\sum_{k\in \Z} |k| |\widehat{w}_k|^2<+\infty$ guarantees that these operators are Hilbert--Schmidt on $L^2(\N)$. We also denote by $(e_1, e_2, \cdots)$ the standard basis of $L^2(\N)$. 

\begin{theorem} \label{thm:BO}
Let $f:\T \to \C$ be a $L^\infty$ function such that  $\sum_{k\in \Z} |k| |\widehat{f}_k|^2<+\infty$ and $\widehat{f}_0 = 0$. Let us also define 
\begin{equation} \label{BOK}
K_f = H_+(e^ { f^- - f^+}) H_-(e^ { f^+ - f^-}) .
\end{equation}
The operator $K_f$  is trace--class and for any $n\in\N$,
\begin{equation} \label{BO}
\E_n\big[ e^{\tr f(\u)} \big]  =  e^{\A(f)} \det[\operatorname{I} -  K_f  Q_n] ,
\end{equation}
where $Q_n$ denotes the orthogonal projection with kernel $\operatorname{span}(e_1, \dots, e_{n-1})$  and the RHS is a Fredholm determinant on $L^2(\N)$. 
\end{theorem}

Since the operator $K_f$ is trace class,  by definition of $Q_n$, we have $\det[\operatorname{I} -  K_f  Q_n] \to 1$ as $n\to+\infty$, so that Theorem~\ref{thm:BO} implies the Szeg\H{o} Strong Limit theorem. 
The \emph{Borodin--Okounkov formula} \eqref{BO} (sometimes also known as \emph{Geronimo--Case formula}) first appeared (formally) in \citep{GC79}.  
\citep{BO00} proved formula \eqref{BO} in a different form when $f$ is analytic using Gessel's Theorem which allows to express Toeplitz determinants as series in Schur functions, \citep{Gessel90}.    
The version from Theorem~\ref{thm:BO} is due to \citep{BW00} -- see also \citep{Bottcher02} for a different proof.
It is possible to remove the condition $f\in L^\infty$ from Theorem~\ref{thm:BO}, see e.g.  \cite[Chapter~6.2]{Simon05}. 

\medskip

Concerning our method, let us point out that in order to obtain the \emph{super-exponential} rate of convergence in  \eqref{JTV}, \citep{Johansson97}
relied on exact formulae for Toeplitz determinants with  certain specific symbols  which are due to \citep{Baxter61} and relates to the original proof of the Strong Szeg\H{o} theorem. 
Observe that according to \eqref{g_function}, we have $\A(\i\g) =- \|\zeta\|^2 /2 =- \|\xi\|^2/2$ so that by \eqref{Fg} and \eqref{BO}, we can rewrite for all $n,m \in\N$ and $\xi\in\R^{2m}$,
\begin{equation} \label{BOg}
F_{n,m}(\xi)  = e^{-\|\xi\|^2/2} \det[\operatorname{I} -  K_{\i\g}  Q_n]  . 
\end{equation}
Hence, by controlling precisely how close the Fredholm determinant $\det[\operatorname{I} -  K_{\i\g}  Q_n]$ is to 1, we are able to significantly improve the rate of convergence from \citep{Johansson97}. 
Even though this might be difficult to verify, it is natural to expect that modulo corrections, $1/\Gamma(N+1)$ should be the true rate of  convergence in Theorem~\ref{thm:main} in the regime where $m\ll N = \frac nm$. 

\medskip 

Throughout this article, we also make crucial use of the following bound. 

\begin{lemma} \label{lem:Laplace}
Suppose that $f\in\Co(\T)$ is real--valued with $\A(f) <+\infty$ where $\A$  is as in \eqref{A}. Then for any $n\in\N$, 
\begin{equation}\label{Laplace}
\E_n\big[e^{ \tr f(\u) } \big] 
\le \exp \big( n\widehat{f}_0 + \A(f)\big) . 
\end{equation}
\end{lemma}

Let us recall that $\E_n[\tr f(\u)] = n\widehat{f}_0 $ and that by Theorem~\ref{thm:Szego}, 
$\operatorname{Var} [\tr f(\u)]  \to  2\A(f)$ as $n\to+\infty$, so that the estimate \eqref{Laplace} is sharp. 
The upper--bound \eqref{Laplace} is classical and it follows for instance from the monotonicity of Toeplitz determinants, \citep{OPUC}. 
For completeness, we show in the appendix (Section~\ref{sect:proof_Laplace}), how one can immediately deduce Lemma~\ref{lem:Laplace} from the \emph{Borodin--Okounkov formula}.


\subsection{Change of variables} \label{sect:cv}

In addition to the \emph{Borodin--Okounkov formula} (Theorem~\ref{thm:BO}) and  Lemma~\ref{lem:Laplace}, our  main tool to prove Theorem~\ref{thm:main} is the change of variables method introduced in \citep{Johansson88}. 
More specifically we rely on an estimate from the proof of  \cite[Proposition~2.8]{Johansson97}. 
Recall that according to \eqref{Fg}, $F_{n,m}$ denotes the characteristic function of the random variable $\tr\g(\u)$. 

\begin{lemma} \label{lem:estF}
Let $\nu>0$ and $h:\T\to\R$ be a $C^1$ function. Then, for any $n,m \in\N$  and $\xi\in\R^{2m}$, 
\[
\big| F_{n,m}(\xi) \big| \le  \E_n\bigg[ \prod_{1\le i<j \le n}   \bigg| \frac{ \sin\big(\frac{\theta_i - \theta_j}{2} + \i \nu \frac{ h(\theta_i) - h(\theta_j)}{2n}  \big)}{\sin\big(\frac{\theta_i - \theta_j}{2}\big)} \bigg|^2  \prod_{j =1}^n  \big| 1+ \i \tfrac \nu n h'(\theta_j)\big| e^{- \Im \g\big(\theta_j +\i \frac{\nu}{n}h(\theta_j)\big)}  \bigg] . 
\]
\end{lemma}

\begin{proof}
For completeness, let us give the proof of Lemma~\ref{lem:estF}. 
Using the explicit formula \eqref{pdf} for the joint law of the eigenvalues of the random matrix $\u$, we obtain
\begin{equation*} 
F_{n,m}(\xi) 
= \mho_n \int_{[-\pi,\pi]^n}
\prod_{1\le i<j \le n} \sin^2\left(\frac{\theta_i-\theta_j}{2}\right) \prod_{k =1}^n  e^{\i \g(\theta_k)} \frac{\d\theta_k}{2\pi} .
\end{equation*}
If we regard $\theta_k$ as complex variables in the previous integral, since the integrand is a entire function, we can deform the contours of integration in the complex plane. Let $\boldsymbol{\gamma} $ be a positively oriented curve given by
\[ 
\boldsymbol{\gamma} = \big\{ \theta + \i \tfrac \nu n h(\theta) : \theta\in[-\pi,\pi] \big\}  . 
\]
Since the functions $\g$ and $\sin^2(\cdot/2)$ are also  $2\pi$--periodic, we have by Cauchy's theorem,
\begin{align} \label{F}
F_{n,m}(\xi) 
&= \mho_n \int_{\boldsymbol{\gamma} ^n}
\prod_{1\le i<j \le n} \sin^2\left(\frac{\theta_i-\theta_j}{2}\right) \prod_{k =1}^n  e^{\i\g(\theta_k)} \frac{\d\theta_k}{2\pi} \\
&\notag
= \mho_n\int_{[-\pi,\pi]^n} \prod_{1\le i<j \le n}  \bigg( \sin\bigg(\frac{\theta_i - \theta_j}{2} + \i \nu \frac{ h(\theta_i) - h(\theta_j)}{2n}  \bigg) \bigg)^2 \prod_{j =1}^n  e^{\i\g\big(\theta_j +\i \frac{\nu}{n}h(\theta_j)\big)} \big( 1+ \i \tfrac \nu n h'(\theta_j)\big) \frac{\d\theta_j}{2\pi} . 
\end{align}
Hence, by \eqref{pdf}, this implies that
\begin{align*}
\big| F_{n,m}(\xi) \big|
&\notag
\le \mho_n\int_{[-\pi,\pi]^n} \prod_{1\le i<j \le n}  \bigg| \sin\bigg(\frac{\theta_i - \theta_j}{2} + \i \nu \frac{ h(\theta_i) - h(\theta_j)}{2n}  \bigg) \bigg|^2 \prod_{k =1}^n  e^{- \Im\g\big(\theta_j +\i \frac{\nu}{n}h(\theta_j)\big)} \big| 1+ \i \tfrac \nu n h'(\theta_j)\big| \frac{\d\theta_j}{2\pi}  \\
&= \E_n\bigg[ \prod_{1\le i<j \le n}   \bigg| \frac{ \sin\big(\frac{\theta_i - \theta_j}{2} + \i \nu \frac{ h(\theta_i) - h(\theta_j)}{2n}  \big)}{\sin\big(\frac{\theta_i - \theta_j}{2}\big)} \bigg|^2  \prod_{j =1}^n  \big| 1+ \i \tfrac \nu n h'(\theta_j)\big| e^{- \Im\g\big(\theta_j +\i \frac{\nu}{n}h(\theta_j)\big)}  \bigg] . 
\end{align*}
\end{proof}

The key idea underlying this change of variables is that the eigenvalues of $\u$ are almost uniformly distributed on the unit circle (like the vertices of a regular $n$-gon).
This means that at first order, we can approximate  the empirical measure $\sum_{j =1}^n  \delta_{\theta_j} \simeq n\frac{\d\theta}{2\pi}$.  
Chooe $h = \U \g$ where $\U$ is the Hilbert transform:
\begin{equation} \label{h}
h(\theta) = -\sum_{\begin{subarray}{c}|k| \le m \\ k\neq 0 \end{subarray}}  \sgn(k)\frac{\i\zeta_k}{\sqrt{2|k|}}  e^{\i k \theta}.
\end{equation}
By making the change of variables $\theta_j $ by $\theta_j +\i \frac{\nu}{n}h(\theta_j)$ in \eqref{F}, we  expect that using first order Taylor approximations:
\begin{equation*} 
\begin{aligned} 
F_{n,m}(\xi) 
& \simeq \mho_n \int_{[-\pi,\pi]^n}
\prod_{1\le i<j \le n} \sin^2\left(\frac{\theta_i-\theta_j}{2}\right) e^{\frac{\nu^2}{n^2} H(\theta_i,\theta_j)} \prod_{j =1}^n  e^{\i \g(\theta_j)- \frac{\nu}{n} \g'(\theta_j) h(\theta_j) + \frac{\nu^2}{n^2} h'(\theta_j)^2} \frac{\d\theta_j}{2\pi}  \\
& = \E_n\Big[  e^{\frac{\nu^2}{2n^2} \sum_{i,j =1}^n H(\theta_i,\theta_j)} e^{\i \sum_{j =1}^n \g(\theta_k)- \frac{\nu}{n} \sum_{j =1}^n \g'(\theta_k) h(\theta_k)}  \Big] ,
\end{aligned}
\end{equation*}
where 
\begin{equation} \label{H}
H(\theta, x)= \begin{cases} \left( \frac{h(\theta)-h(x)}{2\sin(\frac{\theta-x}{2})}\right)^2  & x\neq \theta \\
h'(\theta)^2 & x=\theta 
\end{cases} . 
\end{equation}

Then, since $\widehat{\g}_0 =0$,  we expect that
\[
F_{n,m}(\xi) 
\simeq \exp \left( - \nu \int_{[0,2\pi]}   \hspace{-.3cm}  g'(\theta) h(\theta)  \frac{\d\theta}{2\pi} +  \frac{\nu^2}{2} \iint_{[0,2\pi]^2}   \hspace{-.3cm}H(\theta, x) \frac{\d\theta}{2\pi}\frac{\d x}{2\pi} \right) . 
\]
Then, by Devinatz's formula \cite[Proposition 6.1.10]{OPUC},  since $h =- \U\g$,  we have 
\begin{equation} \label{Devinatz1}
\| h\|_{H^{1/2}}^2  = \int_{[0,2\pi]}   \hspace{-.3cm} \g'(\theta) h(\theta)  \frac{\d\theta}{2\pi} =\iint_{[0,2\pi]^2} \hspace{-.3cm} H(\theta,x) \frac{d\theta}{2\pi} \frac{dx}{2\pi}    
\end{equation}
and 
\begin{equation} \label{Devinatz2}
\| h\|_{H^{1/2}}^2  = \sum_{k\in \Z} |k| |h_k|^2= \sum_{k=1}^m |\zeta_k|^2 =\|\xi\|^2 .
\end{equation}
Whence it follows from this heuristic with $\nu=1$ that
\[
F_{n,m}(\xi) \simeq  e^{-\|\xi\|^2/2} . 
\]

To turn this heuristics rigorous, one needs to justify the approximation $\sum_{j =1}^n  \delta_{\theta_j} \simeq n\frac{\d\theta}{2\pi}$ and to control the errors coming from the Taylor expansions. 
This can  be done by using \emph{rigidity estimates} for the CUE eigenvalues, see \citep{Lambert19}, but we present a different approach below (see $\mathbf{iii)}$ Intermediate regime in the next section). 


\subsection{Estimates for the function $F_{n,m}(\xi)$  in the different regimes} \label{sect:est}

Recall that we let $N=\frac nm$ and that our main goal is to obtain the following bound. 

\begin{proposition} \label{prop:Delta2}
For any $n,m\in \N$ such that $m\ge 3$ and $ N = n/m > 4m$, we have
\begin{equation} \label{mainest}
\Delta_{n,m}^{(2)}  =   \bigg(\int_{\R^{2m}} \bigg| F_{n,m}(\xi) - e^{-\|\xi\|^2/2} \bigg|^2 \d \xi \bigg)^{1/2}   
\le   \cst{3}  \sqrt{\Omega_m} N^{\frac m2}  \Theta_{N,m}   , 
\end{equation}
where $\Theta_{N,m}$ is as in \eqref{Theta0}--\eqref{Theta}.
\end{proposition}

In this section, we present the main estimates for the characteristic function $F_{n,m}(\xi)$ that are required to prove Proposition~\ref{prop:Delta2}.  We postpone the technical details of the arguments to Sections~\ref{sect:GA}--\ref{sect:QF}. All the constants $c_j$ used below, which can depend on $m$ are defined in the Appendix~\ref{sect:approx}
Let us define
\begin{equation} \label{Lambda1}
\Lambda_1 =   \frac{ \cst{4} N}{\sqrt{1+\log m}}  . 
\end{equation}

The proof consists in splitting the integral on the LHS of \eqref{mainest} in three different regimes depending on whether
$\mathbf{i)}\ \| \xi\| \le \Lambda_1$, $\mathbf{ii)}\ \| \xi\| \ge \Lambda_3$ or 
$\mathbf{iii)}\ \Lambda_1 \le \| \xi\| \le \Lambda_3$
where $\Lambda_3 \gg \Lambda_1$ is a parameter that we will choose later.

\medskip

\underline{$\mathbf{i)}$ Gaussian approximation for $\| \xi\| \le \Lambda_1$.}
In this regime, our goal is to compare the characteristic function $F_{n,m}$ with that of a $2m$--dimensional standard Gaussian by using the \emph{Borodin--Okounkov formula} from Theorem~\ref{thm:BO}.
We obtain the following estimates. 

\begin{proposition} \label{prop:GA}
Under the assumptions of Proposition~\ref{prop:Delta2}, we have for all $\xi \in \R^{2m}$ such that $\|\xi\| \le \Lambda_1$,
\[  
\left| F_{n,m}(\xi)  - e^{-\|\xi\|^2/2} \right|^2 \le \cst{8}^2 m^4  e^{2\sqrt{2(1+\log m)} \|\xi\|} \left(\frac{1+\log m}{2}\right)^{2N} \frac{\|\xi\|^{4N}}{\Gamma(N+1)^4} e^{-\|\xi\|^2} . 
\]
\end{proposition}


Let us point out  that Proposition~\ref{prop:GA} gives the main contribution $\Tf{0}$ to $\Delta_{n,m}^{(2)}$. We expect that the main error in the normal approximation should come from the regime where $\|\xi\|$ is not too large. The proof of Proposition~\ref{prop:GA} is given in Section~\ref{sect:GA}. 
Let us observe that according to formula \eqref{BOg}, we have 
\[
\left| F_{n,m}(\xi)  - e^{-\|\xi\|^2/2} \right|^2    =   \big|1- \det[\operatorname{I} -  K_{\i\g}  Q_n]   \big|^2 e^{-\|\xi\|^2} ,
\]
and we expect that if both the degree $m$ and  $\|g\|_{H^{1/2}}^2 = \|\xi\|^2$ are sufficiently small (depending on the dimension $n\in\N$ of the random matrix $\u$), then by definition of the projection~$Q_n$, the operator $K_{\i\g}  Q_n$ is also small (in trace norm) so that $\det[\operatorname{I} -  K_{\i\g}  Q_n] \simeq 1$. 
This can be quantified by using the bound for Fredholm determinant from 
\cite[Theorem 3.4]{Simon05}, 
\begin{equation} \label{J1inequality}
\big|1- \det[\operatorname{I} -  K_{\i\g}  Q_n]   \big|
\le  \| K_{\i\g}  Q_n\|_{\mathscr{J}_1} e^{1+ \| K_{\i\g}  Q_n\|_{\mathscr{J}_1}} ,
\end{equation}
where $\| \cdot\|_{\mathscr{J}_1}$ denotes the Schatten 1-norm or trace norm of an operator. 
Then, in order to  compute $\| K_{\i\g}  Q_n\|_{\mathscr{J}_1}$, we use the product structure of the operator $K_{\i\g}$, \eqref{BOK}, and the Cauchy--Schwartz inequality  (for the Hilbert--Schmidt norm $\| \cdot\|_{\mathscr{J}_2}$):
\[
\| K_{\i\g}  Q_n\|_{\mathscr{J}_1} \le  \| Q_n H_+(e^ {2 \Im g^+})\|_{\mathscr{J}_2} \| H_-(e^ {-2 \Im g^+}) Q_n\|_{\mathscr{J}_2} . 
\]
Moreover, since $H_{\pm}(\cdot)$ are Hankel operators \eqref{HO},  we can estimate the norms $\| H_{\pm}(e^ {-2 \Im g^{\pm}}) Q_n\|_{\mathscr{J}_2}$  by obtaining bounds for the Fourier coefficients of the symbols $e^ {-2 \Im g^{\pm}}$, see Lemma~\ref{Lem:F_coeff} below.
To sum up, we show in  Section~\ref{sect:GA} that $ \| Q_n H_{\pm}(e^ {2 \Im g^{\pm}})\|_{\mathscr{J}_2} \ll 1/ \Gamma(1+N)$ provided that $\|\xi\| \ll N$ and we use this estimate to deduce Proposition~\ref{prop:GA}.
Let us emphasize again that we expect that these estimates are of the right order and hold only in the regime where $\|\xi\| \ll N$. 

\medskip

\underline{$\mathbf{ii)}$ Tail bound for large $\|\xi\|$}. 
If $\|\xi\|$ is very large, we are not looking to compare $F_{n,m}$  with the characteristic function of a standard Gaussian, but rather aiming at obtaining a \emph{good} tail bound  for~$F_{n,m}$. By \emph{good}, we mean that we aim for estimates which yield errors that are smaller than $\Tf{0}$ when $N$ is sufficiently large. 
\cite[Proposition~2.13]{Johansson97} used the Hadamard's inequality 
\begin{equation} \label{Hadamard_inequality}
\big|  F_{n,m}(\xi) \big|^2\le  \prod_{j=1}^n \sum_{i=1}^n \big| (\widehat{e^{\i\g}})_{j-i}\big|^2 
\end{equation}
and an estimate for the Fourier coefficients of the function $e^{\i\g}$ to obtain the tail bound
$\big|  F_{n,m}(\xi) \big|^2 \le \frac{C^n n^{\frac{3n}{2}}}{\|\xi\|^{\frac N 2}}$ for a constant $C>0$. 
By using \eqref{Hadamard_inequality} and a (sharp) \emph{Van der Corput's inequality}, this estimate can be improved and we obtain for all $m, n\ge 3$ and $\xi\in\R^{2m}$,
\begin{equation} \label{TB1}
\big|  F_{n,m}(\xi) \big|^2 \le \frac{\cst{}^n n^n}{\|\xi\|^{\frac{n}{m+1}}} , \qquad \cst{}= 4\pi e(1+1/\sqrt{3}).
\end{equation}
Too obtain a good multi--dimensional approximation for a growing number of traces we would like to have a better estimate that does not contain the very large factor $n^n$. 
We can obtain a different tail bound by relying on Lemma~\ref{lem:estF} with  $h = \g'$. Choosing $\nu>0$ appropriately, we obtain 
\begin{equation}  \label{F1}
\big|  F_{n,m}(\xi)  \big| \le e^{\cst{} n} \E_n \left[ e^{-  \gamma \sum_{j =1}^n  \g'(\theta_j)^2 } \right] ,
\end{equation}
for a constant $\cst{}>0$ and $\gamma\to0$ as $n\to+\infty$ -- see Proposition~\ref{prop:TB1} below for further details.
Then,  to estimate  the RHS of \eqref{F1}, we use that 
\begin{equation} \label{F2}
\E_n \left[ e^{- \gamma \sum_{j =1}^n \g'(\theta_j)^2} \right] 
\le  \frac{e^n}{\sqrt{2\pi n}}  \left( \int_\T  e^{- \gamma \g'(\theta)^2} \frac{d\theta}{2\pi} \right)^n, 
\end{equation}
and  since $\g':\T\to\R$ is a trigonometric polynomial of degree $m\in\N$, 
\begin{equation} \label{F3}
\int_\T  e^{- \gamma \g'(\theta)^2} \frac{\d\theta}{2\pi} 
\le \frac{2e}{(2 \gamma \|\g'\|_{L^2}^2)^{1/4m}} . 
\end{equation}
The estimate \eqref{F2} is rather classical and its proof is given in the appendix -- Lemma~\ref{lem:TB} -- for completeness. 
On the other--hand, \eqref{F3} relies on an estimate of the measure of the set where a trigonometric polynomial is small by its $L^2$ norm which is taken from \citep{Chahkiev08} -- see Lemma~\ref{lem:levelset} below. 
By combining these estimates, we obtain the following bound in Section~\ref{sect:TB}. 

\begin{proposition} \label{prop:Had}
Fix $n,m\in\N$ and suppose that $N \ge 4m$. For any $\xi \in \R^{2m}$, we have 
\[
\big|  F_{n,m}(\xi)  \big|^2   \le \ups{3}(m)^{N/2}  \frac{\cst{15}^{2n} N^{N/4}}{\|\xi\|^{N/2}} . 
\]
\end{proposition}

While this tail bound has a worse decay in  $\|\xi\|$ than \eqref{TB1}, the factor $N^N$ is significantly better than $n^n$ when the  degree $m\in\N$  is large. 
Moreover, we see that this estimate will be useful in the proof of Proposition~\ref{prop:Delta2} in the regime where $\|\xi\|\gg N^{\cst{}}$ for a  sufficiently large constant $\cst{}$.
In the proof, we will actually choose $\Lambda_3 = e^{4\cst{1}\frac{N}{1+\log m}}$ times some corrections --  see formula \eqref{Lambda3} below. 

\medskip

\underline{$\mathbf{iii)}$ Intermediate regime.} It remains to deal with the intermediate regime where $\Lambda_1 \le \| \xi\| \le \Lambda_3$.
As we already pointed out, when $ \| \xi\| \gg N$,  we do not expect that the Fredholm determinant $ \det[\operatorname{I} -  K_{\i\g}  Q_n]$ is close to 1.  However, we still expect that $\big|  F_{n,m}(\xi) \big|^2 \ll 1/\Gamma(1+N)^2$ for all such $\xi\in\R^{2m}$. 
From a technical perspective, this intermediate regime is the most challenging one because the direct estimates (e.g.~the method used in \cite[Section 2.2]{Johansson97}) lead to errors which are bigger than that of  Proposition~\ref{prop:GA}  -- see the estimate \eqref{naivebound} below. Our final bounds are summarized in the next proposition. 
Define 
\begin{equation} \label{Lambda2}
\Lambda_2 = \frac{\cst{0}^{-1}(1-\cst{10})N \sqrt{m+1}}{8(1+\log m)^{3/4}\cst{11}}  ,
\end{equation}
where $\cst{1}(m)$, $\cst{2}(m)$, $\cst{10}(m)$ and $\cst{11}(m)$ are as in  \eqref{c1}. 
We verify that  both $\cst{10}(m)$ and $\cst{11}(m)$  are decreasing for $m\ge 3$. Since $\cst{10}(3) \approx 0.0124$ and $\cst{11}(3) \approx 1.583$,    this shows that $\Lambda_2 \ge \Lambda_1$ and $\Lambda_2$ is increasing as a function of $m$ for all $m\ge 3$. 

\begin{proposition} \label{prop:T0}
Fix $m ,n \in \N$ with $m \ge 3$.     We have for all $\xi \in \R^{2m}$, 
\begin{equation} \label{bd1}
\big| F_{n,m}(\xi) \big|
\le \exp\bigg( \cst{9} -  \frac{\cst{1}(m)N^2 }{1+\log m} \bigg)  \qquad\text{if } \|\xi\| \ge  \Lambda_2,
\end{equation}
and 
  \begin{equation} \label{bd2}
\big| F_{n,m}(\xi) \big| \le   \exp\bigg( \cst{9}-   \frac{\cst{2}(m) N^2}{\sqrt{m+1} (1+\log m)^{3/4}}  \bigg)    
\qquad\text{if } \Lambda_1 \le  \|\xi\| \le  \Lambda_2 . 
\end{equation}
\end{proposition}

Let us observe that these bounds directly relate to the error terms $\Tf{1}$ and $\Tf{2}$ from \eqref{Theta1} and \eqref{Theta2}  respectively. 
The proof of Proposition~\ref{prop:T0} is given in Section~\ref{sect:QF}. 
The starting point of this proof is the change of variables and the heuristics described in Section~\ref{sect:cv}.  Namely, using Lemma~\ref{lem:estF} with  $h = - \U \g$  as in \eqref{h}, we obtain  the following bound. 

\begin{lemma} \label{lem:est1}
Let $n,m\in\N$  and $\xi \in \R^{2m}$. 
We have for any $\nu>0$, 
\begin{equation} \label{F4}
\big| F_{n,m}(\xi) \big|^2
\le  \E_n\bigg[\exp\bigg(\frac{\nu^2}{n^2} \sum_{i,j = 1}^n H(\theta_i , \theta_j) \bigg) \bigg] \E_n\bigg[ e^{-2 \sum_{j=1}^n \Im \g\big(\theta_j +\i \frac{\nu}{n}h(\theta_j)\big)}  \bigg] ,
\end{equation}
where the function $H$ is given by \eqref{H}. 
\end{lemma}

The proof of Lemma~\ref{lem:est1} is postponed to Section~\ref{sect:est1}. 
Using Lemma~\ref{lem:Laplace}, we can easily control the second factor in the RHS of \eqref{F4}. 
We obtain that  there exists a constant $\cst{}>0$ such that  if $\| \xi\| \ll \Lambda_2$, 
\begin{equation} \label{TB2}
\E_n\bigg[ e^{-2 \sum_{j=1}^n \Im g\big(\theta_j +\i \frac{\nu}{n}h(\theta_j)\big)}  \bigg]  \le e^{-\cst{}\nu \|\xi\|^2}
\end{equation}
see Lemma~\ref{prop:est2} below for further details.
Moreover, using the deterministic bound $H(\theta_i, \theta_j) \le \|h'\|_{\infty}^2$ in \eqref{F4} combined with $\|h'\|_{\infty}\le \sqrt{2} m \|\xi\|$, this implies that
\[
\big| F_{n,m}(\xi)  \big| \le\exp\left( - \cst{} \nu \|\xi\|^2 + 2\nu^2 m^2 \|\xi\|^2 \right) . 
\]
If we optimize over $\nu>0$, this leads to
\begin{equation} \label{naivebound}
\big| F_{n,m}(\xi)  \big| \le \exp\left( - \frac{\cst{}^2\|\xi\|^2}{8m^2} \right) . 
\end{equation}

In the regime where the degree $m \in \N$ depends on the dimension $n\in\N$ with $N  \ge 4m$, the naive estimate \eqref{naivebound} is not precise enough to lead to errors which are small compared with that of Proposition~\ref{prop:GA}. 
Therefore, to prove Proposition~\ref{prop:T0},  we need to introduce a new idea. One approach would be to use precise rigidity estimates from \citep{Lambert19} to obtain a better estimate for the first term on the RHS of \eqref{F4}. But, the method that we use consists in writing $\sum_{i,j =1}^n H(\theta_i,\theta_j)$ as a  quadratic form in the random  variables  $\mathrm{T}_k = \tr \u^k$, 
\[
\sum_{i,j =1}^n H(\theta_i,\theta_j) = n^2  \iint_{[0,2\pi]^2} \hspace{-.3cm} H(\theta,x) \frac{d\theta}{2\pi} \frac{dx}{2\pi}    + n \left( \mathbf{a}^*\mathbf{T} +  \mathbf{T}^* \mathbf{a}  \right)  +  \mathbf{T}^* \mathbf{M}\mathbf{T} ,
\qquad \text{where }\
\mathbf{T} = {\small \begin{pmatrix} \mathrm{T}_1 \\ \vdots \\ \mathrm{T}_{2m-1}\end{pmatrix}},
\]
$\mathbf{a}(\xi) \in \C^{2m}$ is a deterministic vector and  $\mathbf{M}(\xi) \in \C^{2m \times 2m}$ is a deterministic matrix which depend on $\xi\in\R^{2m}$, see Lemma~\ref{lem:QF} below. 
 This allows us to express the Laplace transform of the random variable  $\sum_{i,j =1}^n H(\theta_i,\theta_j)$ as an integral against a Gaussian measure on $\C^{2m}$. 
It is not at all clear that the matrix $\mathbf{M}(\xi)$ is positive definite, but {\it if} it were (see the Remark~\ref{rk:pd}), by formulae \eqref{Devinatz1}--\eqref{Devinatz2}, we would obtain for any $\nu>0$, 
\[
\E_n\Big[  e^{ \frac{\nu^2}{n^2} \sum_{i,j =1}^n H(\theta_i,\theta_j)}\Big] = \frac{e^{\nu^2\|\xi\|^2}}{\pi^{2m} \det( \mathbf{M})}  \int_{\C^{2m-1}}  
e^{- \mathbf{z}^* \mathbf{M}^{-1} \mathbf{z}}\  \E_n\big[    e^{  \frac{\nu}{n} ( \mathbf{z}^*\mathbf{T} +  \mathbf{T}^*\mathbf{z}) +  \frac{\nu^2}{n}(\mathbf{a}^*\mathbf{T}+ \mathbf{T}^* \mathbf{a}  ) } \big] \d\mathbf{z} . 
\]
where $ \d\mathbf{z}$ is the Lebesgue measure on $\C^{2m-1}$.
The idea is now to use Lemma~\ref{lem:Laplace} to estimate the expectation on the RHS of the previous formula and then to evaluate the Gaussian integral. 
The details in the implementation of this idea, which requires a modification of $\mathbf{M}$, is somewhat involved and we refer to the proof in Section~\ref{sect:QF} for further details. 


\section{Proof of the main result} \label{sect:mainproof}

\subsection{Proof of Proposition~\ref{prop:Delta2} }

In this section, we give the proof of Proposition~\ref{prop:Delta2} relying on the estimates from Propositions~\ref{prop:GA}, \ref{prop:Had} and \ref{prop:T0}. 
Recall that we assume that $m\ge 3$ and $N= \frac nm >4m$. Then, using the notation \eqref{c1} and \eqref{c2}, we define
\begin{equation} \label{Lambda3}
\Lambda_3 = e^{-4\cst{9}/N} \cst{15}^{4m} \left(\frac{N}{4m}-1\right)^{2/N}   \ups{3}(m) \sqrt{N}  \exp\bigg(\frac{4\cst{1}(m)N }{1+\log m} \bigg) .
\end{equation}
Let us also recall that $\Lambda_1$ is given  by \eqref{Lambda1} and  $\Lambda_2$ is given  by \eqref{Lambda2}, so that  $\Lambda_3 \gg \Lambda_2 \gg \Lambda_1$ as $n\to+\infty$ (and possibly $m\to+\infty$). 
We will also need the following bound which is proved in the appendix (Section~\ref{sect:proof_GT}). 

\begin{lemma} \label{lem:GaussianTail} 
For any $m\in\N$, if $\Lambda > \sqrt{m}$, 
 \[
\int_{\|\xi\| \ge \Lambda} \hspace{-.3cm}  e^{-\|\xi\|^2}   \d \xi
\le  \Omega_{m}  \frac{e^{-\Lambda^2}}{\Lambda^2-m}  . 
\] 
\end{lemma}

Since $\Lambda_1^2 \ge 2m$ for all $m\ge 3$, it follows from  Lemma~\ref{lem:GaussianTail} that
\begin{equation} \label{est:GaussianTail}
\int_{\|\xi\| \ge \Lambda_1} \hspace{-.3cm}  e^{-\|\xi\|^2}   \d \xi
\le  \Omega_{m} \frac{e^{-\Lambda_1^2}}{\Lambda_1^{2}-m} \le \frac{\Omega_m}{m} \exp\left(- \frac{N^2}{8(1+\log m)}\right) .
\end{equation}

Set  $\Lambda_4 = +\infty$ and let us denote  for $k=1,2,3,$
\[
\mathscr{J}_0^2= \int_{\substack{ \R^{2m} \\  \|\xi\| \le \Lambda_{1}}}   \Big| F_{n,m}(\xi) - e^{-\|\xi\|^2/2} \Big|^2 \d \xi ,
\qquad
\mathscr{J}_k^2=  \int_{\substack{ \R^{2m} \\  \Lambda_k \le \|\xi\| \le \Lambda_{k+1}} }  \hspace{-.3cm} | F_{n,m}(\xi) |^2 \d\xi . 
\]
Then, by splitting the integral in \eqref{Delta2} and using the estimate \eqref{est:GaussianTail}, we obtain
\begin{equation} \label{splitting}
\Delta_{n,m}^{(2)}  \le \mathscr{J}_0 + \mathscr{J}_1 + \mathscr{J}_2 +\mathscr{J}_3 
+ \sqrt{\frac{\Omega_m}{m}} \exp\left(- \frac{N^2}{16(1+\log m)}\right) . 
\end{equation}

As we explain in Section~\ref{sect:est}, we expect that the main contribution in \eqref{splitting} comes from $\mathscr{J}_0$.
By Proposition~\ref{prop:GA}, we obtain 
\[
\mathscr{J}_0^2= 
\int_{\|\xi\| \le \Lambda_1 }\left|  e^{-\|\xi\|^2/2} - F_{n,m}(\xi) \right|^2  \d \xi
\le   \frac{\cst{8}^2 m^4  e^{2 \sqrt{2(1+\log m)}\Lambda_1}  \left(\frac{1+\log m}{2}\right)^{2N} }{\Gamma(N+1)^4} \int_{ \|\xi\| \le \Lambda_1} 
\|\xi\|^{4N}e^{-\|\xi\|^2} \d\xi .
\]
Moreover, by going to polar coordinates, we have 
\[ \begin{aligned}
\int_{ \|\xi\| \le \Lambda_1} \|\xi\|^{4N}e^{-\|\xi\|^2} \d\xi  & = m \Omega_m \int_0^{\Lambda_1^2} u^{2N +m -1} e^{-u} \d u \\
&\le  m \Omega_m \Gamma(2N+m) . 
\end{aligned}\]
For  $3 \le m \le N$,  we also have
\[ \begin{aligned}
\frac{\Gamma(2N+m)}{\Gamma(N+1)^2}& \le  e^{-m} \frac{(2N +m)^{2N+m}}{\sqrt{2\pi} N^{2N+1}} =  \frac{4^N}{ \sqrt{2\pi}} N^{m-1} \left( 1+ \frac{m}{2N} \right)^{2N+m} \left(\frac 2e \right)^m  \\
&\le \frac{4^N 2^m e^{\frac{m^2}{2N}}}{\sqrt{2\pi}} N^{m-1} .
\end{aligned}\]
Here we used the upper--bound $\Gamma(k) \le \sqrt{2\pi} k^k e^{-k}$ which holds for all integer $k\ge 9$ and the lower--bound $\Gamma(N+1) \ge \sqrt{2\pi N} N^N e^{-N}$ which holds for all $N \in\N$. For the last step, we used that $(1+x)^\alpha \le e^{\alpha x}$ for any $x, \alpha \ge 0$.
By \eqref{Lambda1}, it holds that
$ 2 \sqrt{2(1+\log m)} \Lambda_1 = N$, so we obtain
\begin{equation} \label{bound0}
\begin{aligned}
\mathscr{J}_0^2
& \le \frac{\cst{8}^2}{\sqrt{2\pi}} m^5 \Omega_m 2^{m}  \frac{e^{N+\frac{m^2}{2N}} \left(1+\log m\right)^{2N} N^{m-1}  }{\Gamma(N+1)^2} \\
& =   \big( \cst{3} \sqrt{\Omega_m} N^{\frac m2}  \Tf{0} \big)^2   ,
\end{aligned}
\end{equation}
according to formula \eqref{Theta0} and \eqref{c1}. 

\medskip

In the rest of the proof, we give bounds for the integrals  $\mathscr{J}_k$ for $k=1,2,3$. First, using the tail bound from Proposition~\ref{prop:Had}, we have 
\begin{equation*} 
\mathscr{J}_3^2 =  \int_{\|\xi\| \ge \Lambda_3} \big| F_{n,m}(\xi)\big|^2 \d\xi   \le 
\Upsilon_3(m)^{N/2}  \cst{15}^{2n} N^{N/4} \int_{\|\xi\| \ge \Lambda_3} \|\xi\|^{-N/2}  \d\xi . 
\end{equation*}
Hence, since we assume that $N > 4m$, the previous integral is finite and we obtain
\begin{equation} \label{bound3} 
\mathscr{J}_3^2 \le  \frac{2m \Omega_m}{N/2-2m}\frac{\Upsilon_3^{N/2}  \cst{15}^{2n} N^{N/4}}{\Lambda_3^{N/2-2m}} . 
\end{equation}
Second, by using the estimate \eqref{bd1} from Proposition~\ref{prop:T0}, we also have
\begin{equation} \label{bound2}
\mathscr{J}_2^2=  \int_{  \Lambda_2 \le \|\xi\| \le \Lambda_{3} }  \hspace{-.3cm} | F_{n,m}(\xi) |^2 \d\xi 
\le \Omega_m \Lambda_3^{2m} \exp\bigg( 2 \cst{9} -  \frac{2\cst{1}N^2} {1+\log m} \bigg) 
\end{equation}

Hence, by combining the estimates \eqref{bound3} and \eqref{bound2}, this implies that
\begin{equation} \label{bound4}
\mathscr{J}_2+\mathscr{J}_3 \le \sqrt{\Omega_m}  \Lambda_3^{m}  \left(   \exp\bigg( \cst{9}-  \frac{\cst{1} N^2 }{1+\log m} \bigg)  +   \frac{N^{N/8}\cst{15}^n \Upsilon_3^{N/4} }{\sqrt{N/4m-1}\Lambda_3^{N/4}} \right) .
\end{equation}
Our choice of $\Lambda_3$ consists in optimizing\footnote{If $\alpha>m$, the minimum of the function $\Lambda^m \epsilon + C\Lambda^{m-\alpha}$ over all $\Lambda>0$ is attained when $\Lambda^\alpha = (\frac{\alpha}{m}-1) \epsilon^{-1}C$ and equals to $\frac{\epsilon \Lambda^m}{1-m/\alpha}$.} the  RHS of  \eqref{bound4}. 
Namely, by choosing $\Lambda_3$ according to \eqref{Lambda3}, we obtain for all $N > 4m$, 
\[ \begin{aligned}
\mathscr{J}_2+\mathscr{J}_3 & \le \frac{ \sqrt{\Omega_m}}{1-4m/N}  \Lambda_3^{m}  \exp\bigg( \cst{9}-  \frac{\cst{1}N^2 }{1+\log m} \bigg)  \\
& \le \left(\frac{N}{4m} \right)^{\frac{2m}{N}} \frac{e^{\cst{9}(1-\frac{4m}{N})} \sqrt{\Omega_m}}{(1-4m/N)^{1-\frac{2m}{N}}} N^{\frac m2} \cst{15}^{4m^2}  \Upsilon_3^m  \exp\bigg( -  \frac{\cst{1}N (N-4m) }{1+\log m} \bigg)  .
\end{aligned}\]

Then, by using the estimate \eqref{ups3} and that $ \left(\frac{N}{4m} \right)^{\frac{2m}{N}}  \le \sqrt{e}$,  this implies that
\begin{equation*}
\mathscr{J}_2+\mathscr{J}_3 \le  \frac{e^{3/2}e^{\cst{9}(1-\frac{4m}{N})} \sqrt{\Omega_m}}{(1-4m/N)^{1-\frac{2m}{N}}} N^{\frac m2}   \cst{7}^m \cst{15}^{4m^2} m^{\frac{5m}{2}}   \exp\bigg( -  \frac{\cst{1}N (N-4m) }{1+\log m} \bigg)  .
\end{equation*}
According to the notation \eqref{c1}, \eqref{ups} and \eqref{c2}, since $\cst{6} = 4 \log \cst{15}$,  we have
\[
e^{\ups{1}(m)} =  e^{3/2} \cst{7}^m \cst{15}^{4m^2} m^{\frac{5m}{2}}   . 
\]
Hence, according to formula \eqref{Theta1}, we have shown that for all $N > 4m$, 
\begin{equation} \label{bound5}
\begin{aligned}
\mathscr{J}_2+\mathscr{J}_3 
&\le  \frac{e^{\cst{9}(1-\frac{4m}{N})} \sqrt{\Omega_m}}{(1-4m/N)^{1-\frac{2m}{N}}} N^{\frac m2}    \exp\bigg( \ups{1}(m)-  \frac{\cst{1}N (N-4m) }{1+\log m} \bigg)   \\
& = \cst{3} \sqrt{\Omega_m} N^{\frac m2}   \Tf{1} . 
\end{aligned}
\end{equation}

Third, by using the estimate \eqref{bd2} from Proposition~\ref{prop:T0}, we also have the bound 
\begin{equation*} 
\mathscr{J}_1^2=  \int_{  \Lambda_1 \le \|\xi\| \le \Lambda_{2} }  \hspace{-.3cm} | F_{n,m}(\xi) |^2 \d\xi 
\le \Omega_m \Lambda_2^{2m}   \exp\bigg(2 \cst{9}-   \frac{2\cst{2}N^2}{\sqrt{m+1} (1+\log m)^{3/4}}  \bigg) 
\end{equation*}
and according to \eqref{Lambda2}, 
\[
\Lambda_2^m \le \sqrt{e}  \frac{(8\cst{0})^{-m} N^m m^{m/2}}{(1+\log m)^{3m/4}} ,
\]
where used that by \eqref{c1}, $0<\cst{10}(m)<1$, $\cst{11}(m) \ge 1$ and $(m+1)^m \le \sqrt{e} m^{m/2}$ for all $m\ge 3$.
According to \eqref{ups}, this shows that 
$\Lambda_2^m \le e^{\ups{2}(m)}$, so that for all $m\ge 3$, 
\begin{equation}\label{bound1}
\begin{aligned}
\mathscr{J}_1  &\le e^\cst{9}\sqrt{\Omega_m} N^m    \exp\bigg(  \ups{2}(m) -   \frac{\cst{2} N^2}{\sqrt{m+1} (1+\log m)^{3/4}}  \bigg) \\ 
&=  \cst{3}\sqrt{\Omega_m} N^{\frac m2}  \Tf{2} ,
\end{aligned}
\end{equation}
according to formula \eqref{Theta2}. 

\medskip

Finally, by collecting the estimates \eqref{bound0}, \eqref{bound5} and \eqref{bound1}, we deduce from the decomposition \eqref{splitting} that for any  $m\ge 3$ and $N>4m$,
\[
\Delta_{n,m}^{(2)} \le   \cst{3} \sqrt{\Omega_m} N^{\frac m2} \bigg(  \Tf{0} 
+   \Tf{1}  +   \Tf{2} 
+ \frac{\cst{3}^{-1}}{\sqrt m}  N^{-\frac m2} \exp\left(- \frac{N^2}{16(1+\log m)}\right) \bigg)  . 
\]
After replacing the last term by $\Tf{3}$ according to \eqref{Theta0},
this completes the proof. 
\hfill$\square$


\subsection{Proof of  Theorem~\ref{thm:main}}

First, the estimate \eqref{main} follows directly from Proposition~\ref{prop:Delta2} and the fact that $\cst{3} \le 8$. 
Then, in order to prove the estimate \eqref{Delta1}, we need the following Gaussian tail--bound which is a straightforward consequence of Lemma~\ref{lem:Laplace}. 

\begin{lemma}[Large deviation estimates] \label{lem:LD}
For any $L>0$, let $\square_L = [-\tfrac L2 , \tfrac L2]^{2m}$. 
Then, we have for any $n,m \in \Z_+$, 
\[
\P_n[ \mathbf{X} \notin \square_L] \le 4m e^{-L^2/8}
\]
and 
\[
\int_{\R^{2m} \setminus \square_L}  \frac{e^{-\|x\|^2/2}}{(2\pi)^m}  \d x  \le \frac{8m }{\sqrt{2\pi} L}  e^{-L^2/8} . 
\]
\end{lemma}

\begin{proof}
For any $k\ge 1$, it follows from Lemma~\ref{lem:Laplace} that for any $t\in\R$, 
\[
\E_n[e^{t \mathrm{X}_{2k}}] , \E_n[e^{t \mathrm{X}_{2k-1}}]  \le  e^{t^2/2} .
\]
By Markov inequality, this implies that for any $k\ge 1$ and $t>0$, 
\[
\P_n[ |\mathrm{X}_{k}| \ge L] \le 2 e^{-tL + t^2/2} .
\]
Choosing $t=L$, we obtain
\[
\P_n[ |\mathrm{X}_{k}| \ge L] \le 2 e^{-L^2/2} .
\]
Hence, by a union bound, we obtain
\[
\P_n[(\mathrm{X}_1,\dots , \mathrm{X}_{2m}) \notin \square_L] \le \sum_{k\le 2m} \P_n[ |\mathrm{X}_{k}| \ge L/2]
\le 4 m  e^{-L^2/8} . 
\]
By a similar union bound, an analogous estimate holds in the Gaussian case.
\end{proof}

Recall the definitions \eqref{Delta} and let us split  
\begin{align*}
\Delta_{n,m}^{(1)}  &
=  \left( \int_{\square_L}  +  \int_{\R^{2m} \setminus \square_L} \right) \bigg| \p_{n,m}(x) - \frac{e^{-\|x\|^2/2}}{(2\pi)^m} \bigg|  \d x  \\
& \le L^m \Delta_{n,m}^{(2)} +   \int_{\R^{2m} \setminus \square_L}  \bigg(  \p_{n,m}(x) + \frac{e^{-\|x\|^2/2}}{(2\pi)^m}  \bigg)  \d x  ,
\end{align*}
where we used the Cauchy--Schwartz inequality to bound the first integral. By Lemma~\ref{lem:LD}, this implies that for any $L\ge 2\sqrt{3}$, 
\begin{equation} \label{Dest1}
\Delta_{n,m}^{(1)} \le  L^m \Delta_{n,m}^{(2)}  +  5 m e^{-L^2/8} , 
\end{equation}
We choose the parameter $L$ which minimizes the RHS of \eqref{Dest1}, that is the (unique) solution of the equation:
\begin{equation} \label{Delta3}
\Delta_{n,m}^{(2)}  = \tfrac 54  L^{2-m}  e^{-L^2/8}   .
\end{equation}
Since $m\ge 3$, the function  $L \mapsto L^{2-m}  e^{-L^2/8} $ is decreasing and it is bounded from below by 
$ 2^{2-m}m^{1-\frac m2} e^{- \frac m2}$ for $L\le 2\sqrt{m}$,  under the assumption that $  \Delta_{n,m}^{(2)}  \le 5 \cdot 2^{-m}m^{1-\frac m2} e^{- \frac m2}$, the solution of the equation \eqref{Delta3} satisfies
\begin{equation} \label{Lconditions}
2\sqrt{m} \le L \le \sqrt{8 \log  \Delta_{n,m}^{(2)-1}} .  
\end{equation}
Hence, by \eqref{Dest1}, this implies that 
\[
\Delta_{n,m}^{(1)} \le L^m \bigg(1 +  \frac{4 m}{L^2} \bigg)  \Delta_{n,m}^{(2)}  .
\]
Finally, using the conditions \eqref{Lconditions} for $L$, we conclude that 
\[
\Delta_{n,m}^{(1)}  \le 2  \big(8 \log  \Delta_{n,m}^{(2)-1} \big)^{\frac m2}\Delta_{n,m}^{(2)}  . 
\]
This completes the proof. 
\hfill$\square$


\subsection{Proof of Theorem~\ref{thm:TV},  Proposition~\ref{cor:TV} and Proposition~\ref{cor:TVuniform}} \label{sect:proofTV}

Throughout the proof, we fix an integer $M\ge 3$. 
By \eqref{Gamma}, we have $\sqrt{\Omega_{m}} \le  \frac{m^{- \frac m2} (e\pi)^{\frac m2}}{(2\pi)^{1/4}}$ and it follows from the estimate \eqref{main} that   the condition  $  \Delta_{n,m}^{(2)}  \le 5 \cdot 2^{-m}m^{1-\frac m2} e^{- \frac m2}$ from Theorem~\ref{thm:main} is satisfied if
\begin{equation} \label{condition:Theta}
\Theta_{n,m} \le  \tfrac{m}{2} \cst{16}^{-m} N^{-\frac m2} . 
\end{equation}
Then we immediately deduce from the estimates of  Proposition~\ref{cor:Theta} that for all $m\ge M$ and $N\ge \cst{}(M) m \sqrt{1+\log m}$, 
\[
\Theta_{n,m} \le   (1+\eps{}) \Tf{0} \le   N^{-\frac m2}   \frac{\exp\big(- 12 m \big( \log m  - 0.26 \big) \big)}{ \cst{16}^m} .
\]
This shows that under the hypothesis of Theorem~\ref{thm:TV}, the condition \eqref{condition:Theta} holds. 
Accordingly, by \eqref{Delta1}, we obtain
\[ \begin{aligned}
\Delta_{n,m}^{(1)}  & \le 2  \big(8 \log  \Delta_{n,m}^{(2)-1} \big)^{\frac m2}\Delta_{n,m}^{(2)}  \\
&  \le 2\cst{3}  \sqrt{\Omega_m}  \big(8 N \log \big( \cst{3}  \sqrt{\Omega_m} N^{\frac m2}  \Theta_{N,m}  \big)^{-1} \big)^{\frac m2}\Theta_{N,m} ,
\end{aligned}\]
where we used that the function $x \mapsto  x (\log x^{-1})^{\frac m2}$ is non--decreasing for $x\in[0,e^{-\frac m2}]$ as well as  Proposition~\ref{prop:Delta2} to get the second bound. 
By \eqref{Gamma}, we also have $\Omega_m N^m \ge  \frac{e^m \pi^m (N/m)^m}{3 \sqrt{m}} $, so that according to formula \eqref{Theta0}, we obtain the (crude) lower--bound
\[
\cst{3}  \sqrt{\Omega_m} N^{\frac m2}  \Tf{0}  \ge N^{-N} .
\]
This implies that  if $m\ge M$ and $N\ge \cst{}(M) m \sqrt{1+\log m}$,  then
\begin{equation} \label{Delta4}
\Delta_{n,m}^{(1)}  \le  2\cst{3}  \sqrt{\Omega_m}  \big(N\sqrt{8\log N} \big)^{m} \Theta_{N,m}  . 
\end{equation}

Using Corollary~\ref{cor:Theta} once more, we obtain
\[
\Delta_{n,m}^{(1)}  \le  2\cst{3} (1+\eps{})  \sqrt{\Omega_m}  \big(N\sqrt{8\log N} \big)^{m} \Tf{0}  . 
\]
Since $2\cst{3} (1+\eps{}) \le 16$,  by \eqref{Theta0} and \eqref{TV},  we conclude that   for all $m\ge M$ which satisfies the condition $N\ge \cst{}(M) m \sqrt{1+\log m}$,
\begin{equation} \label{Delta5}
{\rm d_{TV}}(\X, \G) \le \Delta_{n,m}^{(1)}  \le 16  \sqrt{\Omega_m}   m^{3}  4^me^{\frac{m^2}{4N}}   \big(N\sqrt{\log N} \big)^{m}\frac{ e^{\frac{N}{2}} (1+\log m)^{N}  }{\sqrt{n}\ \Gamma(N+1)}     . 
\end{equation}
This completes the proof of Theorem~\ref{thm:TV}.
\hfill$\square$

\medskip

Now, let us choose $m = \lfloor n^\alpha \rfloor$ with $0<\alpha<1/2$ and let us assume that $m\ge 17$. 
By \eqref{Delta5}, this implies that for all integer $n\in\N$ such that $n^{1-2\alpha}\ge 20.4 \sqrt{\log n}$, 
\[
\Delta_{n,m}^{(1)}  \le 16 \sqrt{\Omega_m} n^{3\alpha-1} 4^me^{\frac{n^{3\alpha-1}}{4}}  \big(N\sqrt{\log N} \big)^{n^{\alpha}} \frac{e^{\frac{N}2}(\log N)^{N}  }{\Gamma(N+1)}
\]
where we used that $N \ge e^2 m$ and that $\cst{}(17) = 28.8 \le 20.4\sqrt{2}$ -- see the Table \eqref{table:cM}. First, observe that  by \eqref{Gamma}, it holds for all $m\in\N$, 
\begin{equation} \label{estimateOmega}
\sqrt{\Omega_m} 4^m\le  \frac{(4\sqrt{\pi e})^m}{(2\pi)^{\frac 14}\sqrt{m}} m^{- \frac m2} \le \frac{e^{8\pi}}{\sqrt{m\sqrt{2\pi}}} ,
\end{equation}
where we used that the $\displaystyle \max_{m\ge 0}\big\{ (16\pi e)^m m^{-m} \big\}  =  \exp\big(16 \pi \big)$.
Second, let us observe that the function 
$\frac{e^{\frac{N}2}(\log N)^{N}  }{\Gamma(N+1)}$ is decreasing for $N\ge e^2$ so that by \eqref{Gamma}, 
\[
\frac{e^{\frac{N}2}(\log N)^{N}  }{\Gamma(N+1)} \le  \frac{n^{\frac{\alpha-1}{2}}}{\sqrt{2\pi}}
\exp\bigg( - n^{1-\alpha} \log( n^{1-\alpha})
\bigg( 1-\frac{\log(\log\sqrt{n}) +3/2}{\log(n^{1-\alpha})} \bigg) \bigg) .
\]
Since $1/\sqrt{m} \le n^{-\frac\alpha2} \sqrt{18/17}$ for  $n\ge 18^{\alpha^{-1}}$,  these estimates imply that 
\[
\Delta_{n,m}^{(1)}  \le   C  n^{3\alpha-\frac32} e^{\frac{n^{3\alpha-1}}{4}}  \big(N\sqrt{\log N} \big)^{n^{\alpha}}\exp\bigg( - n^{1-\alpha} \log( n^{1-\alpha})
\bigg( 1-\frac{\log(\log\sqrt{n}) +3/2}{\log(n^{1-\alpha})} \bigg) \bigg) ,
\]
with $C= \frac{18 e^{8\pi}}{(2\pi)^{\frac 34}}$. 
Now, let us also observe that  $N \le e^{0.0572}n^{1-\alpha}$ for  $n\ge 18^{\alpha^{-1}}$,  so that 
\[
\big(N\sqrt{\log N} \big)^{n^{\alpha}} \le  \exp\bigg( n^{\alpha} \log( n^{1-\alpha})
\bigg( 1+\frac{\log \log n +0.1144}{2\log(n^{1-\alpha})} \bigg) \bigg) .
\]
This shows that if  $n\ge 18^{\alpha^{-1}}$ (so that $m\ge 17$)  and  $n^{1-2\alpha}\ge 20.4 \sqrt{\log n}$, 
\[ \begin{aligned}
\Delta_{n,m}^{(1)}  \le C  n^{3\alpha-\frac32}   \exp\Big( - (1-\epsilon_n) n^{1-\alpha} \log( n^{1-\alpha}) \Big)
\end{aligned}\]
with
\begin{align}\label{TV:error}
\epsilon_n 
& : = \frac{\log \log(\sqrt{n}) +3/2}{\log(n^{1-\alpha})} +n^{-(1-2\alpha)}
\bigg( 1+\frac{\log \log n +0.1144}{2\log(n^{1-\alpha})} \bigg) +\frac{n^{-2(1-2\alpha)}}{4\log(n^{1-\alpha})} \\
& \label{TV:error2} 
\le \frac{2(\log \log n +0.8069)}{\log n}  +   \frac{0.0649}{\sqrt{\log n}}  + \frac{0.0012}{(\log n)^2} , 
\end{align}  
where we have used that $n^{1-2\alpha}\ge 20.4 \sqrt{\log n}$, $\alpha\le 1/2$
and the numerical bound $1+\frac{\log \log n +0.1144}{\log n} \le  20.4 \cdot 0.0649$ for all $n\ge 18^2$ to obtain the estimate \eqref{TV:error2}. 
We also deduce from \eqref{TV:error2} that $\epsilon_n \le 1- 87 \cdot 10^{-3}$ for all $n\ge 18^2$. 
Since ${\rm d_{TV}}(\X, \G)\le \Delta_{n,m}^{(1)}$, this completes the proof of  Proposition~\ref{cor:TV}. \hfill$\square$

\medskip

We now turn to the proof of Proposition~\ref{cor:TVuniform}. 
Let us choose $m = \left\lfloor \sqrt{\frac{n}{41.5\sqrt{\log n}}} \right\rfloor$ and suppose that $n\ge 4322$ so that $m\ge 6$ and we can use the estimate \eqref{Delta5} -- we have $\cst{}(6) = 58.66 \le 41.5\sqrt{2}$ according to the Table \eqref{table:cM}. 
As  $N = \frac nm \ge e^2 m$, this implies that 
\[
\Delta_{n,m}^{(1)} \le 16 \sqrt{\Omega_m}  m^{\frac 72}  4^{m} e^{\frac{m^3}{4n}}  \big(N\sqrt{\log N} \big)^{m} \frac{ e^{\frac 32N}(\log N)^{N}  }{\sqrt{2\pi} n N^{N}}  .
\]
Moreover, we verify that  as $N \ge  \sqrt{41.5n\sqrt{\log n}}$, 
\[
(\log N)^{N}  N^{m-N} e^{\frac 32 N} \le  \exp \bigg( - \frac{\sqrt{41.5}}{2} \sqrt{n}(\log n)^{5/4} \bigg(1- \frac{1}{41.5\sqrt{\log n}} - \frac{2\log(\log\sqrt{n}) +3}{\log n}\bigg) \bigg)
\]
and using the estimate \eqref{estimateOmega}, this shows that 
\[
\Delta_{n,m}^{(1)} \le  \frac{16 e^{8\pi}m^3}{(2\pi)^{\frac 34} n}  e^{\frac{m^3}{4n}}  (\log N)^{\frac m2} \exp \bigg( - \frac{\sqrt{41.5}}{2} \sqrt{n}(\log n)^{5/4} \bigg(1- \frac{1}{41.5\sqrt{\log n}} - \frac{2\log(\log\sqrt{n})+3}{\log n}\bigg) \bigg) . 
\]
Moreover, since $(\log N)^{\frac m2}  e^{\frac{m^3}{4n}}   \le \exp\Big( \frac{\sqrt{n}}{4(41.5\sqrt{\log n})^{3/2}} +\tfrac 12 \sqrt{\frac{n}{41.5\sqrt{\log n}}}  \log\log n \Big) $, we obtain
\[
\Delta_{n,m}^{(1)} \le  \frac{16 e^{8\pi} \sqrt{n}}{(2\pi \log n)^{\frac 34} (41.5)^{\frac 32}} \exp \bigg( - \frac{\sqrt{41.5}}{2} \sqrt{n}(\log n)^{5/4} \big(1- \epsilon_n \big) \bigg) . 
\] 
where 
\[
\epsilon_n = \frac{1}{41.5\sqrt{\log n}} + \frac{3-2\log 2+2\log\log n}{\log n} + \frac{\log \log n}{41.5(\log n)^{\frac 32}} + \frac{1/2}{( 41.5\log n)^2}
\]
We verify numerically that $\epsilon_n \le 0.711$ for all $n\ge 4322$. 
In particular, this implies that 
\[
\Delta_{n,m}^{(1)} \le \sqrt{n}\exp \big( 19.4 - 0.93\sqrt{n}(\log n)^{5/4} \big) . 
\] 
Since $ \Delta_{n,m}^{(1)}$ is non--decreasing in $m\in\N$,  this completes the proof of Proposition~\ref{cor:TVuniform}. \hfill$\square$

\section{Gaussian approximation: Proof of Proposition~\ref{prop:GA} } \label{sect:GA}

Recall that $F_{n,m}$ denotes the characteristic function of the random vector $\X$ and that it is given by formula \eqref{BOg}. 
In particular, it holds for any $\xi \in \R^{2m}$, 
\begin{equation} \label{det1}
\big|  e^{-\|\xi\|^2/2} - F_{n,m}(\xi) \big|^2 = e^{-\|\xi \|^2} \big| 1- \det[\operatorname{I} -  K_{\i\g}  Q_n] \big|^2 , 
\end{equation}
where $\g$  is a trigonometric polynomial \eqref{g_function}, $Q_n$ is the orthogonal projection with kernel $\operatorname{span}(e_1, \dots, e_{n-1})$ and according to formula \eqref{BOK}, 
\begin{equation} \label{kernel}
K_{\i\g}= H_+(e^ {2 \Im \g^+}) H_-(e^ {2 \Im \g^- }) . 
\end{equation}
Recall that the operator $K_{\i\g}$ is trace--class, but observe that it is not self--adjoint since  by \eqref{HO},  $K_{\i\g}^*= H_+(e^ {2 \Im \g^-}) H_-(e^ {2 \Im \g^+ })$ with $\Im \g^- = - \Im \g^+$ because the function $\g$ is real--valued. 
As we explained in Section~\ref{sect:est},  in order to prove  Proposition~\ref{prop:GA}, we provide estimates for the Fredholm determinant on the RHS of \eqref{det1} in the regime where $\|\xi\| \ll N$ in order to guarantee that  the Schatten norm $\| K_{\i\g}  Q_n\|_{\mathscr{J}_1}$  remains small. 

\medskip

The first step of the proof  consists in obtaining a priori estimates on Fourier coefficients of the functions  $e^{2\Im \g^\pm}$. 

\begin{lemma} \label{Lem:F_coeff}
Fix $m \in \N$ and $\xi\in \R^{2m}$.  Let $\rho  = \sqrt{\frac{1+\log m}{2}} \|\xi\|$. 
We have for all integers  $k >  2 m \rho$, 
\[
\Big| \big(\widehat{e^{ \pm 2\Im \g^+}}\big)_{k}  \Big|  \le 2 e^{\rho} \frac{\rho^{ \lceil k/m \rceil}}{\lceil k/m \rceil!} . 
\]
\end{lemma}

\begin{proof}
Let us define $\phi_M(w) = \sum_{k=0}^M \frac{w^k}{k!}$ for $M \ge 1$. 
Since  $\g^+(\theta) = \sum_{k=1}^m \frac{\zeta_k}{\sqrt{2k}} e^{\i k \theta}$ and $\g^- = \overline{\g^+}$, we have for all integers
$k > Mm$, 
\[
\int_\T  \phi_M(-\i \g^+(\theta))  e^{\i \g^-(\theta)-\i k \theta} \frac{\d\theta}{2\pi} =0 . 
\]
This implies that any $k > Mm$, 
\begin{align} \notag
\big| (\widehat{e^{2\Im \g^+}})_k \big|  &=  \left| \int_\T e^{-\i \g^+(\theta) +\i \g^-(\theta)-\i k \theta} \frac{\d\theta}{2\pi} \right| \\
& \label{estim1}
\le  \int_\T  \left|  e^{-\i \g^+(\theta)} -  \phi_M(-\i \g^+(\theta))  \right|  e^{ - \Im \g^-(\theta)}  \frac{\d\theta}{2\pi} . 
\end{align}
Now, let us observe that for any  $|w| \le M/2$,
\begin{align} \notag
\big| e^w -\phi_M(w) \big|  & \le  \frac{|w|^{M+1}}{(M+1)!}  \sum_{j\ge 0} \left(\frac{|w|}{M+2}\right)^j  \\
& \label{estim2}
\le 2\frac{|w|^{M+1}}{(M+1)!} . 
\end{align}
Moreover, by \eqref{g_function} and since $ \sum_{k=1}^m |\zeta_k|^2 = \|\xi\|^2$, we also have
\begin{equation} \label{estim3}
\| \g^+\|_\infty  \le  \sum_{k=1}^m \frac{|\zeta_k|}{\sqrt{2k}} \le \rho=  \sqrt{\frac{1+\log m}{2}} \|\xi \|  ,
\end{equation}
where we used that $\sum_{k=1}^m k^{-1} \le 1+\log m$ for any $m\ge 1$ and the Cauchy--Schwartz inequality. 
Then, using the estimates \eqref{estim1}, \eqref{estim2} and \eqref{estim3}, we obtain if both $M \ge 2\rho$ and $k>M m$,
\[
\big| (\widehat{e^{2\Im \g^+}})_k \big|   \le  2 e^{\rho} \frac{\rho^{M+1}}{(M+1)!} . 
\]
By choosing $M = \lfloor k/m \rfloor $, this implies the claim. Indeed, by the same argument, we obtain the same bound for $\big| (\widehat{e^{-2\Im \g^+}})_k \big|$. 
\end{proof}

Now, let us use these estimates to bound the Schatten norm $\| K_{\i\g}  Q_n\|_{\mathscr{J}_1}$. 

\begin{lemma} \label{Lem:J1norm}
Fix $m \in \N$, $\xi\in \R^{2m}$ and  let $\rho  = \sqrt{\frac{1+\log m}{2}}\|\xi\|$. If we assume that  $N= \frac nm  \ge \cst{*}^{-1}\rho$ with $\cst{*}< \frac12$,  then
\[
\| K_{\i\g}  Q_n\|_{\mathscr{J}_1} \le \frac{4 m^2 e^{2\rho}}{(1-\cst{*}^2)^2}  \frac{ \rho^{2 N}}{\Gamma(N+1)^2} .
\]
\end{lemma}

\begin{proof}
Let us recall that the operators $ H_\pm(e^ {2 \Im \g^\pm})$ are Hilbert--Schmidt and that by \eqref{HO}, we  have
\[
\| Q_n H_{\pm}(e^ {2 \Im \g^\pm})\|_{\mathscr{J}_2}^2 \le \sum_{k\ge n} (k-n+1) \Big|  \big(\widehat{e^{2\Im \g^\pm}}\big)_{k} \Big|^2   . 
\]
Moreover, by formula \eqref{kernel} and the Cauchy--Schwartz inequality (for the Schatten norms), since $Q_n$ is a projection, we have
\[
\| K_{\i\g}  Q_n\|_{\mathscr{J}_1} \le  \| H_+(e^ {2 \Im \g^+}) Q_n \|_{\mathscr{J}_2} \| H_-(e^ {2 \Im \g^-}) Q_n\|_{\mathscr{J}_2} . 
\]

Using the estimates form Lemma~\ref{Lem:F_coeff}, this implies that if the dimension $n >  2m \rho$, then
\begin{equation} \label{J1est} 
\| K_{\i\g}  Q_n\|_{\mathscr{J}_1} \le   4 e^{2\rho} \sum_{k\ge n}  (k-n+1)   \frac{\rho^{ 2 \lceil k/m \rceil}}{ (\lceil k/m \rceil! )^2} . 
\end{equation}
Under the condition $N = \frac nm \ge \cst{*}^{-1} \rho$, since  $j! \ge \Gamma(N+1) N^{j-N}$ for all $j \ge N$,  we obtain
\[ \begin{aligned}
\sum_{k\ge n} (k-n+1)   \frac{\rho^{ 2 \lceil k/m \rceil}}{( \lceil k/m \rceil! )^2} 
&\le m^2  \sum_{j \ge N} (j+1-N) \frac{\rho^{2j}}{(j!)^2} \\
&\le  m^2 \frac{ \rho^{2 N}}{\Gamma(N+1)^2} \sum_{ j \ge 0} (j+1) \left(\frac{\rho}{N}\right)^{2j} \\
&\le \left(1-\cst{*}^{2}\right)^{-2}  m^2 \frac{ \rho^{2 N}}{\Gamma(N+1)^2} . 
\end{aligned}\]
Note that for the last bound, it suffices that $\cst{*}<1$. However, we impose that  $\cst{*}< \frac12$ to guarantee that   $n >  2m \rho$. Then, by combining the previous estimate with \eqref{J1est}, this completes the proof.
\end{proof}

We are now ready to finish the proof of Proposition~\ref{prop:GA}. 
First, let us observe that by Lemma~\ref{Lem:J1norm} and using  formula \eqref{Gamma} for the $\Gamma$ function, 
we obtain that if $N \ge \cst{*}^{-1}\rho$, 
\begin{equation} \label{J1norm}
\| K_{\i\g}  Q_n\|_{\mathscr{J}_1} 
\le \frac{2/\pi}{ (1-\cst{*}^2)^2} m^2 e^{2\rho}  \frac{ (\rho e)^{2 N}}{N^{2N+1}} 
\le  \frac{2/\pi}{ (1-\cst{*}^2)^2}  \frac{m^2}{N}  \big( \cst{*}e^{\cst{*}+1} \big)^{2N} .
\end{equation}
If we choose $\cst{*} = 1/4$, then $\cst{*}e^{\cst{*}+1} \le 0.873$ so that the RHS of \eqref{J1norm} is very small for large $N$. Actually, in the regime where  $N > 4m$ (in particular when $N \ge 13$), this implies that  
\[
\| K_{\i\g}  Q_n\|_{\mathscr{J}_1} \le  \frac{32}{225\cdot\pi}    N \big( \cst{*}e^{\cst{*}+1} \big)^{2N} \le  \frac{416}{225\cdot\pi }   (0.873)^{26}  \le \log(2.766)-1 , 
\]
where we obtained the last two bounds  numerically. 
Hence, using the inequality \eqref{J1inequality} from \cite[Theorem 3.4]{Simon05},  we deduce from Lemma~\ref{Lem:J1norm} with $\cst{*}= 1/4$ and the previous estimate  that if  $N \ge 4(\rho \vee m)$, 
\[
\big|1- \det[\operatorname{I} -  K_{\i\g}  Q_n]   \big|^2 
\le  \| K_{\i\g}  Q_n\|_{\mathscr{J}_1}^2 e^{2(1+\| K_{\i\g}  Q_n\|_{\mathscr{J}_1})}  \le
\cst{8}^2 m^4 e^{4\rho} \frac{ \rho^{4 N}}{\Gamma(N+1)^4} . 
\]
where $\cst{8} = 2.766	 \frac{4}{(1-\cst{*}^2)^2}$ according to \eqref{c1}. 
If we combine this estimate with formula \eqref{det1} and replace $\rho  = \sqrt{(1+\log m)/2}\|\xi\|$, this  implies that for any $N\ge 4m$ and all  $\|\xi\| \le \Lambda_1 = \frac{N}{4\sqrt{(1+\log m)/2}}$, 
\[
\big|  e^{-\|\xi\|^2/2} - F_{n,m}(\xi) \big|^2 \le \cst{8}^2   m^4 e^{4\rho} \frac{ \rho^{4 N}}{\Gamma(N+1)^4} e^{-\|\xi \|^2}. 
\]
This completes the proof. 
\hfill$\square$


\section{Tail bound for large $\|\xi\|$: Proof of Proposition~\ref{prop:Had}} \label{sect:TB}  

Recall that the function $\g$ is given by \eqref{g_function} and let us observe that by choosing $h = \g'$ in Lemma~\ref{lem:estF}, we obtain the following bound. 

\begin{proposition} \label{prop:TB1}
Fix $m ,n \in \N$ and let $N= \frac nm$.  For any $\eta>0$ and any $\xi\in\R^{2m}$, we have 
\begin{equation*}
\big|  F_{n,m}(\xi)  \big| \le \exp\Big(\cst{20} \big(n+\tfrac{2}{\pi^2 }\big)\Big) \E_n \left[ e^{-  \gamma \sum_{j =1}^n  \g'(\theta_j)^2 } \right] ,
\end{equation*}
where $\gamma =    \frac{\eta}{ \sqrt{n}m(m+1)\| \xi\|} \left(1  -\frac{\eta^2\cst{21}}{n}    \right)$ and $\cst{20} = \frac{\pi^2 \eta^2}{8}$. 
\end{proposition}

In order to prove Proposition~\ref{prop:TB1}, we need the following basic estimate which is proved in the Appendix~\ref{sect:proof_unibound}. 

\begin{lemma} \label{lem:unibound}
For any $y\in[-1,1]$, $y\neq 0$ and $x\in\R$, we have
\[
1+ \bigg(\frac{\sinh(x)}{y} \bigg)^2 \le \exp \Big(\frac x y\Big)^2 . 
\]
\end{lemma}

\begin{proof}
We apply  Lemma~\ref{lem:estF} with $h = \g'$ and $ \nu =  \frac{\eta \sqrt{n}}{m(m+1)\| \xi\|}$  where   $\eta>0$. 
We obtain 
\begin{equation} \label{term0}
| F_{n,m}(\xi)|  \le \E_n\bigg[ \prod_{i<j} \bigg| \frac{ \sin\big(\frac{\theta_i - \theta_j}{2} + \i \nu \frac{ \g'(\theta_i) - \g'(\theta_j)}{2n}  \big)}{\sin\big(\frac{\theta_i - \theta_j}{2}\big)} \bigg|^2  \prod_{j =1}^n  \big| 1+ \i \tfrac \nu n \g''(\theta_j)\big| e^{- \Im g\big(\theta_j +\i \frac{\nu}{n}\g'(\theta_j)\big)}  \bigg] . 
\end{equation}
Moreover, by the Cauchy--Schwartz, we have
\begin{equation} \label{g'}
\|\g' \|_\infty  \le \sum_{k=1}^m \sqrt{2 k} |\zeta_k| \le \sqrt{2  {\textstyle \sum_{k=1}^m |\zeta_k|^2}  {\textstyle\sum_{k=1}^m k}} = \sqrt{m(m+1)} \|\xi\| . 
\end{equation}

Observe that with these choices,  we have $\frac{\nu}{n} \|\g' \|_\infty \le \frac{\eta/m}{\sqrt{n(1+1/m)}}$, so that by Taylor's theorem, since $\g$ is real--valued,  we have  for $j\in \{1, \dots, n\}$, 
\[
\left| \Im  \g\big(\theta_j +\i \frac{\nu}{n}\g'(\theta_j)\big) - \frac{\nu}{n} \g'(\theta_j)^2  \right|
\le \frac 16  \left|\frac{\nu}{n} \g'(\theta_j)  \right|^3 \sup_{|\Re z| \le \pi , |\Im z| \le \frac{\eta/m}{\sqrt{n(1+1/m)}}} \big| \g'''(z) \big| . 
\]
We also have 
\[
\big| \g'''(z) \big|  \le \sum_{|k|\le m} |\zeta_k| \frac{|k|^{5/2}}{\sqrt{2}} e^{k|\Im z|} 
\qquad\text{so that}\qquad
\sup_{|\Re z| \le \pi , |\Im z| \le \frac{\eta/m}{\sqrt{n(1+1/m)}}} \big| \g'''(z) \big| \le   6\cst{21} \big(m(m+1) \big)^{3/2}\|\xi\| ,
\] 
where $\cst{21} = \frac{\exp\big(\eta /\sqrt{n(1+1/m)}\big)}{6\sqrt3}$  
and we used that $\sum_{k=1}^m k^5 \le \frac{m^3(m+1)^3}{6}$. 
Then using the estimate \eqref{g'}, the previous bounds imply that
\[ \begin{aligned}
\left| \Im  \g\big(\theta_j +\i \frac{\nu}{n}\g'(\theta_j)\big) - \frac{\nu}{n} \g'(\theta_j)^2  \right|
& \le \cst{21}  \frac{\nu^3 m^2(m+1)^2 \|\xi\|^2}{n^3}  \g'(\theta_j)^2 \\
& = \frac{\nu}{n} \frac{\eta^2 \cst{21}  }{n} \g'(\theta_j)^2 ,
\end{aligned}  \]
where we used  our choice for $\nu$.
This shows that 
\begin{equation} \label{term1}
\prod_{j =1}^n e^{- \Im \g\big(\theta_j +\i \frac{\nu}{n}\g'(\theta_j)\big)} 
\le \exp\bigg( -  \frac{\nu}{n} \left(1  -\frac{\eta^2\cst{21}}{n}    \right)\sum_{j =1}^n  g'(\theta_j)^2  \bigg) . 
\end{equation}  

Moreover, by Lemma~\ref{lem:unibound} and since by convexity, $\sin(u/2) \ge u/\pi$ for all $u \in [0,\pi]$, we  obtain  for any $u \in [-\pi,\pi]$ and $\alpha>0$, 
\[
1+ \left( \frac{\sinh(\alpha u/2)}{\sin(u/2)} \right)^2 \le  \exp \left( \frac{\alpha u }{2\sin(u/2)} \right)^2  \le  \exp\left( \frac{\pi \alpha}{2} \right)^2 .
\]


This estimate implies that for all $i, j\in\{1, \dots, n\}$, 
\[ \begin{aligned}
\bigg| \frac{ \sin\big(\frac{\theta_i - \theta_j}{2} + \i \nu \frac{ \g'(\theta_i) - \g'(\theta_j)}{2n}  \big)}{\sin\big(\frac{\theta_i - \theta_j}{2}\big)} \bigg|^2
&= 1 + \left( \frac{\sinh\big(\nu \frac{ \g'(\theta_i) - \g'(\theta_j)}{2n}\big)}{\sin\big(\frac{\theta_i - \theta_j}{2}\big)} \right)^2 \\
&\le 1 + \left( \frac{\sinh\big(\frac{\nu \| \g''\|_\infty}{n}\frac{(\theta_i- \theta_j)}{2}\big)}{\sin\big(\frac{\theta_i - \theta_j}{2}\big)} \right)^2 \\
& \le \exp\left(\frac{\nu \pi \| \g''\|_\infty}{2n} \right)^2 .
\end{aligned}\]  
Moreover, by the Cauchy--Schwartz inequality, we have
\begin{equation*} 
\|\g''\|_\infty \le \sum_{|k| \le m} \frac{k^{3/2}}{\sqrt{2}} |\zeta_k| \le \sqrt{ \sum_{k=1}^m k^3 \sum_{|k| \le m}  |\zeta_k|^2 } = \frac{m(m+1)}{\sqrt{2}}\|\xi\| . 
\end{equation*}
Hence, this shows that 
\begin{equation} \label{term2} \begin{aligned}
\prod_{1\le i<j \le n} \bigg| \frac{ \sin\big(\frac{\theta_i - \theta_j}{2} + \i \nu \frac{ \g'(\theta_i) - \g'(\theta_j)}{2n}  \big)}{\sin\big(\frac{\theta_i - \theta_j}{2}\big)} \bigg|^2
\le \exp \bigg( \frac{\nu m (m+1)\|\xi\|}{2 \sqrt{2}/\pi} \bigg)^2
= e^{\cst{20} n} ,
\end{aligned} 
\end{equation}
where we used the definition of $\nu$ and set   $\cst{20} = \frac{\pi^2\eta^2}{8}$. 
Similarly, we have 
\begin{equation} \label{term3} \begin{aligned}
\prod_{j =1}^n  \big| 1+ \i \tfrac \nu n \g''(\theta_j)\big| 
& \le  \left( 1 + \frac{\nu^2 \| \g''\|_\infty^2}{n^2}  \right)^{n/2} \\
&\le \exp \left(  \frac{ \big(\nu m(m+1) \|\xi\| \big)^2 }{4n} \right) 
= \exp\left( \frac{2 \cst{20} }{ \pi^2 } \right) , 
\end{aligned}
\end{equation} 
where we used that $1+x \le e^x$ for all $x\in\R$ to obtain the second estimate. 
By combining the estimates \eqref{term1}, \eqref{term2},  \eqref{term3} with \eqref{term0}, we obtain that for all $\xi\in\R^{2m}$,
\begin{equation*}
\big|  F_{n,m}(\xi)  \big| \le  \E_n \bigg[ \exp\bigg( \cst{20} \left(n+\tfrac{2}{\pi^2 }\right) -  \gamma \sum_{j =1}^n  \g'(\theta_j)^2 \bigg) \bigg] ,
\end{equation*}
where $\gamma = \frac{\nu}{n} \left(1  -\frac{\eta^2\cst{21}}{n}    \right) =    \frac{\eta}{ \sqrt{n}m(m+1)\| \xi\|} \left(1- \eta^2 \frac{\exp\big(\frac{\eta}{\sqrt{n(1+1/m)}}\big)}{6\sqrt3 n} \right)$. 
\end{proof}

Thus, in order to estimate $\big|  F_{n,m}(\xi)  \big|$ using  Proposition~\ref{prop:TB1}, we need a bound for  $\E_n \big[ e^{-  \gamma \sum_{j =1}^n  \g'(\theta_j)^2 } \big]$. 
Let us point out that in the regime where $\|\xi\|$ is large compared with $N$, 
we cannot use the bound from Lemma~\ref{lem:Laplace} to estimate this quantity. 
Indeed, we have $\|\g'\|_{L^2}^2 \ge \|\xi\|^2$, while our basic estimate for $\A(\g'^2)$ is of the form 
$\A(\g'^2) \le \cst{} m^5 \|\xi\|^4$ for a numerical constant $\cst{}>0$. 
Then, by optimizing over all $\gamma>0$, we would obtain
\[
\E_n \big[ e^{-  \gamma \sum_{j =1}^n  \g'(\theta_j)^2 } \big] \le
\exp \big(- \gamma n \|\g'\|_{L^2}^2 + \gamma^2 \A(\g'^2)\big) 
\le \exp\left( -  \frac{N^2}{4\cst{}m^3} \right).
\]
This estimate is similar to those from Proposition~\ref{prop:T0} but it not as good for large $m \in\N$.
More importantly,  it does not yield any decay as $\|\xi\| \to+\infty$. So, instead of Lemma~\ref{lem:Laplace}, 
we will use the bound \eqref{F2} which follows from the next Lemma.

\begin{lemma}  \label{lem:TB}
For any function $f:\T \to \R$ such that $e^{-f}$ is integrable, we have for any $n\ge 2$, 
\begin{equation} \label{term5}
\E_n \left[ e^{- \sum_{j =1}^n f(\theta_j) } \right] 
\le  \frac{e^n}{\sqrt{2\pi n}}  \left( \int_\T  e^{- f(\theta)} \frac{d\theta}{2\pi} \right)^n . 
\end{equation}
\end{lemma}

The proof of Lemma~\ref{lem:TB} is given in the appendix (Section~\ref{sect:proof_TB}) and it relies on the fact that the configurations which minimize the \emph{energy} associated with the probability measure $\P_n$ are uniformly distributed on $\T$ (like the vertices of a regular $n$-gon) so that we known explicitly the minimal energy as well as the partition function.

\medskip

To complete the proof of Proposition~\ref{prop:Had},  we also need  \cite[Lemma 2]{Chahkiev08} in order to give an estimate for the integral on the RHS of \eqref{term5}. 

\begin{lemma}[\citep{Chahkiev08}] \label{lem:levelset}
Let $f:\T\to\R$ be a trigonometric polynomial of degree $m\in\N$ and let $\|f\|_{L^2} =\sqrt{ \int_{\T} f(\theta)^2\d\mu}$
where $\d\mu = \frac{\d\theta}{2\pi}$ denotes the uniform measure on $\T$. 
If we let $\mathscr{T}_\lambda = \big\{\theta\in\T : |f(\theta)| \le \lambda \big\}$, then we have for any $\lambda>0$,
\[
\mu(\mathscr{T}_\lambda ) \le  2e \left(\frac{\lambda}{\sqrt{2}\|f\|_{L^2}}\right)^{1/2m} .
\]
\end{lemma}

From Lemma~\ref{lem:levelset}, we deduce that for any trigonometric polynomial $f:\T\to\R$ of degree at most $m\in\N$, we have
\[ \begin{aligned}
\int_\T  e^{- f(\theta)^2} \frac{\d\theta}{2\pi} 
&=  \int_\T \bigg( \int_{0}^{+\infty} e^{-  \lambda} \1_{\big\{|f(\theta)| \le \sqrt{\lambda}\big\}} \d\lambda \bigg) \mu(\d\theta) \\
& = \int_{0}^{+\infty} e^{-  \lambda} \mu\big(\mathscr{T}_{\sqrt{\lambda}}\big)  \d\lambda \\
&\le   2e \int_{0}^{+\infty} e^{-\lambda}  \left(\frac{\lambda}{2\|f\|_{L^2}^2}\right)^{1/4m} \d\lambda  \\
&\le \frac{2e}{(2\|f\|_{L^2}^2)^{1/4m}} , 
\end{aligned}\]
where we used that
$\Gamma(1+1/4m) =   \int_{0}^{+\infty} e^{-\lambda} \lambda^{1/4m} \d\lambda  \le 1$ for any $m\in\N$ in the last step.  
Hence, by combining this estimate with \eqref{term5}, we obtain the following general bound. 

\begin{proposition} \label{prop:TB2}
Let $f:\T\to\R$ be a trigonometric polynomial for degree $m\in\N$,. We have for any $n\ge 2$, 
\[
\E_n \left[ e^{- \sum_{j =1}^n f(\theta_j)^2 } \right]  \le \frac{\cst{15}^n}{\sqrt{2\pi n}(2\|f\|_{L^2}^2)^{N/4}} , 
\]
where $N=\frac nm$, $\cst{15} = 2 e^2$ and $\|f\|_{L^2} =\sqrt{ \int_{\T} f(\theta)^2 \frac{\d\theta}{2\pi}}$. 
\end{proposition}

We are now ready to complete the proof of  Proposition~\ref{prop:Had}. 
By combining the estimates from Proposition~\ref{prop:TB1} and  Proposition~\ref{prop:TB2} with $f= \sqrt{\gamma} \g'$ which is a real--valued\footnote{We verify that for any $n,m\in\N$ and $\eta\in (0,1]$, $\gamma>0$.} trigonometric polynomial of degree $m\in\N$, we obtain that for any $n\ge 2$ and  any $\eta\in (0,1]$,
\begin{align} \notag
\big|  F_{n,m}(\xi)  \big| 
&\le  \exp\Big(\cst{20} \big(n+\tfrac{2}{\pi^2 }\big)\Big)\frac{\cst{15}^n}{\sqrt{2\pi n}(2
\gamma \|g'\|_{L^2}^2)^{N/4}} \\
&\label{term7}
\le  \frac{1}{\sqrt{2\pi n}}\bigg(  \frac{e^{m \eta^2(\frac{\pi^2}{2}+\frac{1}{n})} }{2\gamma\|\xi\|} \bigg)^{N/4}  \frac{\cst{15}^n}{\|\xi\|^{N/4}} ,
\end{align}
where we used that by definition we have $ \| g'\|_{L^2}^2 = \sum_{k=1}^m k |\zeta_k|^2 \ge \| \xi\|^2$ and we replaced $\cst{20}=  \frac{\pi^2 \eta^2}{8}$. 
We still have the freedom to choose the parameter $\eta\in (0,1]$ in the estimate \eqref{term7} and we choose it in such a way to minimize $ \eta^{-1}e^{m \eta^2\frac{\pi^2}{2}}$. That is, we choose $\eta = \frac{1/\pi}{\sqrt m} $ and since $2\gamma\| \xi\| =    \frac{2\eta}{ \sqrt{n}m(m+1)} \left(1  -\frac{\eta^2\cst{21}}{n}    \right)$,  this implies that 
\begin{equation*}
\big|  F_{n,m}(\xi)  \big|   \le \left( \frac{\pi e^{\frac 12(1+\frac{2}{\pi^2n})} \sqrt{n}m^{3/2}(m+1) }{2 \Big(1  - \frac{\cst{21}/\pi^2}{nm}   \Big)}  \right)^{N/4}  \frac{\cst{15}^n}{\|\xi\|^{N/4}} .
\end{equation*}
Finally,  let us observe that in the regime where $n \ge 4m^2$ (note that it is the only place where we use this condition), this implies that 
\[
\big|  F_{n,m}(\xi)  \big|^2   \le \ups{3}(m)^{N/2}  \frac{\cst{15}^{2n} n^{N/4}}{\|\xi\|^{N/2}} .
\]
where $\ups{3}(m) =  \frac{\pi  m^{3/2}(m+1)e^{\frac 12 \big(1+\frac{1/2}{(\pi m)^2}\big)} }{2 \Big(1  - \frac{\cst{21}}{4 \pi^2m^3}   \Big)}$ according to \eqref{c2}. This completes the proof. 
\hfill$\square$


\section{Intermediate regime} \label{sect:QF}

The goal of this section is to prove Proposition~\ref{prop:T0}.
Recall that the polynomial $\g$ is given by \eqref{g_function} and that $h = - \U \g$ is the Hilbert transform of the function $-\g$ -- see \eqref{h}. 
We will make use of the following basic estimates. We have for any $\xi\in\R^{2m}$,
\begin{equation} \label{hL0}
\|\g\|_\infty ,  \| h \|_\infty \le \sqrt{2(1+\log m)} \| \xi\| .  
\end{equation}
Similarly, for any $\xi\in\R^{2m}$ and any  integer $\kappa \ge 0$, 
\begin{equation} \label{hL1}
\| h^{(\kappa+1)}\|_\infty \le  \sum_{k\le m} |\zeta_k|  k^\kappa \sqrt{2k}  \le C_\kappa \|  \xi\| \big(m(m+1)\big)^{\frac{\kappa+1}{2}} , 
\end{equation}
with $C_0 = 1$, $C_1=1/\sqrt{2}$, $C_2=  1/\sqrt{3}$
and 
\begin{equation} \label{hL2}
\| h^{(\kappa+1)}\|_{L^2}  = \sqrt{ \sum_{k\le m}  |\zeta_k|^2 k^{2\kappa+1}  } \le m^{\kappa+1/2} \|\xi\| . 
\end{equation}

We will also make use of Lemma~\ref{lem:est1} which is proved in Section~\ref{sect:est1} and we fix (throughout this section) the parameter
\begin{equation} \label{nu1}
\nu = \frac{ \nu_* N}{ \sqrt{m+1} (1+\log m)^{1/4} \|\xi\|}  ,
\end{equation}
where $N = \frac nm$ and  $0<\nu_* \le \cst{0}$ as in \eqref{c0}. 
This last condition is necessary for our proof of Proposition~\ref{prop:QF} below and we will optimize over the parameter $\nu_*$ in the proof of Proposition~\ref{prop:T0} which is given in the next section.
This proof relies crucially on the following two estimates. 

\begin{proposition} \label{prop:est2}
Let $n,m\in\Z_+$  and $\xi \in \R^{2m}$. If $\nu>0$ is given by \eqref{nu1},  then
\[
\E_n\bigg[ e^{-2 \sum_{j=1}^n \Im \g\big(\theta_j +\i \frac{\nu}{n}h(\theta_j)\big)}  \bigg] 
\le \exp\bigg( - 2 \nu\| \xi\|^2  \bigg( 1- \cst{10} -
\frac{4 \cst{11}\nu_* \|\xi\|} {N \sqrt{m+1}} (1+\log m)^{3/4} \bigg)\bigg) . 
\]
\end{proposition}

\begin{proposition} \label{prop:QF}
Let $n,m\in\Z_+$ with $m\ge 3$, $\xi \in \R^{2m}$ and suppose that the parameter $\nu$ is given by \eqref{nu1} with $0<\nu_*\le \cst{0}$. If $H$ is given by \eqref{H}, we have
\[
\E_n\bigg[ \exp\bigg( \frac{\nu^2}{n^2} \sum_{i,j=1}^n H(\theta_i , \theta_j) \bigg) \bigg] 
\le  \exp\left(2\cst{9} +  \frac{\nu_*^2 N^2(1+ \eps{0})}{(m+1)\sqrt{1+\log m}}   \right)  . 
\] 
\end{proposition}

The proof of Proposition~\ref{prop:est2} is given in Section~\ref{sect:est2} while the proof of Proposition~\ref{prop:QF} is given in Section~\ref{sect:est3}. 
Now that we are equipped with these two estimates, we can proceed with the proof of  Proposition~\ref{prop:T0}.  


\subsection{Proof of  Proposition~\ref{prop:T0}}

Let us recall that the parameter $\nu$ is  chosen according to \eqref{nu1} and we assume that $0<\nu_* \le \cst{0}=\sqrt{\frac1{6\sqrt{2}}}$. 
By combining Lemma~\ref{lem:est1}, Lemma~\ref{prop:est2} and Proposition~\ref{prop:QF}, we obtain 
\begin{equation} \label{boundF1}
\big| F_{n,m}(\xi) \big|^2
\le   \exp\bigg( 2\cst{9} + \frac{\nu_*^2 N^2(1+ \eps{0})}{(m+1)\sqrt{1+\log m}}   -  \frac{2 \nu_* N \|\xi\|}{ \sqrt{m+1} (1+\log m)^{1/4} }  \bigg( 1- \cst{10} -
\frac{4\cst{11} \nu_*\|\xi\| }{N \sqrt{m+1}} (1+\log m)^{3/4} \bigg)\bigg) . 
\end{equation}
Let $\Lambda_2$ be as in \eqref{Lambda2}, that is  
\[ 
\Lambda_2 =   \frac{\cst{0}^{-1}(1-\cst{10})N \sqrt{m+1}}{8(1+\log m)^{3/4}\cst{11}} . 
\] 
In order to maximize the polynomial 
$ \nu_*\big( 1 - \cst{10}-  4\nu_* \cst{11} \frac{ \|\xi\|}{N\sqrt{m+1}} (1+\log m)^{3/4} \big)$, we choose $\nu_* =  \frac{(1-\cst{10})N \sqrt{m+1}}{8 \|\xi\| (1+\log m)^{3/4}\cst{11}}$. 
Then, we verify that in the regime where $\|\xi\| \ge \Lambda_2$, we have $\nu_* 
\le \cst{0}$ so that we are allowed to use the estimate \eqref{boundF1}. We obtain
\[
\big| F_{n,m}(\xi) \big|^2
\le   \exp\bigg( 2\cst{9}+ \frac{\cst{0}^2N^2(1+ \eps{0})}{(m+1)\sqrt{1+\log m}} -  \frac{(1-\cst{10})^2 N^2 }{8 (1+\log m)\cst{11}}\bigg)  
\]
If $\cst{1} =  \frac{(1-\cst{10})^2}{16 \cst{11}} - \cst{0}^2(1+\eps{0}) \frac{\sqrt{1+\log m}}{2(m+1)}$
according to \eqref{c1}, it follows from the previous formula that  in the regime where $\|\xi\| \ge \Lambda_2$,
\[
\big| F_{n,m}(\xi) \big|^2
\le  \exp\bigg( 2\cst{9} - \frac{2\cst{1} N^2}{1+\log m} \bigg) . 
\]
This proves the estimate \eqref{bd1}. 

\medskip

On the other hand, in the regime where  $\|\xi\| \le \Lambda_2$ if we choose $\nu_*=\cst{0}$ in the estimate \eqref{boundF1},  by \eqref{Lambda2}, we verify that
\[ \begin{aligned}
\big| F_{n,m}(\xi) \big|^2
&\le    \exp\bigg(2\cst{9} + \cst{0}^2\bigg( \frac{N^2(1+ \eps{0})}{(m+1)\sqrt{1+\log m}} -  8\cst{11} \frac{\sqrt{1+\log m}} {m+1}\|\xi\|  \big(2\Lambda_2- \|\xi\| \big)\bigg)\bigg)  \\
&\le    \exp\bigg(2\cst{9} + \cst{0}^2\bigg( \frac{N^2(1+ \eps{0})}{(m+1)\sqrt{1+\log m}} -  8\cst{11}  \frac{\sqrt{1+\log m}} {m+1}\Lambda_1 \big(2\Lambda_2- \Lambda_1 \big)\bigg)    \bigg) , 
\end{aligned}\]
where we used that the minimum of the function $\xi \mapsto \|\xi\|  \big(2\Lambda_2- \|\xi\|\big)$ for $\Lambda_1 \le \|\xi\|\le \Lambda_2$ equals 
\[ \begin{aligned}
\Lambda_1 \big(2\Lambda_2- \Lambda_1 \big)
& =  \frac{ \cst{4} N}{\sqrt{1+\log m}} \bigg(\frac{\cst{0}^{-1}(1-\cst{10})N \sqrt{m+1}}{4\cst{11}(1+\log m)^{3/4}} -    \frac{ \cst{4} N}{\sqrt{1+\log m}} \bigg) \\
& =   \frac{\cst{0}^{-1}\cst{4} \sqrt{m+1} }{4\cst{11}(1+\log m)^{5/4}} N^2 \left(1- \cst{10} - \frac{4\cst{4}\cst{0} \cst{11} (1+\log m)^{1/4}}{\sqrt{m+1}} \right) .
\end{aligned} \]

Hence, if  $\cst{2} = \cst{0} \cst{4} \left(1- \cst{10} - \frac{4\cst{4} \cst{0} \cst{11} (1+\log m)^{1/4}}{\sqrt{m+1}} \right) -\cst{0}^2\frac{(1+ \eps{0}) (1+\log m)^{1/4}}{2\sqrt{m+1}}$
according to \eqref{c1}, it follows from the previous formulae that  in the regime where  $\Lambda_1 \le \|\xi\|\le \Lambda_2$, 
\[
\big| F_{n,m}(\xi) \big|^2
\le  \exp\bigg(2\cst{9} -   \frac{2\cst{2}(m) N^2}{\sqrt{m+1} (1+\log m)^{3/4}}  \bigg)    .
\]
This proves the estimate \eqref{bd2} and it completes the proof. It just remains to prove Lemma~\ref{lem:est1} as well as Propositions~\ref{prop:est2} and \ref{prop:QF} which is the task that we undertake in the next sections. 
\hfill$\square$


\subsection{Proof of Lemma~\ref{lem:est1}} \label{sect:est1}

Let us recall that by Lemma~\ref{lem:estF}, we have for any $\nu >0$, 
\begin{equation*}
\big| F_{n,m}(\xi) \big|
\le  \E_n\bigg[ \prod_{1\le i<j \le n} \bigg| \frac{ \sin\big(\frac{\theta_i - \theta_j}{2} + \i \nu \frac{ h(\theta_i) - h(\theta_j)}{2n}  \big)}{\sin\big(\frac{\theta_i - \theta_j}{2}\big)} \bigg|^2  \prod_{j =1}^n  \big| 1+ \i \tfrac \nu n h'(\theta_j)\big| e^{- \Im \g\big(\theta_j +\i \frac{\nu}{n}h(\theta_j)\big)}  \bigg] . 
\end{equation*}
By Lemma~\ref{lem:unibound},  we obtain for all  $\theta_i ,\theta_j \in\T$ with $\theta_i \neq \theta_j$, 
\[ \begin{aligned}
\bigg| \frac{ \sin\big(\frac{\theta_i - \theta_j}{2} + \i \nu \frac{ h(\theta_i) - h(\theta_j)}{2n}  \big)}{\sin\big(\frac{\theta_i - \theta_j}{2}\big)} \bigg|^2
& = 1 +\Bigg( \frac{\sinh\left( \nu \frac{ h(\theta_i) - h(\theta_j)}{2n} \right)}{\sin\left(\frac{\theta_i - \theta_j}{2}\right) }\Bigg)^2 \\
&\le \exp \bigg(  \nu \frac{ h(\theta_i) - h(\theta_j)}{2n \sin\left(\frac{\theta_i - \theta_j}{2}\right)} \bigg)^2 \\
& = \exp\bigg(\frac{\nu^2}{n^2} H(\theta_i , \theta_j) \bigg), 
\end{aligned}\]
where the function $H$ is as in \eqref{H}. Moreover, we also have
\[
\prod_{j =1}^n  \big| 1+ \i \tfrac \nu n h'(\theta_j)\big|^2 \le \exp\left( \frac{\nu^2}{n^2} {\textstyle\sum_{j=1}^n} H(\theta_j , \theta_j) \right) . 
\]
Combining these bounds, we obtain for any $\theta_1, \dots , \theta_n \in\T$ distinct and any $\nu>0$, 
\begin{equation*} 
\prod_{1\le i<j \le n} \bigg| \frac{ \sin\big(\frac{\theta_i - \theta_j}{2} + \i \nu \frac{ h(\theta_i) - h(\theta_j)}{2n}  \big)}{\sin\big(\frac{\theta_i - \theta_j}{2}\big)} \bigg|^2  \prod_{j =1}^n  \big| 1+ \i \tfrac \nu n h'(\theta_k)\big|
\le\exp\bigg( \frac 12\frac{\nu^2}{n^2} \sum_{i,j = 1}^n H(\theta_i , \theta_j) \bigg)  .
\end{equation*}
Hence, by the Cauchy--Schwartz inequality, this implies that 
\[
\big| F_{n,m}(\xi) \big|^2
\le  \E_n\bigg[\exp\bigg(\frac{\nu^2}{n^2} \sum_{i,j = 1}^n H(\theta_i , \theta_j) \bigg) \bigg] \E_n\bigg[ e^{-2 \sum_{j=1}^n \Im \g\big(\theta_j +\i \frac{\nu}{n}h(\theta_j)\big)}  \bigg] . 
\]
\hfill$\square$


\subsection{Proof of Proposition~\ref{prop:est2}} \label{sect:est2}

Recall that according to \eqref{nu1}, we assume that $\nu  = \frac{ \nu_* n}{ m\sqrt{m+1} (1+\log m)^{1/4} \|\xi\|} $ for a constant $\nu_*>0$. Using the estimate  \eqref{hL0}, this implies that $\frac{\nu}{n}\| h\|_\infty \le  \sqrt{2} \nu_* \frac{(1+\log m)^{1/4}}{m\sqrt{m+1} }$.
Then, since both functions $\g,h$ are real--valued on $\T$ and $\g$ is an analytic function, we have
\[
\left| \Im \g\big(\theta_j +\i \frac{\nu}{n}h(\theta_j)\big) - \frac{\nu}{n} \g'(\theta_j) h(\theta_j) \right|
\le \frac{\nu^3}{6 n^3}  |h(\theta_j)|^3
\hspace{-2.8cm} \sup_{\begin{subarray}{c} z\in\C :  \\  \hspace{3cm} | \Re z| \le \pi ,  |\Im z| \le \sqrt{2} \nu_* \frac{(1+\log m)^{1/4}}{m\sqrt{m+1} } \end{subarray}} \hspace{-2cm}  \big| \g'''(z) \big|   \hspace{ 1cm} .
\]
Moreover, by \eqref{g_function},  we have for any $z\in\C$,  
\[
\g'''(z) = \frac{- \i}{\sqrt{2}}  \sum_{|k| \le m} |k|^{5/2} \zeta_k e^{\i k z}
\] 
so that if $ | \Re z| \le \pi , |\Im z| \le \sqrt{2} \nu_* \frac{(1+\log m)^{1/4}}{m\sqrt{m+1}}$,  then  
\[
\big|\g'''(z) \big| \le  \sqrt{2} e^{ \sqrt{2} \nu_* \frac{(1+\log m)^{1/4}}{\sqrt{m+1}}}\sum_{k =1}^m |\zeta_k| k^{5/2} \le  3\sqrt{2} \cst{19} \|\xi\| m^{3/2}(m+1)^{3/2} ,
\]
where $\cst{19}(m) = \frac{1}{3\sqrt{6}}  e^{\sqrt{2} \cst{0} \frac{(1+\log m)^{1/4}}{\sqrt{m+1}}}$ and we used that $\sum_{k = 1}^m k^5 \le \frac{m^3(m+1)^3}{6}$.
These bounds and the estimate  \eqref{hL0} show that
\[ \begin{aligned}
&\E_n\bigg[ \exp\bigg( - 2 \sum_{j =1}^n  \Im  \g\big(\theta_j +\i \frac{\nu}{n}h(\theta_j)\big) \bigg) \bigg]\\
&\qquad\le 
\E_n\bigg[ \exp\bigg( - \frac{2\nu}{n} \sum_{j =1}^n \g'(\theta_j) h(\theta_j) +  2 \cst{19} \|\xi\|^2 \frac{\nu^3m^{3/2}(m+1)^{3/2}}{n^3}\sqrt{1+\log m} \sum_{j =1}^n h(\theta_j)^2   \bigg) \bigg] .
\end{aligned}\]
Let us denote $\gamma = \cst{19}  \|\xi\|^2  \frac{ \nu^2 m^{3/2}(m+1)^{3/2}}{n^2} \sqrt{1+\log m}$ and $f = g' - \gamma h$.
By Lemma~\ref{lem:Laplace}, this implies that
\[ \begin{aligned}
\E_n\bigg[ \exp\bigg( - 2 \sum_{j =1}^n  \Im  g\big(\theta_j +\i \frac{\nu}{n}h(\theta_j)\big) \bigg) \bigg]
&\le  \E_n\bigg[ \exp\bigg(- \frac{2\nu}{n}  \sum_{j =1}^n f(\theta_j) h(\theta_j) \bigg) \bigg] \\
&\le 
\exp\bigg( - 2\nu \int_\T  f(\theta) h(\theta)\frac{\d\theta}{2\pi}  + \frac{4\nu^2}{n^2}  \A(fh) \bigg) .
\end{aligned}\]
First observe that since we have chosen $h = - \U \g$, we have 
\begin{equation} \label{hL3}
\int_\T  h(\theta)^2 \frac{d\theta}{2\pi}  \le \|\xi\|^2  
\end{equation}
and by formulae \eqref{Devinatz1}--\eqref{Devinatz2},  we obtain
\begin{equation} \label{est1}
\E_n\bigg[ \exp\bigg( - 2 \sum_{j =1}^n  \Im  g\big(\theta_j +\i \frac{\nu}{n}h(\theta_j)\big) \bigg) \bigg] 
\le  \exp\bigg(   - 2 \nu(1-  \gamma) \|\xi\|^2  + \frac{4\nu^2}{n^2}  \A(fh) \bigg) .
\end{equation}
It remains to estimate the quantities $ \A(fh)$ where the seminorm $\A$ is given by \eqref{A} and $f = \g' - \gamma h$.
To that end, we may use the bound $ \A(u) \le \| u\|_{L^2} \|u'\|_{L^2}$ which holds for any smooth function $u:\T \to \C$.  
First, we have
\[ \begin{aligned}
\|fh\|_{L^2} \le \|h\|_{\infty} \| f\|_{L^2} \le \sqrt{2(1+\log m)} \left( \sqrt{m} + \gamma \right) \|\xi\|^2
\end{aligned}\]
where we used the estimates \eqref{hL0}, \eqref{hL2} and \eqref{hL3}. Second, we have 
\[\begin{aligned}
\|(fh)'\|_{L^2} 
&\le  \|h\|_\infty \|f'\|_{L^2} + \|f\|_\infty \|h'\|_{L^2}  \\
&\le  \left( \sqrt{2 m(1+\log m)} \left( m  + \gamma \right) + \left( \sqrt{m(m+1)} + \gamma\sqrt{2(1+\log m)} \right) \sqrt{m} \right) \|\xi\|^2 \\
& = m  \sqrt{2 m(1+\log m)}  \left( 1  + \frac{2\gamma}{m} + \sqrt{\frac{1+1/m}{2(1+\log m)}} \right) \|\xi\|^2 
\end{aligned}\]
Here we used that $\| g^{(\kappa)}\|_{L^2} = \| h^{(\kappa)}\|_{L^2}$ for any $\kappa\ge 0$ since $h$ is the Hilbert transform of $g$ and the estimates  \eqref{hL0}--\eqref{hL2}. 
Combining all these estimates, we deduce from formula \eqref{est1} that
\begin{equation*}
\begin{aligned}
&\E_n\bigg[ \exp\bigg( - 2 \sum_{j =1}^n  \Im\Big\{ g\big(\theta_j +\i \frac{\nu}{n}h(\theta_j)\big)\Big\} \bigg) \bigg] \\
&\le \exp\left( - 2 \nu\| \xi\|^2  \left( 1- \gamma -
\frac{ 4\nu\|\xi\|^2 m^{2}}{n^2}(1+\log m) \left( 1 + \frac{\gamma}{\sqrt{m}} \right) \left(1+ \frac{2\gamma}{m} + \sqrt{\frac{1+1/m}{2(1+\log m)}} \right) \right) \right) . 
\end{aligned}
\end{equation*}
To complete the proof,  it remains to observe that by \eqref{nu1} and \eqref{c1}, we have
\[
\gamma = \cst{19}  \|\xi\|^2  \frac{ \nu^2 (m+1)^{3/2}}{N^2 \sqrt{m}} \sqrt{1+\log m} = \cst{19}\nu_*^2\sqrt{1+1/m}
\le \cst{10}(m)  =  \cst{0}^2 \tfrac{\sqrt{1+1/m}}{3\sqrt{6}} e^{\cst{0} \frac{(1+\log m)^{1/4}}{\sqrt{(m+1)/2}}} 
\]
after replacing $\cst{19} = \frac{1}{3\sqrt{6}}  e^{\cst{0} \frac{(1+\log m)^{1/4}}{\sqrt{(m+1)/2}}}$ and using that $\nu_* \le \cst{0}$. 
Moreover,  by \eqref{nu1}, we also have $ \frac{\nu\|\xi\|^2 m^{2}}{n^2}(1+\log m) = \frac{\nu_* \|\xi\| }{N\sqrt{m+1}}(1+\log m)^{3/4}$, so as 
$\cst{11}= \left( 1 + \frac{\cst{10}}{\sqrt{m}} \right) \left(1+ \frac{2\cst{10}}{m} + \sqrt{\frac{1+1/m}{2(1+\log m)}} \right)$, 
this proves the claimed bound.
\hfill$\square$


\subsection{Proof of Proposition~\ref{prop:QF}} \label{sect:est3}

Let us denote for any $k\in\Z$,
\begin{equation*} 
\mathrm{T}_k = \tr \u^k = {\textstyle \sum_{j=1}^n } e^{\i k \theta_j} =  \sqrt{\frac{k}{2}} \big( \mathrm{X}_{2k-1} + \i \mathrm{X}_{2k} \big) . 
\end{equation*}
The idea of the proof is to view $\sum_{i,j = 1}^nH(\theta_i, \theta_j)$ as a quadratic form in  the random variables  $(\mathrm{T}_k)_{k\in\Z}$  and to use this observation to express the Laplace transform of the random variable  $\sum_{i,j = 1}^nH(\theta_i, \theta_j)$  as a (multivariate) Gaussian integral as explained at the end of Section~\ref{sect:est}. 

\begin{lemma} \label{lem:QF}
We have the identity
\[
\sum_{i,j = 1}^nH(\theta_i, \theta_j)= \frac{1}{2} \Re\bigg\{  \sum_{p,q \in \Z} A_{pq} \mathrm{T}_p \mathrm{T}_q +  \sum_{p,q \in \Z} B_{pq} \mathrm{T}_p \mathrm{T}_q \bigg\} ,
\]
where  
\[
A_{pq}=  \sum_{1\le k \le \ell \le m } \big( \1_{1\le p-k +1 \le  p+q-\ell \le m}+   \1_{1\le q-k+1\le p+q -\ell \le m} \big)
\frac{\zeta_\ell \zeta_{p+q-\ell}}{\sqrt{\ell(p+q-\ell)}} 
\]
and
\[
B_{pq} = \sum_{1\le k \le \ell \le m } \big( \1_{1\le k-p \le \ell - p-q \le m}+   \1_{1\le k-q \le \ell - p-q \le m} \big)
\frac{\zeta_\ell \zeta_{p+q-\ell}}{\sqrt{\ell(\ell-p-q)}} . 
\]
\end{lemma}

\begin{proof}
An elementary computation gives that for any $\ell \in\Z$,
\[
\frac{e^{\i \ell \theta}-e^{\i \ell x}}{2\i \sin(\frac{\theta-x}{2})} =   {\textstyle \sum_{k=1}^\ell} e^{\i (k - 1/2) \theta} e^{\i (\ell-k +1/2) x} , \qquad x,\theta \in\T . 
\]
By \eqref{h} -- \eqref{H}, this directly implies that for any $i,j = 1, \dots , n$, 
\[ \begin{aligned}
H(\theta_i, \theta_j) & =  \Re\bigg\{ \sum_{1\le k \le \ell \le m}  \sum_{1\le r \le s \le m}   \frac{\zeta_\ell \zeta_s}{\sqrt{\ell s}}  e^{\i (k - 1/2) \theta_i} e^{\i (\ell-k +1/2) \theta_j}  e^{\i (r - 1/2) \theta_i} e^{\i (s-r +1/2) \theta_j} \bigg\} \\
&\quad+  \Re\bigg\{ \sum_{1\le k \le \ell \le m}  \sum_{1\le r \le s \le m}   \frac{\zeta_\ell  \overline{\zeta_s}}{\sqrt{\ell s}}  e^{\i (k - 1/2) \theta_i} e^{\i (\ell-k +1/2) \theta_j}  e^{- \i (r - 1/2) \theta_i} e^{-\i (s-r +1/2) \theta_j} \bigg\}  . 
\end{aligned}\]
Then summing over all variables   $\theta_i, \theta_j$, we obtain
\begin{align} \label{sum1}
\sum_{1\le i,j \le n}  H(\theta_i, \theta_j) =  \Re\bigg\{ \sum_{1\le k \le \ell \le m}  \sum_{1\le r \le s \le m}   \frac{\zeta_\ell \zeta_s}{\sqrt{\ell s}}  \mathrm{T}_{k+r-1} \mathrm{T}_{\ell+s-k-r+1} \bigg\} \\
\label{sum2}
+  \Re\bigg\{  \sum_{1\le k \le \ell \le m}  \sum_{1\le r \le s \le m}   \frac{\zeta_\ell  \overline{\zeta_s}}{\sqrt{\ell s}}  \mathrm{T}_{k-r} \mathrm{T}_{\ell-s+r-k} \bigg\}  . 
\end{align}
In \eqref{sum1} we make the change of variables $(r,s) \leftrightarrow (p,q)$ given by $r=p-k+1$ and $s=  q+p  -\ell $. Similarly, in \eqref{sum2} we make the change of variables $(r,s) \leftrightarrow (p,q)$ given by $r=k-p$ and $s=\ell -q-p$. This implies that
\begin{align}  \label{sum3}
\sum_{1\le i,j \le n}  H(\theta_i, \theta_j) & =   \Re\bigg\{ \sum_{1\le k \le \ell \le m}  \sum_{p,q \in\Z} \frac{\zeta_\ell \zeta_{q+p  -\ell }}{\sqrt{\ell( q+p  -\ell)}}   \1_{1\le p-k+1 \le  q+p  -\ell \le m} \mathrm{T}_p \mathrm{T}_q \bigg\} \\
& \label{sum4}
\quad+  \Re\bigg\{  \sum_{1\le k \le \ell \le m}  \sum_{p,q\in\Z}   \frac{\zeta_\ell  \overline{\zeta_{\ell -q-p}}}{\sqrt{\ell(\ell -q-p)}}  \1_{1\le k-p \le \ell -q-p \le m} \mathrm{T}_p \mathrm{T}_q  \bigg\}  . 
\end{align}
To finish the proof, it remains to symmetrize the previous formula over $(p,q)$ and use that 
$\overline{\zeta_{-j}} = \zeta_{j}$ for all $j=1,\dots , m$. Then \eqref{sum3} corresponds to $\frac{1}{2} \Re\big\{ \sum_{p,q \in \Z} A_{pq} \mathrm{T}_p \mathrm{T}_q \big\}$ and  \eqref{sum4} corresponds to $\frac{1}{2} \Re\big\{ \sum_{p,q \in \Z} B_{pq} \mathrm{T}_p \mathrm{T}_q \big\} $. 
\end{proof}

Let us observe that in the notation of Lemma~\ref{lem:QF} , $A_{pq} \neq 0$ only if $ 1\le p,q \le 2m-1$ and $B_{pq} \neq 0$  only if  $|p|, |q| \le m-1$, so we may view $\mathbf{A} = (A_{pq})_{p,q = 1}^{2m-1}$ and $\mathbf{B} = (B_{pq})_{ 1 \le |p|,|q| < m}$  as symmetric matrix--valued functions of the parameters  $(\zeta_k)_{k=1}^m$. 
In the following, we denote 
\[
\mathfrak{Q}_{\mathbf{A}} =   \Re \bigg\{  \sum_{p,q \in \Z} A_{pq} \mathrm{T}_p \mathrm{T}_q \bigg\}  , \qquad 
\mathfrak{Q}_{\mathbf{B}} =  \Re \bigg\{ \sum_{ \substack{ p,q \in \Z \\ p, q\neq 0}} B_{pq} \mathrm{T}_p \mathrm{T}_q  \bigg\} 
\qquad\text{and}\qquad
\mathfrak{L} =   \Re \bigg\{nB_{00}+ 2\sum_{\substack{ p\in\Z \\ p\neq 0}} B_{p0}  \mathrm{T}_p \bigg\} . 
\]
We introduce this decomposition because $\mathrm{T}_0 = n $ is not a random variable and should be treated 	individually. 
By Cauchy--Schwartz inequality, Lemma~\ref{lem:QF} implies that
\begin{equation} \label{GI3}
\E_n\bigg[ \exp\bigg( \delta \sum_{i,j=1}^n H(\theta_i , \theta_j) \bigg) \bigg] 
\le \Big( \E_n \big[ \exp (2 \delta \mathfrak{Q}_{\mathbf{A}} )\big]  \E_n \big[ \exp (2 \delta \mathfrak{Q}_{\mathbf{B}} )\big]    \Big)^{1/4}  \sqrt{\E_n \big[ \exp ( \delta n \mathfrak{L} ) \big]  } ,
\end{equation}
where $\delta = (\frac{\nu}{n})^2$. 
Our first observation is that $\mathfrak{L}$ is a linear statistic  associated with the trigonometric polynomial $f= nB_{00}+ 2 \Re \big\{  \sum_{ p\neq 0} B_{p0}  e^{\i p\theta}\}$ so that by Lemma~\ref{lem:Laplace}, we have the estimate
\[
\E_n \big[ \exp ( \delta n \mathfrak{L} ) \big] \le  \exp\bigg(  \delta n^2 B_{00} +  ( \delta n)^2  \sum_{0<p<m} p |B_{p0} + \overline{B_{-p0}} |^2  \bigg)  . 
\]
In combination with Lemma~\ref{lem:B0} below, this implies that
\begin{equation} \label{GI4}
\E_n \big[ \exp ( \delta n \mathfrak{L} ) \big] \le  \exp\bigg(  2\delta  n^2 \|\xi\|^2+  (\delta n)^2 \frac{4m^3}{3} \| \xi\|^4   \bigg) .
\end{equation}

\begin{lemma} \label{lem:B0} 
In the notation of Lemma~\ref{lem:QF}, we  have $B_{00} =  2\|\xi\|^2 $  and 
\[
\sum_{0<p<m} p |B_{p0} + \overline{B_{-p0}} |^2  \le \frac{4 m^3}{3} \|\xi\|^4 .
\]
\end{lemma}

\begin{proof}
First of all, by definition, we have
\[
\frac{B_{00}}{2} =   \sum_{1\le k \le \ell \le m } \frac{\zeta_{\ell} \zeta_{-\ell}}{\ell} = \sum_{1\le k \le m} |\zeta_\ell|^2
= \|\xi\|^2 .
\]
Secondly, we also have for any $p\in\Z$, 
\[
B_{p0} =   \sum_{1\le k \le \ell \le m } \big( \1_{1\le k-p \le \ell - p \le m}+   \1_{1\le k \le \ell - p \le m} \big)
\frac{\zeta_\ell \zeta_{p-\ell}}{\sqrt{\ell(\ell-p)}} .
\]
This shows that for $p\ge 1$,
\[ \begin{aligned}
| B_{p0} | & \le \|\xi\|   \sum_{1\le k \le \ell \le m } \big( \1_{p+ 1\le k \le \ell  \le m}+   \1_{1\le k \le \ell - p \le m} \big)\frac{| \zeta_\ell |}{\sqrt{\ell(\ell-p)}}  \\
&= 2  \|\xi\|   \sum_{ p+1 \le \ell \le m}  \sqrt{1-p/\ell} | \zeta_\ell | , 
\end{aligned}\]
where at the second step we computed the sum over $k$. By  Cauchy--Schwartz inequality, this shows that
\[
| B_{p0} |  \le  2 \sqrt{m - p}   \|\xi \|^2
\]
This estimate implies that 
\[
\sum_{p =1}^{m-1} p | B_{p0} |^2 \le 4 \|\xi\|^4 \sum_{p =1}^{m-1}  p(m-p) \le \frac{2 m^3}{3} \|\xi\|^4 . 
\]
Similarly, we can show that $| B_{-p0} |  \le  2 \sqrt{m - p}   \|\xi \|^2$ for any $p\ge 1$ so that we also have 
$ \sum_{p =1}^{m-1} p | B_{-p0} |^2 \le \frac{2 m^3}{3} \|\xi\|^4$. 
This completes the proof.
\end{proof}

\medskip

In the remainder of this section,  our task  is to bound the terms which involve the quadratic forms $ \mathfrak{Q}_{\mathbf{A}} $ and $ \mathfrak{Q}_{\mathbf{B}} $ on the RHS of \eqref{GI3}. 
In order to do this, we  need a priori estimates for the norms of the corresponding matrices $\mathbf{A}$ and $\mathbf{B}$. 

\begin{lemma} \label{lem:matnorm}
Let $ \|\mathbf{A}\|={\displaystyle \max_{1 \le p < 2m}} \sum_{q=1}^{2m-1} |A_{pq}|$ and  
$\| \mathbf{B}\| ={\displaystyle \max_{ 1 \le |p | <m}} \sum_{|q|=1}^{m-1} |B_{pq}|$. We have 
\[ 
\|\mathbf{A}\|,  \|\mathbf{B}\| \le  \sqrt{2m(m+1)(1+\log m)} \|\xi\|^2 .
\]
\end{lemma}

\begin{proof}
By definition, we have 
\[
\sum_{q=1}^{2m-1} |A_{pq}| \le  2 \sum_{1\le k \le \ell \le m }  \frac{|\zeta_\ell|}{\sqrt\ell}  \sum_{q=1}^{2m-1}
\1_{1\le  p+q-\ell \le m} \frac{ |\zeta_{p+q-\ell}|}{\sqrt{p+q-\ell}} . 
\]
The last sum is bounded by $\sum_{r=1}^{m} \frac{|\zeta_r|}{\sqrt r} $, so we obtain
\[
\sum_{q=1}^{2m-1} |A_{pq}| \le  2 \sum_{k , r =1}^{m} \sqrt{\frac{\ell}{r}}  |\zeta_\ell| |\zeta_r|. 
\]
By the Cauchy--Schwartz inequality, this shows that
\[
\sum_{q=1}^{2m-1} |A_{pq}| \le  2 \|\zeta\|^2   \sqrt{\sum_{k , r =1}^{m} \frac{\ell}{r}}   \le \sqrt{2m(m+1)(1+\log m)} \|\zeta\|^2 .
\]
Since $ \|\zeta\| = \| \xi\|$, this gives the estimate for  $\|\mathbf{A}\|$ -- the argument for  $\|\mathbf{B}\|$ is exactly the same. 
\end{proof}

Let us define new objects. For $\delta_1 , \delta_2 >0$, we set
\begin{equation} \label{def:M}
\mathbf{M} = \begin{pmatrix} 
\mathbf{I}_{2m-1} & \delta_2  \mathbf{A}^* \\
\delta_2  \mathbf{A} &    \mathbf{I}_{2m-1} 
\end{pmatrix} 
\qquad \text{and}\qquad
\mathbf{v} = \sqrt{\delta_1} \begin{pmatrix} \mathrm{T}_1 \\ \vdots \\ \mathrm{T}_{2m-1} \end{pmatrix} . 
\end{equation}

\begin{remark} \label{rk:pd}
 As explained in Section~\ref{sect:est}, it is not clear whether the matrices $\mathbf{A}$ and $\mathbf{B}$ are positive definite. This issue is resolved by bounding the quadratic from $\mathfrak{Q}_{\mathbf{A}}$ using the matrix $\mathbf{M}$ (see the estimate \eqref{GI6}) and by choosing the parameter $\delta_2$ small enough to guarantee that $\mathbf{M}$ is positive definite and the Gaussian integral \eqref{GI} is convergent. 
 \hfill$\blacksquare$
\end{remark}

Since $\mathbf{A}$ is a symmetric matrix, we have  $\|\mathbf{A}^*\| = \|\mathbf{A}\|$ and  \[
\|  \mathbf{M} - \mathbf{I}_{4m-2} \| =  \max_{1 \le p < 4m-1} {\textstyle \sum_{q=1}^{4m-2}} |M_{pq} - \1_{pq}|
= \delta_2 \|\mathbf{A}\| . 
\]
Hence, by Lemma~\ref{lem:matnorm}, if  $\delta_2 \le \frac{1}{ 3 \sqrt{2m(m+1)(1+\log m)}\|\xi\|^2 }$, then 
\begin{equation} \label{Mnorm}
\|  \mathbf{M} - \mathbf{I}_{4m-2} \| \le \frac{1}{3}  ,
\end{equation}
so that the matrix $ \mathbf{M}$ is invertible with $ \mathbf{M}^{-1} = \sum_{k=0}^{+\infty} \big(\mathbf{I}_{4m-2} -\mathbf{M} \big)^k$ (convergent Neumann series).
This also implies that  $ \mathbf{M}$ is positive definite and we have
\[
\begin{pmatrix} \mathbf{v}^*  & \overline{\mathbf{v}}^*\end{pmatrix} 
\begin{pmatrix} 
\mathbf{I} & \delta_2  \mathbf{A}^* \\
\delta_2  \mathbf{A} &   \mathbf{I} &
\end{pmatrix} 
\begin{pmatrix} \mathbf{v} \\ \overline{\mathbf{v}}\end{pmatrix} 
=  2 \delta_2  \Re\left\{  \mathbf{v}^{\mathrm{t}}  \mathbf{A}  \mathbf{v}  \right\}  + 2 | \mathbf{v} |^2 .
\]
Thus, if  we set $\delta= \delta_1 \delta_2$ and use the notation \eqref{def:M}, this shows that 
\begin{equation} \label{GI6}
2\delta \mathfrak{Q}_{\mathbf{A}}  = 2 \delta  \Re\bigg\{  \sum_{p,q \in \Z} A_{pq} \mathrm{T}_p \mathrm{T}_q \bigg\}  \le 
\begin{pmatrix} \mathbf{v} \\ \overline{\mathbf{v}}\end{pmatrix}^*
\mathbf{M}
\begin{pmatrix} \mathbf{v} \\ \overline{\mathbf{v}}\end{pmatrix} . 
\end{equation}

Then, in order to estimate the quantity $\E_n \big[ \exp (2 \delta \mathfrak{Q}_{\mathbf{A}} )\big] $, we may use the identity
\begin{equation} \label{GI}
\pi^{4m-2} \det( \mathbf{M}) \exp\left(  \begin{pmatrix} \mathbf{v} \\ \overline{\mathbf{v}}\end{pmatrix}^*
\mathbf{M}
\begin{pmatrix} \mathbf{v} \\ \overline{\mathbf{v}}\end{pmatrix} \right)
= \int_{\C^{4m-2}}  \hspace{-.3cm} \exp\left(- \mathbf{z}^*  \mathbf{M}^{-1} \mathbf{z} + \mathbf{z}^* \begin{pmatrix} \mathbf{v} \\ \overline{\mathbf{v}}\end{pmatrix} +\begin{pmatrix} \mathbf{v} \\ \overline{\mathbf{v}}\end{pmatrix}^* \mathbf{z} \right) \d^2\mathbf{z} , 
\end{equation}
where $\mathbf{z} = \begin{pmatrix} z_1 \\ \vdots \\ z_{4m-2} \end{pmatrix}$ and
$\d^2\mathbf{z} = \prod_{k=1}^{4m-2} \d\Re(z_k)\d\Im(z_k)$ denotes the Lebesgue measure on $\C^{4m-2}$. 
Formula \eqref{GI} is a simple Gaussian integration on  $\C^{4m-2}$ and it makes sense since we have seen that the matrix $\mathbf{M}$ is positive definite by \eqref{Mnorm}. Moreover, it  is useful since 
\[
\begin{pmatrix} \mathbf{v} \\ \overline{\mathbf{v}}\end{pmatrix}^* \mathbf{z} = \sqrt{\delta_1} \sum_{k=1}^{2m-1} \left( z_k \overline{\mathrm{T}_k} + z_{2m-1+k} \mathrm{T}_k \right)
\]
is a (mean--zero) linear statistic of a trigonometric polynomial, so that by Lemma~\ref{lem:Laplace},  we have
\begin{align} \notag
\E_n \left[ \exp\left( \mathbf{z}^* \begin{pmatrix} \mathbf{v} \\ \overline{\mathbf{v}}\end{pmatrix} +\begin{pmatrix} \mathbf{v} \\ \overline{\mathbf{v}}\end{pmatrix}^* \mathbf{z} \right)  \right]
&\le  \exp\left( \delta_1 \sum_{k=1}^{2m-1} k  \left|  \overline{z_k} + z_{2m-1+k} \right|^2
\right) \\
&\label{GI2}
\le \exp\left(2 \delta_1 \mathbf{z}^*\mathbf{C}\mathbf{z}  \right)
\end{align}
where $\mathbf{C}$ is a diagonal matrix given by
\[
\mathbf{C} =  {\small\begin{pmatrix} 1 \\ & \ddots \\ &&  \hspace{-.5cm} 2m-1 \\ && \hspace{.5cm} 1 \\ &&& \ddots \\ &&&& \hspace{-.5cm} 2m-1  \end{pmatrix}}
\]

Hence, taking expectation in formula \eqref{GI} and using the bound \eqref{GI2}, we obtain
\begin{align} \notag
\E_n \Bigg[ \exp\left(  \begin{pmatrix} \mathbf{v} \\ \overline{\mathbf{v}}\end{pmatrix}^*
\mathbf{M}  \begin{pmatrix} \mathbf{v} \\ \overline{\mathbf{v}}\end{pmatrix} \right) \Bigg]
& \le \frac{1}{\pi^{4m-2} \det( \mathbf{M})}
\int_{\C^{4m-2}}  \hspace{-.3cm} \exp\left(- \mathbf{z}^*  (\mathbf{M}^{-1}  - 2\delta_1 \mathbf{C}) \mathbf{z} \right) \d^2\mathbf{z}  \\
& \label{GI1}
=\frac{\det (\mathbf{M}^{-1}  - 2\delta_1 \mathbf{C})^{-1} }{\det(\mathbf{M})} = \frac{1}{\det (\mathbf{I}  - 2\delta_1 \mathbf{M}\mathbf{C})} . 
\end{align}
Here we used that the matrix $\mathbf{M}^{-1}  - 2\delta_1 \mathbf{C}$ is also positive definite.
Indeed, it follows from the above discussion (in particular from the estimate \eqref{Mnorm}) that if $\delta_2 \le \frac{1}{ 3\sqrt{2m(m+1)(1+\log m)}\|\xi\|^2 }$ and  $\delta_1 \le \frac{1}{2 m^{3/2}\sqrt{m+1 }}$, then  for any $m\ge 3$, 
\[
\| \mathbf{M}^{-1} -\mathbf{I}_{4m-2} \| \le   \sum_{k=1}^{+\infty} \| \mathbf{M} -\mathbf{I}_{4m-2} \|^k \le 
\frac{1}{2}
\qquad\text{and}\qquad
2\delta_1 \| \mathbf{C} \| \le   \frac{2m-1}{m^{3/2}\sqrt{m+1 }} \le \frac{5}{6\sqrt{3}}  < \frac{1}{2}. 
\]
Note that the condition $m\ge 3$ is crucial  in order to obtain the second estimate.  
Moreover, since $\mathbf{M}, \mathbf{C}$ are Hermitian matrices with $\| \mathbf{M}\|\le 4/3 $, it follows that  for all $m\ge 3$, 
\begin{equation*} 
\det (\mathbf{I}  - 2\delta_1 \mathbf{M}\mathbf{C}) \ge \left(1 - \frac{4(2m-1)}{3m^{3/2}\sqrt{m+1 }} \right)^{2(2m-1)} \ge e^{- \cst{17}} , 
\end{equation*}
where $\cst{17} = \frac{32}{3}(1+ \frac{(2-1/m)^3}{3(m+1)})$ and we used that $1-x \ge e^{-x - x^2}$ for $0 \le x\le 2/3$. 
Hence, by formula \eqref{GI6} and  \eqref{GI1}, we obtain for $m\ge 3$,
\begin{equation} \label{GI7}
\E_n\big[ \exp(2\delta \mathfrak{Q}_{\mathbf{A}}) \big] \le  \E_n \Bigg[ \exp\left(  \begin{pmatrix} \mathbf{v} \\ \overline{\mathbf{v}}\end{pmatrix}^*
\mathbf{M}  \begin{pmatrix} \mathbf{v} \\ \overline{\mathbf{v}}\end{pmatrix} \right) \Bigg]
\le  e^{\cst{17}}.
\end{equation}

In a analogous way,  let us denote 
\[
\mathbf{N} = \begin{pmatrix} 
\mathbf{I} & \delta_2  \mathbf{B}^* \\
\delta_2  \mathbf{B} &   \mathbf{I} &
\end{pmatrix} ,
\qquad
\mathbf{D} =  {\small\begin{pmatrix} m-1 \\  \hspace{.7cm} \ddots \\ &  1 \\ && 1 \\ &&& \ddots \\ &&&&  \hspace{-.3cm}  m-1  \end{pmatrix}}
\qquad \text{and}\qquad
\mathbf{w} = \sqrt{\delta_1}  {\small \begin{pmatrix} \mathrm{T}_{-m+1} \\ \vdots \\ \mathrm{T}_{-1}  \\ \mathrm{T}_1 \\ \vdots  \\ \mathrm{T}_{m-1} \end{pmatrix}} .
\]
Then we have 
\[
2\delta \mathfrak{Q}_{\mathbf{B}} = 2 \delta \Re \bigg\{ \sum_{ \substack{ p,q \in \Z \\ p, q\neq 0}} B_{pq} \mathrm{T}_p \mathrm{T}_q  \bigg\}   \le 
\begin{pmatrix} \mathbf{w} \\ \overline{\mathbf{w}}\end{pmatrix}^*
\mathbf{N}
\begin{pmatrix} \mathbf{w} \\ \overline{\mathbf{w}}\end{pmatrix} . 
\]
By Lemma \ref{lem:matnorm},  if $\delta_2 \le \frac{1}{ 3\sqrt{2m(m+1)(1+\log m)}\|\xi\|^2}$ ,  then 
$\| \mathbf{N} - \mathbf{I}_{4m-2} \| \le 1/3$ so that  both $\mathbf{N}$ and $\mathbf{N}^{-1}  - 2\delta_1 \mathbf{D}$ are positive definite matrices.
Like in our previous computations, this implies that
\[ \begin{aligned} 
\E_n \bigg[ \exp\left(  \begin{pmatrix} \mathbf{w} \\ \overline{\mathbf{w}}\end{pmatrix}^*
\mathbf{N}
\begin{pmatrix} \mathbf{w} \\ \overline{\mathbf{w}}\end{pmatrix} \right) \bigg]
& = \frac{1}{\pi^{4m-4} \det( \mathbf{N}) } \int_{\C^{4m-4}}  \hspace{-.3cm} \exp\left(- \mathbf{z}^*  \mathbf{N}^{-1} \mathbf{z} \right)  \E_n \bigg[ \exp\left(  \mathbf{z}^* \begin{pmatrix} \mathbf{w} \\ \overline{\mathbf{w}}\end{pmatrix} +\begin{pmatrix} \mathbf{w} \\ \overline{\mathbf{w}}\end{pmatrix}^* \mathbf{z} \right)  \bigg]\d^2\mathbf{z}  \\
& \le  \frac{1}{\pi^{4m-4} \det( \mathbf{N}) }  \int_{\C^{4m-4}}   \hspace{-.3cm} \exp\left(- \mathbf{z}^*( \mathbf{N}^{-1} -2 \delta_1 \mathbf{D}) \mathbf{z} \right) \d^2\mathbf{z}  \\
& = \frac{1}{\det (\mathbf{I}  - 2\delta_1 \mathbf{N}\mathbf{D})} ,
\end{aligned}\]
where at the second step we used an estimate analogous to \eqref{GI2}. 
Moreover, since $\mathbf{N}, \mathbf{D}$ are Hermitian matrices with $\| \mathbf{N}\|\le 4/3 $ and $\| \mathbf{D}\| = m-1$ for $m\ge 3$, if  $\delta_1 \le \frac{1}{2m^{3/2}\sqrt{m+1}}$, we have
\[
\det (\mathbf{I}  - 2\delta_1 \mathbf{N}\mathbf{D}) \ge  \left(1 - \frac{4(m-1)}{3m^{3/2}\sqrt{m+1}} \right)^{2(m-1)} 
\ge e^{-\cst{18}} ,
\]
where $\cst{18} = \frac{8}{3}(1+\frac{4(1-1/m)^3}{3(m+1)})$ and we used that  $1-x \ge e^{-x -  x^2}$ for $0 \le x\le 2/3$. 
Combining these estimates, this implies that  for $m\ge 3$, 
\begin{equation} \label{GI8}
\E_n \big[ \exp\left(2 \delta \mathfrak{Q}_{\mathbf{B}} \right) \big]
\le \E_n \bigg[ \exp\left(  \begin{pmatrix} \mathbf{w} \\ \overline{\mathbf{w}}\end{pmatrix}^*
\mathbf{N}
\begin{pmatrix} \mathbf{w} \\ \overline{\mathbf{w}}\end{pmatrix} \right) \bigg]
\le  e^{\cst{18}}.
\end{equation}

\medskip

Now, let us recall that we must have $\delta =\left(\frac{\nu}{n} \right)^2 = \delta_1 \delta_2$.
Hence, if we choose
\[
\delta_1 =   \frac{1}{2m^{3/2}\sqrt{m+1}}
\qquad\text{and}\qquad
\delta_2 =   \frac{2 \nu^2 m^{3/2} \sqrt{m+1}}{n^2} =  \frac{2 \nu_*^2 }{\sqrt{m(m+1)(1+\log m)} \|\xi\|^2}
\]
according to \eqref{nu1}, then we have $\delta_2 \le \frac{1}{ 3\sqrt{2m(m+1)(1+\log m)}\|\xi\|^2 }$  as required provided that 
$\nu_*^2 \le \frac{1}{6\sqrt{2}}$. 
Observe that this explains our choice for $\cst{0}$ as the maximum admissible value for $\nu_*$. 
In the end, if we combine all our  estimates \eqref{GI3},  \eqref{GI4},  \eqref{GI7} and  \eqref{GI8}, if the parameter $\nu$ is given by \eqref{nu1} and $m\ge 3$, then we obtain
\begin{equation} \label{GI9}
\E_n\bigg[ \exp\bigg( \delta \sum_{i,j=1}^n H(\theta_i , \theta_j) \bigg) \bigg] 
\le \exp\left(2 \frac{\cst{17}+ \cst{18}}{8} + \delta  n^2 \|\xi\|^2+  (\delta n)^2 \frac{2m^3}{3} \| \xi\|^4   \right) .
\end{equation} 
By definitions, we have $\frac{\cst{17}+ \cst{18}}{8} = \frac{1}{3}(5+  4\frac{(1-1/m)^3+(2-1/m)^3}{3(m+1)})$. This function attains it maximum over the positive integers for $m=3$, so that  $\frac{\cst{17}+ \cst{18}}{4} \le \cst{9} = \frac{538}{243}$. 
Hence, if we replace $\delta = \frac{\nu^2}{n^2}$ in formula \eqref{GI9} and use \eqref{nu1}, we conclude that
\[ \begin{aligned} 
\E_n\bigg[ \exp\bigg( \frac{\nu^2}{n^2} \sum_{i,j=1}^n H(\theta_i , \theta_j) \bigg) \bigg] 
& \le \exp\left(2\cst{9} +  \nu^2 \|\xi\|^2+   \frac{2\nu^4m^3}{3n^2} \| \xi\|^4   \right) \\
&= \exp\left(2\cst{9} +  \frac{\nu_*^2 N^2}{(m+1)\sqrt{1+\log m}} + \frac{2\nu_*^4 N^2}{3(m+1)(1+1/m)(1+\log m)} \right) .
\end{aligned}\] 
By definition of $\eps{0}(m) \ge 0$, \eqref{epsilon0},  this completes the proof. \hfill$\square$


\section{Proof of Theorem~\ref{thm:Stein}} \label{sect:Stein}

The method used in this section relies on the formalism introduced in \citep{LLW19} which provides a normal approximation result for certain observable of a Gibbs--type distribution and the following moment identities from \citep{DS94}. According to \eqref{def:X}, we let for any $k\ge 1$, 
\begin{equation} \label{TX}
\mathrm{T}_{k} = \sqrt{\frac{2}{k}}  \tr \u^k = \mathrm{X}_{2k} + \i \mathrm{X}_{2k+1}. 
\end{equation}

\begin{theorem}[\citep{DS94}] \label{thm:DS}
Fix $m \in \N$ and let $\mathbf{a} , \mathbf{b} \in \{0,1,\dots \}^m$. Then, for all $n\ge \sum_{k =1}^m k a_k \vee \sum_{k =1}^m k b_k$, 
\[
\E_n\bigg[ \prod_{k=1}^m \mathrm{T}_k^{a_k} \overline{\mathrm{T}_k^{b_k}}\bigg] = \E\bigg[ \prod_{k=1}^m \mathrm{Z}_k^{a_k} \overline{\mathrm{Z}_k^{b_k}}\bigg] .
\]
where  $\mathrm{Z}_k  = \mathrm{G}_{2k} + \i \mathrm{G}_{2k+1}$ for all $k\ge 1$ and $\mathrm{G}_k$
are i.i.d. standard Gaussian random variables. 
\end{theorem}

Note that the hypothesis of Theorem~\ref{thm:DS} are incorrectly stated in \citep{DS94} and we refer instead to \citep{DE01} for a correct version of this Theorem as well as several applications to the asymptotic distributions of linear statistics of the eigenvalues of the CUE.

\medskip

One can interpret the the law \eqref{pdf} of the eigenvalues of the CUE as a Gibbs distribution\footnote{This means that the probability measure \eqref{pdf} also describes a 2d Coulomb gas of $N$ point charges confined on the unit circle at inverse temperature $\beta=2$.} on $\T^n$ 
with energy $\Phi(\boldsymbol \theta)  : = \sum_{1\le i<j \le n}\log|2\sin(\frac{\theta_i-\theta_j}{2})|^{-2}$. This implies that formally,  \eqref{pdf}  is the stationary measure of a diffusion with  generator
\begin{equation} \label{def:L}
\mathrm{L} =  - \Delta + \nabla \Phi \cdot \nabla = - \sum_{j=1}^n  \bigg( \partial_{jj}  + \sum_{i\neq j} \frac{ \partial_{j}}{\tan\left(\frac{\theta_j-\theta_i}{2}\right)}  \bigg) .
\end{equation}

We view the vector  $\mathbf{X} : \T^n \to \R^{2m}$ as a smooth function in $L^\infty(\P_n)$, so that we can define the vector $\mathrm{L}\mathbf{X}$ and  the $2m \times 2m$ matrix 
\begin{equation} \label{Gamma2}
\boldsymbol{\Gamma}_{k ,\ell } =  \nabla \mathrm{X}_k \cdot \nabla \mathrm{X}_\ell . 
\end{equation}
Recall also the definition of the  Kantorovich or Wasserstein distance \eqref{def:W}. 
Then, by applying \cite[Corollary~2.4]{LLW19} to the random variable $\mathbf{X}$ we obtain the following result. 

\begin{proposition} \label{prop:LLW}
For all $n,m\in\N$ and for any positive definite diagonal matrix $\mathbf{K}$ of size $2m \times 2m$, we have
\begin{equation} \label{LLW}
\mathrm{W}_2(\X,\G) 
\le \sqrt{\E_n \big[ |  \mathbf{K}^{-1} \mathrm{L}\mathbf{X} - \mathbf{X}|^2 \big]}
+ \sqrt{  \E_n\big[  \| \mathbf{I} -   \mathbf{K}^{-1} \boldsymbol{\Gamma} \|^2 \big] } , 
\end{equation}
where $\| \cdot\|$ denotes the Hilbert--Schnmidt norm. 
\end{proposition}

The reason the RHS of \eqref{LLW} is small is because the random variables $\mathrm{T}_1,\mathrm{T}_2, \dots$ are \emph{approximate eigenfunctions} of the generator $\mathrm{L}$ and the matrix $\mathbf{K}$ records the corresponding eigenvalues. 
The following Lemma makes this claim precise.
Observe that $\x$ is small compared to $\mathbf{K}$ which is of order~$n$. 

\begin{lemma} \label{lem:eig}
For all $n,m\in\N$, we have $\mathrm{L}\X=   \mathbf{K} \X +   \x$ where
\[\begin{aligned}
\mathbf{K}  & = n \cdot \mathrm{diag}(1,1,2,2,\cdots, m,m) \\
\x & = \big( \Re \zeta_1  ,  \Im \zeta_1   ,  \Re \zeta_2 , \Im \zeta_2 , \cdots ,  \Re \zeta_m  ,  \Im \zeta_m \big)
\end{aligned}\]
and for all $k\ge 1$, 
\vspace*{-.3cm}
\[
\zeta_k =   \sqrt{\frac k2} \sum_{\ell = 1}^{k-1} \sqrt{\ell(k-\ell)}  \mathrm{T}_\ell \mathrm{T}_{k-\ell} . 
\]
\end{lemma}

\begin{proof}
The Lemma follows from the fact that $\mathrm{T}_k = \sqrt{\frac 2 k} \sum_{j=1}^n e^{\i k \theta_j}$ and explicit computations.
Let us  fix $k \in \N$ and observe that
\begin{equation} \label{explicit2}
\Delta \mathrm{T}_k  =  - k^2 \mathrm{T}_{k}.  
\end{equation}
Second, since  $\tan\left(\frac{\theta_j-\theta_i}{2}\right) = - \i\frac{e^{\i \theta_j}- e^{\i \theta_i}}{e^{\i \theta_j}+ e^{\i \theta_i}}$ for any $i, j =1, \dots , n$, we have
\[
\sum_{i\neq j} \frac{ \partial_{j} \mathrm{T}_k}{\tan\left(\frac{\theta_j-\theta_i}{2}\right)}
=  - \sqrt{2 k} \sum_{i\neq j}  \frac{e^{\i \theta_j}+ e^{\i \theta_i}}{e^{\i \theta_j}- e^{\i \theta_i}} e^{\i k \theta_j} .
\]
By symmetry, this implies that 
\begin{align}
\sum_{i\neq j} \frac{ \partial_{j} \mathrm{T}_k}{\tan\left(\frac{\theta_j-\theta_i}{2}\right)}  
& \notag= - \sqrt{2 k}  \sum_{i\neq j}  \frac{e^{\i k \theta_j}  - e^{\i k \theta_i} }{e^{\i \theta_j}- e^{\i \theta_i}}  e^{\i \theta_j} \\
&\notag =  - \sqrt{2 k}  \sum_{i\neq j}  {\textstyle  \sum_{\ell = 1}^{k}  e^{\i \ell \theta_j} e^{\i (k-\ell) \theta_i}} \\
&  \label{explicit1}
=  - \sqrt{2 k}  \sum_{i ,  j }  {\textstyle  \sum_{\ell = 1}^{k}  e^{\i \ell \theta_j} e^{\i (k-\ell) \theta_i}}  + k^2 \mathrm{T}_k ,
\end{align}
where we have used that $\frac{e^{\i k \theta_j}  - e^{\i k \theta_i} }{e^{\i \theta_j}- e^{\i \theta_i}}  =
\sum_{\ell = 1}^{k} e^{\i (\ell-1) \theta_j} e^{\i (k-\ell) \theta_i} $. 
Note that in the sum on the RHS of \eqref{explicit1}, the term $\ell=k$  equals to $ -nk \mathrm{T}_k$ while the other terms can be expressed in terms of the variables $(\mathrm{T}_\ell)_{\ell=1}^{k-1}$. 
Hence, according to formula \eqref{def:L} and by combining formulae \eqref{explicit2} and \eqref{explicit1}, this shows that  for any $k\ge 1$, 
\begin{equation} \label{explicit3} 
\mathrm{L} \mathrm{T}_k =  nk \mathrm{T}_k + \zeta_k
\end{equation}
Taking real and imaginary parts of the equation \eqref{explicit3}, this  completes the proof.
\end{proof}

Remarkably with Lemma~\ref{lem:eig} and Theorem~\ref{thm:DS}, we can exactly compute the error terms on the RHS of the estimate \eqref{LLW}. We obtain the following results.

\begin{lemma} \label{lem:A}
For any $n,m\in\N$ such that $m\le n$, 
\[
\E_n \big[ |  \mathbf{K}^{-1} \x|^2 \big] =  \frac{(2m+5)m(m-1)}{9 n^2} 
\]
\end{lemma}

\begin{lemma} \label{lem:B}
For any $n,m\in\N$ such that $m \le n/2$, we have
\[
\E_n\big[  \| \mathbf{I} -  \mathbf{K}^{-1} \boldsymbol{\Gamma} \|^2 \big] = 
\frac{(8m+7)(m+1)m}{6 n^2} .
\]
\end{lemma}

Using Lemmas~\ref{lem:A} and~\ref{lem:B}, we can complete the proof of Theorem~\ref{thm:Stein}.  
According to Lemma~\ref{lem:eig}, we have 
\[
\mathbf{K}^{-1} \mathrm{L}\mathbf{X} - \mathbf{X} =    \mathbf{K}^{-1}\x , 
\]
so that for all $m,n \in\N$ such that $m\le n/2$,
\[ \begin{aligned}
\sqrt{\E_n \big[ |  \mathbf{K}^{-1} \mathrm{L}\mathbf{X} - \mathbf{X}|^2 \big]}
+ \sqrt{  \E_n\big[  \| \mathbf{I} -   \mathbf{K}^{-1} \boldsymbol{\Gamma} \|^2 \big] } 
& \le  \frac{\sqrt{(2m+5)m(m-1)} +\sqrt{(8m+7)(m+1)9m/6} }{3n} \\
& \le (\sqrt{8}+\sqrt{2})\frac{ (m+1)\sqrt{m}}{3n} . 
\end{aligned}\]
By Proposition~\ref{prop:LLW}, we obtain the required bound for the Kantorovich distance $\mathrm{W}_2(\X,\G)$ between the random vector $\X$ and a standard Gaussian random variable on $\R^{2m}$. 
Thus, to complete the proof, it remains to prove  Lemmas~\ref{lem:A} and~\ref{lem:B}. 

\begin{proof}[Proof of  Lemma~\ref{lem:A}] 
According to the notation of Lemma~\ref{lem:eig}, we have 
\[  
|  \mathbf{K}^{-1} \x|^2 =  \sum_{k=1}^m  \frac{|\zeta_k|^2}{n^2 k^2} = \sum_{1\le \ell , \ell' < k \le m} \hspace{-.3cm}    \frac{\sqrt{\ell(k-\ell)} \sqrt{\ell'(k-\ell' )}}{2k n^2}  \mathrm{T}_\ell \mathrm{T}_{k-\ell}  \overline{\mathrm{T}_{\ell'} \mathrm{T}_{k-\ell'}} .  
\]
Moreover, according to Theorem~\ref{thm:DS}, if $m\le n$, then it holds  for any integers $1\le \ell , \ell' < k \le m$, 
\[ \begin{aligned}
\E_n\big[ \mathrm{T}_\ell \mathrm{T}_{k-\ell}  \overline{\mathrm{T}_{\ell'} \mathrm{T}_{k-\ell'}} \big] \
& = ( \1_{\{\ell = \ell',\ell\neq k/2\}} + \1_{\{\ell = k- \ell',\ell\neq k/2\}} ) \E_n\big[  |\mathrm{Z}_\ell|^2  |\mathrm{Z}_{k-\ell}|^2 \big]  + \1_{\ell = \ell' = k/2}  \E_n\big[  |\mathrm{Z}_\ell|^4 \big]  \\
& = 4 ( \1_{\{\ell = \ell',\ell\neq k/2\}} + \1_{\{\ell = k- \ell',\ell\neq k/2\}} )  +  8\1_{\ell = \ell' = k/2}  .
\end{aligned}\]
This implies that 
\[  
\E_n \big[ |  \mathbf{K}^{-1} \x|^2 \big]  =  \frac{4}{n^2} \sum_{1\le \ell  < k \le m} \hspace{-.3cm}   \frac{\ell(k-\ell)}{ k} 
= \frac{(2m+5)m(m-1)}{9 n^2} ,
\]
where we have used that 
$\sum_{1\le \ell  < k}    \frac{\ell(k-\ell)}{k}   = \frac{k^2-1}{6}$. 
\end{proof} 

\begin{proof}[Proof of  Lemma~\ref{lem:B}]   
Let us decompose  $ \boldsymbol{\Gamma}   = \widetilde{\boldsymbol{\Gamma}} + \boldsymbol{\Delta}$ where  $\boldsymbol{\Delta} = \operatorname{diag}(\boldsymbol{\Gamma})$. 
The point is that 
\begin{equation} \label{estimate0}
\| \mathbf{I} -  \mathbf{K}^{-1} \boldsymbol{\Gamma} \|^2 =  \| \mathbf{I} -  \mathbf{K}^{-1} \boldsymbol{\Delta} \|^2 +   \| \mathbf{K}^{-1} \widetilde{\boldsymbol{\Gamma}} \|^2. 
\end{equation}
Since $\mathrm{X}_{2k-1} = \sqrt{\frac 2 k} \sum_{j=1}^n \cos( k \theta_j)$ and $\mathrm{X}_{2k} = \sqrt{\frac 2 k} \sum_{j=1}^n \sin( k \theta_j)$, by  \eqref{Gamma2}, we have for any $k=1, \dots, m$, 
\begin{equation} \label{estimate3}
\begin{aligned}
\boldsymbol{\Gamma}_{2k-1, 2k-1 }  &=2 k  {\textstyle \sum_{j=1}^n  \sin^2(k \theta_j)}   = nk -   \frac{k^{3/2}}{\sqrt{2}} \Re(\mathrm{T}_{2k}) \\
\boldsymbol{\Gamma}_{2k, 2k}  &= 2 k  {\textstyle \sum_{j=1}^n  \cos^2(k \theta_j)}   = nk +  \frac{k^{3/2}}{\sqrt{2}} \Re(\mathrm{T}_{2k})  . 
\end{aligned}
\end{equation}
According to the notation of Lemma~\ref{lem:eig}, this shows that 
\begin{equation} \label{estimate1}
\| \mathbf{I} -  \mathbf{K}^{-1} \boldsymbol{\Delta} \|^2   =
\sum_{k=1}^m  \frac{k}{n^2} \Re(\mathrm{T}_{2k})^2 . 
\end{equation}

It remains to compute the second term on the RHS of \eqref{estimate0}. 
Let $\mathbf{K}^{1/2}$ be the positive square--root of the diagonal matrix $\mathbf{K}$ and observe that by definition of the Hilbert Schmidt norm:
\begin{equation} \label{estimate2}
\begin{aligned}
 \|  \mathbf{K}^{-1}\widetilde{\boldsymbol{\Gamma}}  \| =
\|  \mathbf{K}^{-1/2}\widetilde{\boldsymbol{\Gamma}}  \mathbf{K}^{-1/2}\|^2 =  \frac{2}{n^2}  \bigg( \sum_{1\le k < \ell \le m} \frac{\boldsymbol{\Gamma}_{2\ell, 2k}^2 + \boldsymbol{\Gamma}_{2\ell-1, 2k-1}^2}{k\ell}
+   \sum_{1\le k \le \ell \le m} \frac{\boldsymbol{\Gamma}_{2\ell-1, 2k}^2 +     \boldsymbol{\Gamma}_{2\ell, 2k-1}^2}{k\ell} \bigg) . 
\end{aligned} 
\end{equation}
Like \eqref{estimate3}, we can compute the coefficients on the RHS of  \eqref{estimate2}. We check that for any $1\le  k \le \ell \le m$,
\begin{align*}
\boldsymbol{\Gamma}_{2\ell, 2k} 
&= 2\sqrt{k\ell}  {\textstyle \sum_{j=1}^n}  \cos(k \theta_j) \cos(\ell \theta_j)  
=  \sqrt{k\ell(\tfrac{\ell-k}{2})} \Re(\mathrm{T}_{\ell-k})  + \sqrt{k\ell(\tfrac{\ell+k}{2})} \Re(\mathrm{T}_{\ell+k}) ,\\  
\boldsymbol{\Gamma}_{2\ell-1, 2k-1 }  
&= 2 \sqrt{k\ell}  {\textstyle \sum_{j=1}^n}  \sin(k \theta_j) \sin(\ell \theta_j) 
=  \sqrt{k\ell(\tfrac{\ell-k}{2})} \Re(\mathrm{T}_{\ell-k})  - \sqrt{k\ell(\tfrac{\ell+k}{2})} \Re(\mathrm{T}_{\ell+k}) , \\
\boldsymbol{\Gamma}_{2\ell-1, 2k}   
&= - 2\sqrt{k\ell}  {\textstyle \sum_{j=1}^n}  \cos(k \theta_j) \sin(\ell \theta_j) 
=-\sqrt{k\ell(\tfrac{\ell-k}{2})} \Im(\mathrm{T}_{\ell-k})  - \sqrt{k\ell(\tfrac{\ell+k}{2})} \Im(\mathrm{T}_{\ell+k}),  \\
\boldsymbol{\Gamma}_{2\ell, 2k-1 }  
&= - 2 \sqrt{k\ell}  {\textstyle \sum_{j=1}^n}  \sin(k \theta_j) \cos(\ell \theta_j)  
= + \sqrt{k\ell(\tfrac{\ell-k}{2})} \Im(\mathrm{T}_{\ell-k}) - \sqrt{k\ell(\tfrac{\ell+k}{2})} \Im(\mathrm{T}_{\ell+k}).   
\end{align*}
By \eqref{estimate2}, this implies that
\[\begin{aligned}
\|  \mathbf{K}^{-1}\widetilde{\boldsymbol{\Gamma}} \|^2  & = \frac{2}{n^2} \bigg( \sum_{1\le k < \ell \le m} \hspace{-.3cm} \big( (\ell-k)  \Re(\mathrm{T}_{\ell-k})^2 +  (\ell+k)  \Re(\mathrm{T}_{\ell+k})^2 \big)
+ \hspace{-.3cm}\sum_{1\le k \le \ell \le m} \hspace{-.3cm} \big( (\ell-k)  \Im(\mathrm{T}_{\ell-k})^2 +  (\ell+k)  \Im(\mathrm{T}_{\ell+k})^2 \big) \bigg) \\
&= \frac{2}{n^2} \bigg( \sum_{1\le k < \ell \le m} \hspace{-.3cm} \big((\ell-k)  |\mathrm{T}_{\ell-k}|^2 +(\ell+ k) |\mathrm{T}_{\ell+k}|^2 \big) + 2 \sum_{k=1}^m k  \Im(\mathrm{T}_{2k})^2  \bigg) . 
\end{aligned} \]
Combining the previous formula with \eqref{estimate0} and \eqref{estimate1}, we obtain
\[
\E_n\big[  \| \mathbf{I} -  \mathbf{K}^{-1} \boldsymbol{\Gamma} \|^2 \big]  = 
\frac{2}{n^2} \bigg( \sum_{1\le k < \ell \le m} \hspace{-.3cm} (\ell-k)  \E_n\big[  |\mathrm{T}_{\ell-k}|^2  \big]+ (\ell+ k) \E_n\big[  |\mathrm{T}_{\ell+k}|^2 \big] \bigg) + \frac{5}{n^2} \sum_{k=1}^m k   \E_n\big[ \Re(\mathrm{T}_{2k})^2 \big]  ,
\]
where we used that the random variables $ \Re(\mathrm{T}_{k})$ and $ \Im(\mathrm{T}_{k})$ have the same law for all $k\ge 1$. 
Hence,  by Theorem~\ref{thm:DS}, we conclude that if $m \le n/2$, 
\begin{equation*} 
\E_n\big[  \| \mathbf{I} -  \mathbf{K}^{-1} \boldsymbol{\Gamma} \|^2 \big]   
=    \sum_{1\le k < \ell \le m} \frac{4\ell}{n^2} + \sum_{k=1}^m  \frac{5k}{n^2} \
=  \frac{(8m+7)(m+1)m}{6 n^2} .
\end{equation*}
This completes the proof. 
\end{proof}


\appendix

\section{Additional proofs}

\subsection{Proof or Lemma~\ref{lem:Laplace}} \label{sect:proof_Laplace}

Without loss of generality, we assume that $\widehat{f}_0 = 0$, then by \eqref{BO}, we have
\begin{equation} \label{BO1}
\E_n[ \exp \tr f(\u)]  = e^{\A(f)} \det[\operatorname{I} -  K_f  Q_n] ,
\end{equation}
where according to \eqref{BOK},  if we let $w= e^{-2\i \Im (f^+)}$, the kernel  $K_f$ is given by
\[
K_f = H_+(w) H_-(\overline{w}) = H(w) H(w)^*,
\]
where $H(w)^*$ is the adjoint of $H(w)$. Therefore, $K_f>0$ as a trace--class operator and this implies that for  any $n\in\Z_+$, 
\[
0< \det[\operatorname{I} -  K_f  Q_n]  \le 1. 
\] 
Then, the claim follows directly from \eqref{BO1}.


\subsection{Proof of Lemma \ref{lem:GaussianTail}} \label{sect:proof_GT}

Recall that for any $m\in\N$ and $\Lambda>0$, we let 
\[
g_m(\Lambda) =  e^{-\Lambda^2}    \sum_{0\le k < m}   m^k \Lambda^{-2(k+1)} . 
\]
First, by going to polar coordinates and making the change of variable $u = \|\xi\|^2$,  we have for any $\Lambda>0$, 
\[
\int_{\substack{ \R^{2m} \\ \|\xi\| \ge \Lambda} }   e^{-\|\xi\|^2}   \d \xi
=  \Omega_{m} \int_{\Lambda^2}^{+\infty} e^{- u} \d(u^{m})
\]
The integral on the RHS corresponds to the incomplete Gamma function -- see \cite[Formula (8.2.2)]{DLMF} -- and repeated integrations by parts give for any $\lambda>0$,
\[
\int_{\lambda}^{+\infty} e^{- u} \d(u^{m})  = (m-1)! e^{-\lambda}  \sum_{0\le k < m}  \frac{\lambda^{k}}{k!} .
\]
Using the bound $(m-k)! \ge \frac{(m-1)!}{m^{k-1}}$ valid for all $k= 1,\dots, m$, this implies that 
\[
\int_{\substack{ \R^{2m} \\ \|\xi\| \ge \Lambda} }   e^{-\|\xi\|^2}   \d \xi 
\le   \Omega_{m}  e^{-\Lambda^2}  \sum_{1\le k \le m}  m^{k-1}\Lambda^{2k} =  \Omega_{m}  g_m(\Lambda)  . 
\]
Finally, if $\Lambda^2 > m$,  by summing the geometric sum, we obtain  $g_m(\Lambda) \le \frac{e^{-\Lambda^2}}{\Lambda^2-m}  $ . 


\subsection{Proof of Lemma~\ref{lem:unibound}} \label{sect:proof_unibound}
By symmetry, it suffices to prove that for all $y \in (0,1]$ and $x\ge0$,
\[
1+ \bigg(\frac{\sinh(x)}{y} \bigg)^2 \le \exp \Big(\frac x y\Big)^2 . 
\]
We have for any fixed $y \in (0,1]$ and $x\ge 0$, 
\[ \begin{aligned}
\frac{\d}{\d x} \left(\left( 1+ \bigg(\frac{\sinh(x)}{y} \bigg)^2 \right) \exp\bigg( -\frac{x^2}{y^2} \bigg) \right)
& = -\frac{2}{y^2}  \exp\bigg( -\frac{x^2}{y^2} \bigg)  \left( x \left( 1+ \bigg(\frac{\sinh(x)}{y} \bigg)^2 \right) -\sinh(x) \cosh(x) \right)  \\
& \le -\frac{2}{y^2}  \exp\bigg( -\frac{x^2}{y^2} \bigg)  \Big( x \left( 1+ \sinh(x)^2 \right) -\sinh(x) \cosh(x) \Big)\\   
& \le  -\frac{2}{y^2}  \exp\bigg( -\frac{x^2}{y^2} \bigg) \cosh(x) \Big( x \cosh(x) -\sinh(x) \Big)  . 
\end{aligned}\]
Since $ x \cosh(x) -\sinh(x) \ge 0$, this shows that for any fixed $y \in (0,1]$ and $x \ge 0$
\[
\frac{\d}{\d x} \left(\left( 1+ \bigg(\frac{\sinh(x)}{y} \bigg)^2 \right) \exp\bigg( -\frac{x^2}{y^2} \bigg) \right)
\le 0 .
\]
This implies that for any  $y \in (0,1]$
\[
\max_{x>0} \left\{ \left( 1+ \bigg(\frac{\sinh(x)}{y} \bigg)^2 \right) \exp\bigg( -\frac{x^2}{y^2} \bigg) \right\}
= 1 .
\]
Since the RHS is independent of  $y \in (0,1]$, this completes the proof.


\subsection{Proof of Lemma~\ref{lem:TB}} \label{sect:proof_TB}

Let us define the function $\Phi(\boldsymbol \theta) = \sum_{1\le i <j \le n} \log|e^{\i\theta_i} - e^{\i\theta_j} |^{-2} $
for $\boldsymbol\theta \in \triangle$ where $\triangle := \{ \boldsymbol\theta \in \R^n : \theta_1=0< \theta_2 <\cdots < \theta_n < 2\pi\}$ is a convex set. 
Observe that by symmetry, we have 
\[
\max_{\theta_1,\dots,\theta_n\in\T} \bigg( \prod_{1\le i <j \le n} |e^{\i\theta_i} - e^{\i\theta_j} |^2 \bigg)
=  \max_{\boldsymbol \theta \in \triangle} \big( e^{- \Phi(\boldsymbol \theta) } \big) 
= e^{-  \min_{\boldsymbol \theta \in \triangle} \Phi(\boldsymbol \theta) }  . 
\]
Since function $\Phi$ is smooth on $\triangle$, by computing its Hessian (with respect to $\theta_2, \dots, \theta_n$), we verify that $\Phi$  is strictly convex\footnote{This follows from the fact that the Hessian $\nabla^2\Phi$ has a strictly dominant diagonal with positive entries on $\triangle$.}. 
Moreover, if we let $\boldsymbol\vartheta =(0, \frac{2\pi}{n}, \dots , \frac{2\pi(n-1)}{n} )$, we see that by symmetry for any $j=2,\dots, n$, 
\[
\nabla_j\Phi(\boldsymbol\vartheta) = \sum_{i\neq j} \frac{1}{\tan(\frac{\vartheta_i- \vartheta_j}{2})} =
\sum_{i\neq j} \frac{1}{\tan(\pi\frac{i-j}{n})}  =0 . 
\]
This implies that $\boldsymbol\vartheta$ is the only critical point of $\Phi$ inside $\triangle$  and since $\Phi = +\infty$ on $\partial\triangle$, we have
\[
\min_{\boldsymbol \theta \in \triangle} \Phi(\boldsymbol \theta) = \Phi(\boldsymbol \vartheta) . 
\]
Moreover, by definition of the Vandermonde determinant, 
\[
e^{- \Phi(\boldsymbol \vartheta) }  =  \prod_{1\le i <j \le n} |e^{\i\vartheta_i} - e^{\i\vartheta_j} |^2  = \big|\det_{n\times n}(e^{\i(j-1)\vartheta_i}) \big|^2 = \big|\det_{n\times n}(e^{\i 2\pi \frac{(j-1)(i-1)}{n}}) \big|^2 . 
\]
This shows that for any $n\ge 2$,  
\[
\max_{\theta_1,\dots,\theta_n\in\T} \prod_{1\le i <j \le n} |e^{\i\theta_i} - e^{\i\theta_j} |^2  = 
e^{- \Phi(\boldsymbol \vartheta) }  = n^n  \big|\det_{n\times n} A \big|^2, 
\]
where $A_{ij}  = \frac{e^{\i 2\pi \frac{(j-1)(i-1)}{n}}}{\sqrt{n}}$. 
We easily verify that the columns of the matrix $A$ are orthonormal so that $A$ is a unitary matrix and $\big|\det_{n\times n} A \big| =1$. 
This proves that for any integer $n\ge 2$, 
\[
\max_{\theta_1,\dots,\theta_n\in\T} \prod_{1\le i <j \le n} |e^{\i\theta_i} - e^{\i\theta_j} |^2  = n^n . 
\]
We immediately deduce from this fact and formula \eqref{pdf} for the joint density of $\P_n$ that
\[ \begin{aligned}
\E_n \left[ e^{- \sum_{j =1}^n f(\theta_j) } \right] 
& =   \frac{1}{n!} \int_{\T^n}  \prod_{1\le i <j \le n} |e^{\i\theta_i} - e^{\i\theta_j} |^2 e^{- \sum_{j =1}^n f(\theta_j) } \frac{d\theta_1}{2\pi}\cdots \frac{d\theta_n}{2\pi}\\
& \le \frac{e^n}{\sqrt{2\pi n}}  \left( \int_\T  e^{-  f(\theta)} \frac{d\theta}{2\pi} \right)^n ,
\end{aligned}\] 
where we used that $\frac{n^n}{n!} \le \frac{e^n}{\sqrt{2\pi n}} $ by \eqref{Gamma}. 


\section{Constants and numerical approximations} \label{sect:approx}

As we pointed out in the introduction, one of the main challenge of the proof of Theorem~\ref{thm:main} is to try to optimize and keep track of all the constants involved in our different estimates. 
For the convenience of the readers, these constants as well as the error terms in Theorem~\ref{thm:main} are collected in this section. 
The constants are denoted by  $\cst{j} = \cst{j}(m)$, $\eps{j} = \eps{j}(m)$ and $\ups{j} = \ups{j}(m)$ for $j\in\N_0$ since they are allowed to depend on~$m$ but not on the dimension $n$ the random matrix.
They are positive for all $m\ge 3$ and we use the following conventions:
\begin{enumerate}
\item[$\bullet$]  $\cst{j}(m) \to \wcst{j}$ as $m\to+\infty$ where $ \wcst{j} >0$. 
\item[$\bullet$]  $\eps{j}(m) \to 0$ as  $m\to+\infty$. 
\item[$\bullet$]  $\ups{j}(m) \to +\infty$ as  $m\to+\infty$ and $\ups{j}(m)$ is a regularly varying function. 
\end{enumerate}

When relevant, we also provide a numerical approximation or an estimate for these constants. 


\subsection{Errors in Theorem~\ref{thm:main}} \label{sect:constants} 
We let 
\begin{equation} \label{c0}
\cst{0} = \sqrt{\tfrac1{6\sqrt{2}}} \approx 0.343
\end{equation}
and  
\begin{equation} \label{epsilon0}
\eps{0}(m)= \frac{2\cst{0}^2}{3(1+1/m)\sqrt{1+\log m}} .
\end{equation}
We note that $\eps{0}(m) \le 0.041$ for all $m\ge 3$. 
The constants which are directly involved in $\Tf{}$ from  Theorem~\ref{thm:main}  are given by
\begin{equation} \label{c1}
\begin{aligned}
\cst{1}(m) & =  \frac{(1-\cst{10})^2}{16 \cst{11}} - \cst{0}^2(1+\eps{0}) \frac{\sqrt{1+\log m}}{2(m+1)} \\
\cst{2}(m) & = \cst{0} \cst{4} \left(1- \cst{10} - \frac{4\cst{4} \cst{0} \cst{11} (1+\log m)^{1/4}}{\sqrt{m+1}} \right) -\cst{0}^2\frac{(1+ \eps{0}) (1+\log m)^{1/4}}{2\sqrt{m+1}}\\
\cst{3} & =\cst{8}/(2 \pi)^{1/4} \approx 7.98 \\
\cst{4} &=\tfrac{1}{2\sqrt{2}}\approx 0.354  \\
\cst{5} & = \cst{3}^{-1} e^{\cst{9}} \approx 1.147   \\
\cst{6} &= 4(2+ \log 2) \approx 10.78\\
\cst{7} &= \frac{\pi \sqrt{e}}{2} \approx 2.59 \\
\cst{8} &=  \frac{4 \cdot 2.766}{(1-1/16)^2} \approx 12.63 \\
\cst{9} & =   \frac{538}{243} \approx  2.21 \\ 
\cst{10}(m)  &=   \cst{0}^2 \cst{19}(m) \sqrt{1+1/m}  =    \cst{0}^2 \tfrac{\sqrt{1+1/m}}{3\sqrt{6}} e^{\cst{0} \frac{(1+\log m)^{1/4}}{\sqrt{(m+1)/2}}} ,
\qquad\qquad   \wcst{10} = \tfrac{\cst{0}^2}{3\sqrt{6}} \approx 0.016  \\
\cst{11}(m) &=\left( 1 + \frac{\cst{10}}{\sqrt{m}} \right) \left(1+ \frac{2\cst{10}}{m} + \sqrt{\frac{1+1/m}{2(1+\log m)}} \right)  .
\end{aligned}
\end{equation}
We verify that both $\cst{1}(m)$ and $\cst{2}(m)$   are increasing for $m\ge 3$ and we have 
\[
\wcst{1} =  \big(1- \wcst{10}\big)^2/16 \approx  0.0605 ,
\qquad 
\wcst{2} =  4 \cst{0}\cst{4} \sqrt{\wcst{1}}  \approx 0.119. 
\]
Note also that the convergence is slow since $\cst{j} = \wcst{j} + \O\left(\frac{(1+\log m)^{1/4}}{\sqrt{(m+1)}}\right)$ for $j=1,2$ as $m\to+\infty$.

\medskip

Let us also define for all $m\ge 1$,
\begin{equation} \label{ups}
\begin{aligned}
\ups{1}(m)  & = \cst{6}m^2 + \tfrac52 m \log m+  \log(\cst{7}) m  + \tfrac 32 \\
\ups{2}(m)  & =  \tfrac 12 m\log m - \tfrac 34 m \log(1+\log m) - m \log(8\cst{0}) +  \tfrac 12 . 
\end{aligned}
\end{equation}

It will turn out that we need the following estimates for the functions $\ups{1}(m)$ and $\ups{2}(m)$. 
We have for all $m\ge 3$,
\begin{equation} \label{Upsilon1}
\sqrt{ \cst{1}(m)^{-1}(1+\log m)\ups{1}(m)} \ge  34 m  
\end{equation}
and 
\begin{equation} \label{Upsilon2}
\frac{\cst{1}(m)\ups{2}(m) \sqrt{m+1}}{\cst{2}(m)(1+\log m)^{1/4}}  \le \frac{\ups{1}(m)}{ 1500 } . 
\end{equation}
The numerical constants in \eqref{Upsilon1} and \eqref{Upsilon2} are not optimal but they suffice for our applications.

\medskip

For any $N,m \ge 1$, let us define the following functions:
\begin{equation} \label{Theta0} 
\Tf{0}   =  m^{\frac 52}  2^{\frac{m}{2}}e^{\frac{m^2}{4N}} \frac{ e^{\frac{N}{2}} (1+\log m)^{N}  }{\sqrt{N}\ \Gamma(N+1)}  , 
\qquad
\Tf{3}  =  \frac{\cst{3}^{-1}}{\sqrt{m}} N^{-\frac{m}{2}}  \exp\left(- \frac{N^2}{16(1+\log m)}\right)  , 
\end{equation}
\begin{equation} \label{Theta2}
\Tf{2}  = \cst{5} N^{\frac{m}{2}}    \exp\bigg( \Upsilon_2(m)  -   \frac{\cst{2}(m) N^2}{\sqrt{m+1} (1+\log m)^{\frac 34}}  \bigg)  ,
\end{equation}
and if $N>4m$, 
\begin{equation} \label{Theta1} 
\begin{aligned}
\Tf{1}  =   \frac{  \cst{5} e^{- \cst{9}\frac{4m}{N}}}{(1-\frac{4m}{N})^{1-\frac{2m}{N}}} 
\exp\bigg(\Upsilon_1(m) -\cst{1}(m) \frac{N (N-4m) }{(1+\log m)}  \bigg) . 
\end{aligned}
\end{equation}

Then, the error in Theorem~\ref{thm:main} is given by 
\begin{equation} \label{Theta}
\Tf{} = \Tf{0}  +  \Tf{1}  +  \Tf{2}  +\Tf{3}   . 
\end{equation} 
One should keep in mind that $ \Tf{0}$ is the main term, the term $\Tf{3}$ is  always negligible, while $\Tf{1}$ and $\Tf{2}$ are corrections which become negligible when $m \ll N$. 
This is quantified by Proposition~\ref{prop:Theta} below.

\medskip

In Sections~\ref{sect:proofs}--\ref{sect:QF}, the following constants will come in play. 
\begin{equation} \label{c2}
\begin{aligned}
\cst{12}  & =  \frac{1+\sqrt{290}}{17} \approx 1.06  \\
\cst{13} &= \frac{(\log 108) \sqrt{1+\log 3}}{68 \sqrt{108} \cst{2}(m) } \\
\cst{14}  & = \frac{\sqrt{13}}{1500} \approx 0.0024   \\
\cst{15} &=  2 e^2 \approx 14.78 \\
\cst{16} & = 2e \sqrt{\pi} \approx 9.64 \\
\cst{17} & = \frac{32}{3}(1+ \frac{(2-1/m)^3}{3(m+1)}) \\
\cst{18} & = \frac{8}{3}(1+\frac{4(1-1/m)^3}{3(m+1)}) 
\end{aligned}
\end{equation}
\begin{equation} \label{c3}
\begin{aligned}
\cst{19}(m)  &= \frac{1}{3\sqrt{6}}  e^{\sqrt{2} \cst{0} \frac{(1+\log m)^{1/4}}{\sqrt{m+1}}} \\
\cst{20}(\eta) &= \frac{\pi^2 \eta^2}{8} ,  \qquad \eta =\frac{1/\pi}{\sqrt m}  \\
\cst{21}(m,\eta) &= \frac{\exp\big(\frac{\eta}{2\sqrt{m(m+1)}}\big)}{6\sqrt3} ,  \qquad \eta =\frac{1/\pi}{\sqrt m}   \\
\ups{3}(m) & =  \frac{\pi  m^{3/2}(m+1)e^{\frac 12 \big(1+\frac{1/2}{(\pi m)^2}\big)} }{2 \Big(1  - \frac{\cst{21}}{4 \pi^2m^3}\Big)} . 
\end{aligned}
\end{equation}
Observe that $\ups{3}(m) =   \cst{7} m^{\frac 52}\big(1+ \O(m^{-1})\big) $ as $m\to+\infty$.  Moreover, as $\cst{21} \le \frac{1}{12}$,  we verify that for all $m\ge 3$, 
\begin{equation} \label{ups3}
\ups{3}^m \le  \cst{7}^m m^{\frac{5m}{2}} \frac{e^{\frac{1}{4\pi^2 m}}(1+1/m)^m }{\Big(1  - \frac{m^{-3}}{48 \pi^2}   \Big)^m } \le  e\, \cst{7}^m m^{\frac{5m}{2}} . 
\end{equation}


\subsection{Estimates for errors -- Proof of Proposition~\ref{prop:Thetabound}. } \label{sect:error}

\begin{proposition} \label{prop:Theta}
Fix  $\gamma> \cst{12} =  \frac{1+\sqrt{290}}{17} $. 
For all $N,m \ge 3$ such that $N \ge  \gamma  \sqrt{ \cst{1}(m)^{-1}(1+\log m)\Upsilon_1(m)}$, 
we have the estimates
\begin{equation} \label{Theta3} 
\Tf{1} \le \cst{5} \exp\bigg(- \big( 1 - \tfrac{2\gamma^{-1}}{17}  - \gamma^{-2} \big)   \frac{ \cst{1}(m) N^2 }{1+\log m}  \bigg) 
\end{equation} 
and 
\begin{equation} \label{Theta4} 
\Tf{2} 
\le \cst{5} N^{\frac{m}{2}}    \exp\bigg( -    \frac{ \big( 1 - \tfrac{\gamma^{-2}}{1500}\big)\cst{2}(m) N^2}{\sqrt{m+1} (1+\log m)^{\frac 34}}  \bigg)
\le  \cst{5} \exp\bigg( -   \big( \sqrt{13\gamma} - \cst{13} \gamma^{-1} -\cst{14} \gamma^{-\frac 32}\big)  \frac{\cst{2}(m) N^{\frac 32}}{\sqrt{1+\log m}}  \bigg)  . 
\end{equation}
Moreover, we have the lower--bounds: $\cst{1}(m) \ge 0.0148$, $\cst{2}(m) \ge0.077$ and $\cst{13} \le 0.125$ for all $m\ge 3$.  
\end{proposition}

\begin{proof}
Since $e^\cst{9} \ge 4$, we verify that the function $x\mapsto e^{-\cst{9}x}(1-x)^{1-2x} $ is decreasing on $[0,\tfrac 12]$  so that we deduce from \eqref{Theta1} that for all $N \ge 8m$,
\[
\Tf{1} \le \cst{5}  \exp\bigg(\Upsilon_1(m) -\cst{1}(m) \frac{N (N-4m) }{(1+\log m)}  \bigg) . 
\]
Then, we verify that if the condition $N \ge  \gamma  \sqrt{ \cst{1}(m)^{-1}(1+\log m)\Upsilon_1(m)}$ holds with $\gamma>0$,  
\begin{equation} \label{quadraticbound}
\cst{1}(m) \frac{N (N-4m) }{(1+\log m)}  - \Upsilon_1(m)  \ge  \frac{\cst{1}(m) N^2}{(1+\log m)}  \bigg( 1 - \frac{2\gamma^{-1}}{17}  - \gamma^{-2} \bigg) ,
\end{equation}
where we used the lower--bound \eqref{Upsilon1}. 
The RHS of \eqref{quadraticbound} is positive so long as $\gamma > \cst{12} =  \frac{1+\sqrt{290}}{17} $ and this yields the estimate \eqref{Theta3}.
For the estimate \eqref{Theta4}, let us also observe that according to \eqref{Upsilon2}, we have for all $N, m\ge 3$ such that $N \ge   \gamma\sqrt{ \cst{1}(m)^{-1}(1+\log m)\Upsilon_1(m)}$,
\[
\cst{2}(m)^{-1}\Upsilon_2(m) {\sqrt{m+1} (1+\log m)^{3/4}} \le  \frac{\gamma^{-2}}{1500} N^2 . 
\]
This implies that
\begin{equation*}
\Tf{2} \le \cst{5} N^{\frac{m}{2}}    \exp\bigg( -   \big( 1 - \tfrac{\gamma^{-2}}{1500}\big)  \frac{\cst{2}(m) N^2}{\sqrt{m+1} (1+\log m)^{\frac 34}}  \bigg)  . 
\end{equation*}
Moreover, we also verify  that for all $m\ge 3$, 
\[
(m+1)^2 \le \frac{ \cst{1}(m)^{-1} \Upsilon_1(m) }{13^2}    
\]
so that under our hypothesis,
\[
\frac{ N^2}{\sqrt{m+1} (1+\log m)^{\frac 34}}  \ge \frac{\sqrt{13}\  N^2}{\sqrt{1+\log m}\big( \cst{1}(m)^{-1}(1+\log m)\Upsilon_1(m)\big)^{\frac 14}}  \ge   \frac{\sqrt{13\gamma}\ N^{3/2}}{\sqrt{1+\log m}} . 
\]
Hence, if we agree to loose the Gaussian decay in $N$ of $\Tf{2}$, we obtain that for all $N,m\ge 3$  such that $N \ge  \gamma  \sqrt{ \cst{1}(m)^{-1}(1+\log m)\Upsilon_1(m)}$,
\begin{align*} 
\Tf{2}   \le  \cst{5} N^{\frac{m}{2}}    \exp\bigg( -   \big( \sqrt{13\gamma} - \tfrac{\sqrt{13}}{1500} \gamma^{-\frac 32}\big)  \frac{\cst{2}(m) N^{\frac 32}}{\sqrt{1+\log m}}  \bigg)  . 
\end{align*}
Finally, it follows from \eqref{Upsilon1} that under our hypothesis, $N \ge 34 \gamma m$ with $\gamma>\cst{12}$ so that $N\ge 108$ and 
\[
N^{\frac{m}{2}}    \le \exp\bigg(\frac{ N \log N }{68 \gamma}   \bigg) \le \exp\bigg(\frac{ N^{\frac 32}}{\sqrt{1+\log m}} \frac{\log N  \sqrt{1+\log(N/36)} }{68 \gamma \sqrt{N}}   \bigg) 
\le \exp\bigg(\frac{\cst{13}\cst{2}(m)  N^{\frac 32}}{\gamma \sqrt{1+\log m}} \bigg) , 
\]
where $\cst{13}= \frac{(\log 108) \sqrt{1+\log 3}}{68 \sqrt{108} \cst{2}(m) }$. 
This yields the estimate \eqref{Theta4}. 
Since $\cst{1},\cst{2}$ are increasing functions for $m\ge 3$, we obtain the numerical estimates for $\cst{1},\cst{2}$ and  $\cst{13}$ by evaluating these functions for $m=3$ on \emph{Mathematica}. 
This completes the proof. 
\end{proof}

We will also need the following basic estimates for the main error term $ \Tf{0}$. 

\begin{lemma} \label{lem:Theta0}
For all $m,N \in\N$ such that $m\ge 3$ and $N \ge 5m$, it holds 
\[
\Tf{0}   \le      \frac{1}{\sqrt\pi} \exp\bigg(- N \log m \bigg( 1 - \frac{\log(1+\log m)}{\log m} \bigg) \bigg) . 
\]
\end{lemma}

\begin{proof}
Let us recall from \cite[Formula (5.6.1)]{DLMF} that for any $x>0$, 
\begin{equation} \label{Gamma}
\Gamma(x+1) = \sqrt{2\pi x} x^x \exp\left(-x+ \frac{\theta_x}{12x} \right) 
\qquad\text{where}\quad  \theta_x\in(0,1) .
\end{equation}
In addition, since $e^{\frac 32} \le 5$, let us observe that we have for all $N \ge 5m$,
\[
\frac{m^{\frac 32}  2^{\frac{m}{2}}e^{\frac{m^2}{4N}} e^{\frac{3N}{2}}}{\sqrt{2\pi} 5^{N+1}}
\le  \frac{m^{\frac 32}  e^{-\cst{} m}}{5\sqrt{2\pi}} \le \frac{1}{\sqrt\pi}, 
\]
where we used that  $\cst{} = 5 \log 5 - \frac{15}{2} - \frac{1}{20} - \frac{\log 2}{2}  \ge 0.15$. 
By \eqref{Gamma}, this implies that for all $N \ge 5m$,
\[
\Tf{0}   \le  \frac{m^{\frac 32}  2^{\frac{m}{2}}e^{\frac{m^2}{4N}} e^{\frac{3N}{2}}}{\sqrt{2\pi}5^{N+1}}
\frac{(1+\log m)^{N}  }{m^N}   \le   \frac{1}{\sqrt\pi} \exp\bigg(- N \log m \bigg( 1 - \frac{\log(1+\log m)}{\log m} \bigg) \bigg) . 
\]
\end{proof}

Using the previous estimates, we are now ready to prove Proposition~\ref{prop:Thetabound}.

\begin{proposition} \label{cor:Theta}
Fix $M\ge 3$.
For all $m\ge M$ and $N\ge \cst{}(M) m \sqrt{1+\log m}$, we have 
\begin{equation}  \label{Theta:UB}
\Tf{1} + \Tf{2} + \Tf{3}  \le  0.011\frac{(1+\log m)^N e^{\frac{N}{2}}}{\sqrt{N}\Gamma(N+1)}  \le \eps{}  \Tf{0}
\end{equation}
where $\eps{} \le 25\cdot 10^{-5}$ and the constant $\cst{}(M)$ are explicitly given by the Table \eqref{table:cM} below. 
Moreover, under the same conditions, we also have  $\Tf{0} \le N^{-\frac m2}   \frac{\exp\big(- 12 m \big( \log m  - 0.26\big) \big)}{\sqrt\pi \cst{16}^m} $. 
\end{proposition} 

\begin{proof}
First, observe that for all $N \ge   \gamma\sqrt{ \cst{1}(m)^{-1}(1+\log m)\Upsilon_1(m)}$, if  $\theta(m) : = \sqrt{ \frac{\cst{1}(m)\Upsilon_1(m)}{1+\log m}}  \ge   ( \gamma - \tfrac{2}{17}  - \gamma^{-1})^{-1}$, then it holds that  
\[
\frac{e^{-3/2} \cdot N }{1+\log m}  \exp\bigg(- \frac{ ( 1 - \tfrac{2\gamma^{-1}}{17}  - \gamma^{-2})\cst{1} N }{1+\log m}  \bigg) 
\le \cst{1}^{-1} \gamma  \theta e^{-3/2 - (\gamma  - 2/17 -\gamma^{-1}) \theta  } . 
\]
Let us suppose that $\gamma \le 5.12$. 
This shows that if we choose $\gamma$ depending on $m\ge 3$ in such a way that
\begin{equation} \label{eq:gamma}
(\gamma  - 2/17 -\gamma^{-1}) \theta   \ge \log(5.12\cst{1}^{-1}\theta) -1.48  >0 , 
\end{equation}
then we have 
\[
\frac{e^{-3/2} \cdot N }{1+\log m}  \exp\bigg(- \frac{ ( 1 - \tfrac{2\gamma^{-1}}{17}  - \gamma^{-2})\cst{1} N }{1+\log m}  \bigg)  
\le e^{-0.02} . 
\]
Any solution of \eqref{eq:gamma} satisfies $\gamma> \cst{12}$ and we can choose a (numerical) solution $\gamma(m)$ which is non-increasing in the following way:
\begin{equation}  \label{tablegamma}
\begin{array}{c|c|c|c|c|c|c|c|c|c|c|c|}
m & 3 &4&5 & 6 & 8 &12 &17 &23 &30 & 40 & \ge 70  \\
\hline
\gamma(m) & 5.119 & 3.806& 3.149  & 2.754 & 2.30 &1.882 & 1.65 &1.507& 1.413 & 1.334 & 1.230
\end{array} \, .
\end{equation}
Observe that the function $ m^{-1}\sqrt{ \cst{1}(m)^{-1}\Upsilon_1(m)}$ is also decreasing for $m\ge 3$.
By \eqref{Theta3}, this implies that for any $M\ge 3$, if $m\ge M$ and $N \ge \cst{}(M) m \sqrt{1+\log m}$, then we obtain 
\begin{equation} \label{Theta5}
\begin{aligned}
\Tf{1} & \le \cst{5}  \exp\bigg(- \frac{ ( 1 - \tfrac{2\gamma^{-1}}{17}  - \gamma^{-2})\cst{1} N^2 }{1+\log m}  \bigg) \\
& \le \cst{5} e^{-0.02 N }  \bigg(   \frac{e^{-3/2} \cdot N }{1+\log m}  \bigg)^{-N}
\end{aligned}
\end{equation}
where 
$\cst{}(M)=  \gamma(M) M^{-1} \sqrt{ \cst{1}(M)^{-1}\Upsilon_1(M)} $ is a descreasing function
\begin{equation}  \label{table:cM}
\begin{array}{c|c|c|c|c|c|c|c|c|c|c|c|}
M & 3 &4&5 & 6 & 8 &12 &17 &23 &30 & 40 & 70  \\
\hline
\cst{}(M) & 146.5 & 93.8 & 71.1  & 58.66 & 45.5 & 34.5 & 28.8 & 25.5 & 23.4 & 21.64 & 19.4 \\
\end{array}
\end{equation} 

Moreover, since $N \ge 600$ in the regime that we consider, we also verify that 
\[
N  e^{-0.02 N }\le 600 e^{-12} \le 3.7\cdot 10^{-3} , 
\]
so that by \eqref{Theta5},  this implies  that  for  $m\ge M$ and $N \ge \cst{}(M) m \sqrt{1+\log m}$, 
\begin{equation} \label{Tf1est}
\Tf{1}  \le  0.0107 \frac{(1+\log m)^N e^{\frac{N}{2}}}{\sqrt{N} \Gamma(N+1)}
\end{equation}
where we used that according to \eqref{Gamma}, $\Gamma(N+1) \le  2.52 N^{N+1/2} e^{-N} $. 
By a similar argument, we have  $\Tf{3} \le 0.073 \exp\left(- \frac{0.0625 \cdot N^2}{1+\log m}\right)$ and 
$\displaystyle\tfrac{ N }{1+\log m} \ge \min_{M\ge 3} \tfrac{\cst{}(M)M}{\sqrt{1+\log M}} \ge 200$, so that 
\[
\frac{e^{-3/2} \cdot N }{1+\log m} \exp\left(- \frac{0.0625 \cdot N}{1+\log m}\right)
\le 200 e^{-11 }  \le 3.4\cdot 10^{-3}
\]
and 
\begin{equation} \label{Tf3est}
\Tf{3} \le 0.2 N (3.4\cdot10^{-3})^N   \frac{(1+\log m)^N e^{\frac{N}{2}}}{\sqrt{N} \Gamma(N+1)} .
\end{equation}
Using the estimate \eqref{Theta4}  and the fact that according to the Table \eqref{tablegamma}  $\displaystyle  \min_{m\ge 3} \big\{ ( \sqrt{13\gamma(m)} - \cst{13}/ \gamma(m) -\cst{14}/ \gamma(m)^{\frac 32}) \cst{2}(m) \big\} \ge 0.422  $, we obtain the estimate
\[
\Tf{2} \le  \cst{5}   \exp\bigg( -  \frac{ 0.422 \cdot N^{\frac 32}}{\sqrt{1+\log m}}  \bigg) . 
\]
Since we have seen that  $\frac{N}{1+\log m} \ge 200 $, this implies that 
\[
\frac{ N }{1+\log m} \exp\bigg( -  \frac{ 0.422  \sqrt{N}}{\sqrt{1+\log m}}  \bigg)
\le \max_{x\ge \sqrt{200}} \big\{ x^2 e^{-0.422 x} \big\}   \le 1 . 
\]
So, using the same argument once more, we obtain that  for any  $m\ge M$ and $N \ge \cst{}(M) m \sqrt{1+\log m}$, 
\begin{equation} \label{Tf2est}
\Tf{2} \le  3 N e^{-3N/2}
\frac{(1+\log m)^N e^{\frac{N}{2}}}{\sqrt{N} \Gamma(N+1)} . 
\end{equation}
By combining the estimates \eqref{Tf1est}, \eqref{Tf3est} and \eqref{Tf2est}, we easily verify that for any  $m\ge M$ and for all $N \ge \cst{}(M) m \sqrt{1+\log m}$, 
\[
\Tf{1} + \Tf{2} + \Tf{3}  \le  0.011\frac{(1+\log m)^N e^{\frac{N}{2}}}{\sqrt{N}\Gamma(N+1)} . 
\]
Then, from \eqref{Theta0}, we deduce the bound \eqref{Theta:UB} with $\eps{} \le 0.011 \cdot  3^{-\frac 52}  2^{-\frac{3}{2}}  \le 25 \cdot 10^{-5}$. 
Finally, it remains to obtain the upper--bound for $\Tf{0}$. 
According to Lemma~\ref{lem:Theta0}, we have for all $m\ge 3$ and $N\ge 5m$,
\[
N^{\frac m2} \Tf{0}  \le   \frac{1}{\sqrt\pi} \exp\big(- 0.3 N \log m  + 0.5 m \log N  \big) . 
\]
Since the function $N\mapsto  0.3 N \log m  - 0.5 m \log N $ is increasing and $ \displaystyle \min_{M\ge 3}  \cst{}(M)  \sqrt{1+\log M} \ge 42$, this implies that for $N\ge 42 m$, 
\[
N^{\frac m2} \Tf{0}  \le   \frac{\exp\big(- 12 m \log m  + \tfrac{\log 42}{2}m  \big)}{\sqrt\pi} 
\le  \frac{\exp\big(- 12 m \big( \log m  - 0.26 \big) \big)}{\sqrt\pi \cst{16}^m}  ,
\]
where we used that $\tfrac{\log(42 \cst{16})}{24} \le 0.26$. 
\end{proof}

\subsection{Numerics for $m=3$} \label{sect:plot}

\vspace*{-.5cm}

\begin{figure}[H]
\caption{\small 
Log-Log plot of the errors \eqref{Theta0}--\eqref{Theta2} for $m=3$ as functions of $N=n/3$ where $n$ is the dimension of the random unitary matrix. We observe that $\Theta^0_{N,3} \ge  \Theta^1_{N,3}$ when $N \ge 631$ which is consistent with the threshold $3\cst{}(3) \sqrt{1+\log 3} \approx 637$ from Proposition~\ref{cor:Theta}.  
}
\centering
\vspace*{.2cm}
\includegraphics[width=.6\textwidth]{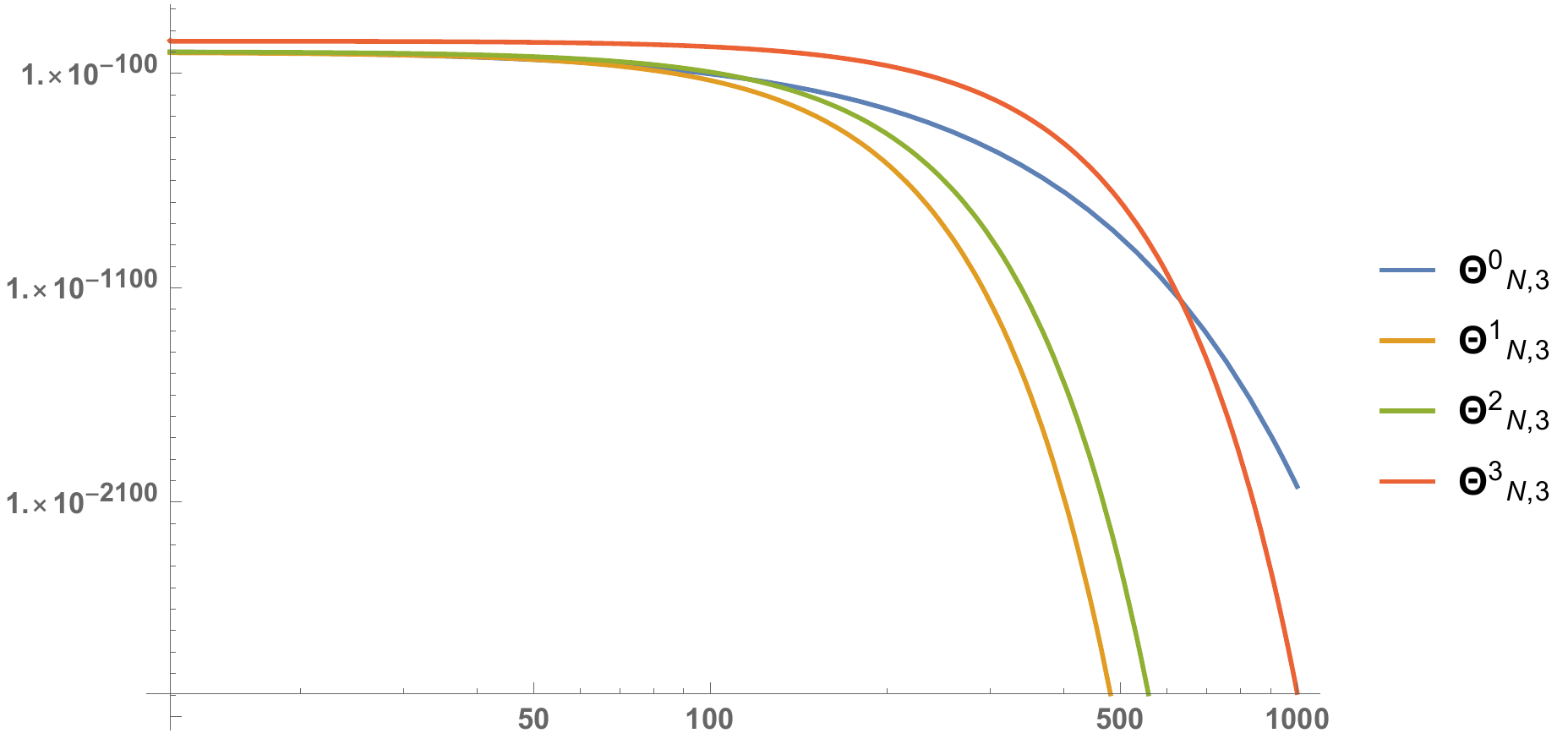}
\end{figure}

\clearpage

\begin{figure}[H]
\caption{\small 
Plot of $\log(\Delta_{n,3}^{(2)})$ as a function of the dimension $n$ of the random unitary matrix. 
By Theorem~\ref{thm:main}, this quantity controls the total variation distance between $\X$ and a standard Gaussian vector in $\R^6$. We observe that our estimates become relevant as soon as $n\ge 400$ which can still be considered a small size random matrix.  
}
\centering
\includegraphics[width=.45\textwidth]{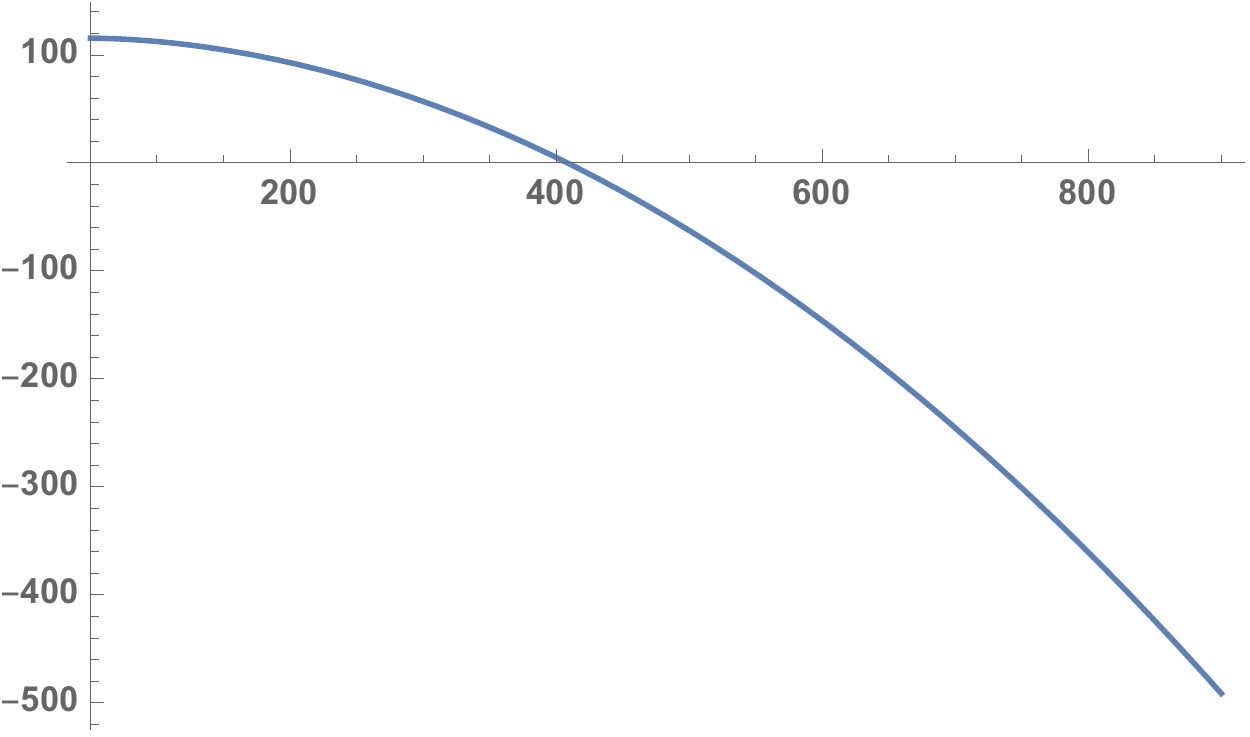}
\end{figure}


\end{document}